\documentclass[12pt]{amsart}

\usepackage[margin=1.3in]{geometry}
\usepackage[utf8]{inputenc}
\usepackage[T1]{fontenc}
\usepackage{subfiles}
\usepackage{comment}
\usepackage{float}
\usepackage{marginnote}
\usepackage{euscript}
\usepackage[dvipsnames]{xcolor}
\usepackage{graphicx}
\usepackage{amssymb}
\usepackage{mathrsfs}
\usepackage{amsmath}
\usepackage{mathtools} 

\usepackage[cal=cm]{mathalfa}
\usepackage{stmaryrd}
\usepackage{float}
\usepackage{todonotes} \setuptodonotes{color=blue!30}
\usepackage{changebar}
\usepackage[full]{textcomp}
\usepackage{booktabs}
\usepackage{subcaption}

\usepackage[colorlinks=true, allcolors=highlight, hyperfootnotes=false, linktocpage=true]{hyperref}
\usepackage{amsthm}
\usepackage{aliascnt} 
\usepackage{url}

\definecolor{highlight}{cmyk}{90,99,0,0}

\usepackage{fancyvrb}
\usepackage{fvextra}
\usepackage{inconsolata}
\DefineVerbatimEnvironment{MyVerb}{Verbatim}{commandchars=\\\{\}, fontfamily=zi4, fontsize=\small}
\fvset{baselinestretch=1.2}
\fvset{breaklines=true, breakanywhere=true}
\usepackage{framed}
\definecolor{shadecolor}{rgb}{0.95,0.95,0.95}
\definecolor{code}{cmyk}{20, 0, 100, 0}

\usepackage{tikz,tikz-cd,tikz-3dplot}
\usepackage{pgfplots}
\usetikzlibrary{calc}
\usetikzlibrary{fadings}
\usetikzlibrary{decorations.pathmorphing}
\usetikzlibrary{decorations.pathreplacing}
\usetikzlibrary{patterns}
\usetikzlibrary{arrows,shadows,positioning, calc, decorations.markings,
	hobby,quotes,angles,decorations.pathreplacing,intersections,shapes}
\usepgflibrary{shapes.geometric}
\usetikzlibrary{fillbetween,backgrounds}
\usetikzlibrary{decorations.text}

\usepackage[skip = .5\baselineskip]{parskip}
\usepackage{anyfontsize}
\usepackage[lining]{libertine}%
\usepackage{courier}
\usepackage{csquotes}
\makeatletter
\renewcommand*\libertine@figurestyle{LF}
\makeatother
\usepackage[libertine,libaltvw,liby]{newtxmath}
\usepackage{microtype}

\usepackage{enumerate}
\usepackage{enumitem}
\setlist[enumerate,1]{label=(\roman*),itemsep=0.9ex}
\setlist[itemize]{itemsep=0.9ex}

\hypersetup{bookmarksdepth = 3}

\makeatletter
\def\@tocline#1#2#3#4#5#6#7{\relax
	\ifnum #1>\c@tocdepth 
	\else
	\par \addpenalty\@secpenalty\addvspace{#2}%
	\begingroup \hyphenpenalty\@M
	\@ifempty{#4}{%
		\@tempdima\csname r@tocindent\number#1\endcsname\relax
	}{%
		\@tempdima#4\relax
	}%
	\parindent\z@ \leftskip#3\relax \advance\leftskip\@tempdima\relax
	\rightskip\@pnumwidth plus4em \parfillskip-\@pnumwidth
	#5\leavevmode\hskip-\@tempdima
	\ifcase #1
	\or\or \hskip 1em \or \hskip 2em \else \hskip 3em \fi%
	#6\nobreak\relax
	\dotfill\hbox to\@pnumwidth{\@tocpagenum{#7}}\par
	\nobreak
	\endgroup
	\fi}
\makeatother

\newcommand{\mr}{\mathrm}

\DeclareMathOperator{\Aut}{Aut}

\newcommand{\Tot}{\mathrm{Tot}}
\DeclareMathOperator{\Rep}{Rep}
\DeclareMathOperator{\Exp}{Exp}
\DeclareMathOperator{\SL}{SL}

\newcommand{\memind}{membrane index}
\newcommand{\meminds}{membrane indices}
\newcommand{\GWPT}{K-theoretic GW/PT correspondence}
\newcommand{\CYth}{X}
\newcommand{\CYfo}{Y}
\newcommand{\CYfi}{Z}

\newcommand{\define}[1]{\textbf{#1}}
\newcommand{\ol}{\overline}
\newcommand{\cev}[1]{\reflectbox{\ensuremath{\vec{\reflectbox{\ensuremath{#1}}}}}}

\newcommand{\gpT}{\mathsf{T}}
\newcommand{\LieT}{\mathfrak{t}}
\newcommand{\invariantfont}[1]{\mathsf{#1}}
\newcommand{\modulifont}[1]{#1}
\newcommand{\Mbar}{\overline{\modulifont{M}}}
\newcommand{\vir}{\mathrm{virt}}

\newcommand{\GW}{\invariantfont{GW}}
\newcommand{\GWs}{GW series}
\newcommand{\GWarg}[3]{\GW_{#3}(#1,#2)}

\newcommand{\MTwoMod}{\invariantfont{M2}}

\newcommand{\PT}{\invariantfont{PT}}
\newcommand{\PTs}{PT series}
\newcommand{\NO}{\invariantfont{NO}}
\newcommand{\PTarg}[5]{\PT_{#1}(#3,#4,#5)}

\newcommand{\ZNek}{\mathcal{Z}^{\mr{Nek}}}

\newcommand{\qnum}[2]{[#1]_{#2}}
\newcommand{\sqabs}[1]{\mathbin{\|} #1 \mathbin{\|}^2}

\newcommand{\Aone}{\mathbb{C}}
\newcommand{\Aaff}[1]{\mathbb{C}^{#1}}
\usepackage{xifthen}
\newcommand{\Gm}[1][]{%
	\ifthenelse{\isempty{#1}}%
	{\mathbb{C}^\times}
	{(\mathbb{C}^\times)^{#1}}
}

\newcommand{\intEquiv}[1]{
	\int}

\newcommand{\hhh}{\mathsf{H}}
\newcommand{\Heff}[1]{\hhh_2^{\mathrm{eff}}(#1,\mathbb{Z})}
\newcommand{\HodgeLambda}[2][]{\Lambda_{#1} (#2)}%

\makeatletter 
\def\makeCal#1{%
	\expandafter\newcommand\csname c#1\endcsname{\mathcal{#1}}}

\def\makeBB#1{%
	\expandafter\newcommand\csname b#1\endcsname{\mathbb{#1}}}

\def\makeFrak#1{%
	\expandafter\newcommand\csname f#1\endcsname{\mathfrak{#1}}}

\def\makeScr#1{%
	\expandafter\newcommand\csname s#1\endcsname{\mathscr{#1}}}

\def\makeSF#1{%
	\expandafter\newcommand\csname sf#1\endcsname{\mathsf{#1}}}

\count@=0
\loop
\advance\count@ 1
\edef\y{\@Alph\count@} 
\expandafter\makeCal\y
\expandafter\makeBB\y
\expandafter\makeFrak\y
\expandafter\makeScr\y
\expandafter\makeSF\y
\ifnum\count@<26
\repeat

\newcommand{\ri}{\mathrm{i}}
\newcommand{\re}{\mathrm{e}}

\theoremstyle{plain}
\newtheorem{thm}{Theorem}[section]

\newtheorem*{prop*}{Proposition}

\newtheorem*{conj*}{Conjecture}

\newtheorem*{cor*}{Corollary}

    \newaliascnt{prop}{thm}
    \newtheorem{prop}[prop]{Proposition}
    \aliascntresetthe{prop}
    
    \newaliascnt{conj}{thm}
    \newtheorem{conj}[conj]{Conjecture}
    \aliascntresetthe{conj}

\theoremstyle{definition}

    \newaliascnt{rmk}{thm}
    \newtheorem{rmk}[rmk]{Remark}
    \aliascntresetthe{rmk}
    
    \newaliascnt{example}{thm}
    \newtheorem{example}[example]{Example}
    \aliascntresetthe{example}
    
    \newaliascnt{nonexample}{thm}
    \newtheorem{nonexample}[nonexample]{Non-example}
    \aliascntresetthe{nonexample}
    
    \newaliascnt{evidence}{thm}
    \newtheorem{evidence}[evidence]{Evidence}
    \aliascntresetthe{evidence}
    
    \newaliascnt{question}{thm}
    \newtheorem{question}[question]{Question}
    \aliascntresetthe{question}

\theoremstyle{plain}
\newtheorem{innercustomthm}{Theorem}
\newenvironment{customthm}[1]
{\renewcommand\theinnercustomthm{#1}\innercustomthm}
{\endinnercustomthm}

\theoremstyle{plain}
\newtheorem{innercustomconj}{Conjecture}
\newenvironment{customconj}[1]
{\renewcommand\theinnercustomconj{#1}\innercustomconj}
{\endinnercustomconj}

\usepackage{cleveref}

\crefname{equation}{Eq.}{Eqs.}
\crefname{eqnarray}{Eq.}{Eqs.}
\crefname{algo}{algorithm}{algorithms}
\crefname{conj}{conjecture}{conjectures}
\crefname{lem}{lemma}{lemmas}
\crefname{thm}{theorem}{theorems}
\crefname{claim}{claim}{claims}
\crefname{rmk}{remark}{remarks}
\crefname{prop}{proposition}{propositions}
\crefname{section}{section}{sections}
\crefname{appendix}{appendix}{appendices}
\crefname{cor}{corollary}{corollaries}
\crefname{figure}{figure}{figures}
\crefname{table}{table}{tables}
\crefname{example}{example}{examples}
\crefname{nonexample}{non-example}{non-examples}
\crefname{evidence}{evidence}{evidence}
\crefname{question}{question}{questions}
\crefname{prob}{problem}{problems}
\crefname{assm}{assumption}{assumptions}
\crefname{defn}{definition}{definitions}
\crefname{notation}{notation}{notations}
\crefname{speculation}{speculation}{speculations}
\crefname{construction}{construction}{constructions}
\crefname{observation}{observation}{observations}
\crefname{innercustomthm}{Theorem}{Theorems}
\crefname{innercustomconj}{Conjecture}{Conjectures}

\usepackage[
        backend=biber,
        style=alphabetic,
        natbib=true,
        url=false, 
        doi=false,
        eprint=true,
        isbn=false,
        maxcitenames=5,
        maxbibnames=4,
        maxalphanames=4,
        sortcites=true
    ]{biblatex}

\AtEveryBibitem{\clearlist{language}}
\renewbibmacro{in:}{%
	\ifboolexpr{%
		test {\ifentrytype{article}}%
		or
		test {\ifentrytype{inproceedings}}%
	}{}{\printtext{\bibstring{in}\intitlepunct}}%
}

\DeclareFieldFormat[article,inbook,incollection,inproceedings,patent,thesis,unpublished,techreport,misc,book]{title}{\mkbibquote{#1}}

\usepackage{multirow}

\title{Experiments with membranes, maps and sheaves}

\author[D.~Holmes]{Daniel Holmes}
\address[DH]{ISTA, Am Campus 1, 3400 Klosterneuburg, Austria}
\email{daniel.holmes@ist.ac.at}
\urladdr{https://www.daniel-holmes.at}

\author[Y.~Schuler]{Yannik Schuler}
\address[YS]{ETH Zurich, Rämistrasse 101, 8092 Zurich, Switzerland}
\email{yannik.schuler@math.ethz.ch}
\urladdr{https://yannikschuler.github.io}

\date{September 2026}

\begin{document}

\begin{abstract}
    Motivated by the conjectural existence of M-theory, we investigate a web of correspondences between enumerative invariants of a Calabi--Yau fivefold $Z$ with a torus action. This includes a correspondence between the fivefold Gromov--Witten invariants and K-theoretic Pandharipande--Thomas invariants of a threefold $X \subset Z$, mediated by so-called membrane indices, which generalise Gopakumar--Vafa invariants. We establish the correspondence for strip geometries for restricted torus actions and test it numerically for general torus actions. Further numerical evidence is provided for the closed vertex, local surfaces and local projective spaces. For the latter we present conjectural formulae for their low-degree membrane indices. When the fivefold is the product of the affine plane with a suitable toric variety, we prove geometric engineering and equate the generating series of Pandharipande--Thomas invariants with the corresponding instanton partition function while equality with the Gromov--Witten series is probed numerically.
\end{abstract}

\maketitle

\tableofcontents

\section*{Introduction}
Over the past decades, research in enumerative geometry has been profoundly shaped by its connection to mathematical physics.
In particular, many correspondences between different enumerative invariants are expected to originate from the fact that their associated physical theories admit a common lift to topological M-theory. For instance, if $\CYth$ is a smooth Calabi--Yau threefold the enumeration of M2-branes on $\CYth\times \Aaff{2}$ (or more precisely a fibration of this fivefold over a circle) should tie together the enumeration of sheaves on $\CYth$, the enumeration of sheaves on $\Aaff{2}$ and the enumeration of stable maps to $X\times \Aaff{2}$. Each of these correspondences is, however, expected to hold in larger generality.

The purpose of this paper is to provide new support for this web of correspondences and to explore their consequences, both by proving theorems and, wherever a general proof appears out of reach, by verifying predictions numerically. \Cref{fig: summary paper} summarises the interconnection between our results and all conjectures investigated. The remainder of this section presents each of our results in detail and puts them into context.
\begin{figure}
    \centering
    \begin{tikzpicture}[outer sep = 0.5em, scale = 0.95]
        \node[draw, align = center, minimum height = 3.5em, minimum width = 8em] (A) at (0,6.5) {M-theory\\on CY5 $\CYfi$};
        \node[draw, align = center, minimum height = 3.5em, minimum width = 8em] (B) at (0,4) {Membrane index\\$\Omega_\beta$ of CY5 $\CYfi$};
        \node[draw, align = center, minimum height = 3.5em, minimum width = 8em] (C) at (-6,0) {GW theory\\of CY5 $\CYfi$};
        \node[draw, align = center, minimum height = 3.5em, minimum width = 8em] (D) at (0,-2) {DT theory of\\3-fold $\CYth$};
        \node[draw,align = center, minimum height = 3.5em, minimum width = 8em] (E) at (6,0) {Pure $\mr{SU}(N)$ gauge\\theory on $\Aaff{2}$};
        \draw[<->, thick] ([yshift=0.75em]C.east) to[bend left=30] ([yshift=0.75em]E.west);
        \draw[fill=white, draw = white] ([xshift=-2em]B.south) rectangle ([xshift=2em]D.north);
        \draw[<->, thick] ([yshift=-0.75em]C.east) to[bend left=30] ([xshift=-1.5em]D.north);
        \draw[<->, thick] ([xshift=1.5em]D.north) to[bend left=30] ([yshift=-0.75em]E.west);
        \node[rotate=27] at (-3.3,1) {\tiny Conj.~\ref{conj: gauge GW correspondence}};
        \node[rotate=9] at (-3.1,-0.4) {\tiny Conj.~\ref{conj: GW PT}};
        \node[rotate=-7] at (3.1,-0.4) {\tiny Thm.~\ref{thm: gauge PT correspondence}};
        \draw[->, thick] (A.south) -- (B.north);
        \draw[->, thick] (B.south west) -- (C.north) node[midway, sloped, inner sep=-0.1em, anchor=south] {\tiny Conj.~\ref{conj: gen GV} \& \cite[Conj.~M]{BS24:refGW}};
        \draw[->, thick] (B.south) -- (D.north) node[midway, sloped, inner sep=0.1em, anchor=north] {\tiny $Z = \Tot_X \, \cL_4 \oplus \cL_5$} node[midway, sloped, inner sep=-0.1em, anchor=south] {\tiny \cite[Conj.~2.1]{NO14:membranes}};
        \draw[->, thick] (B.south east) -- (E.north) node[midway, sloped, inner sep=0.1em, anchor=north] {\tiny $Z=X_{N,m} \times \Aaff{2}$} node[midway, sloped, inner sep=-0.1em, anchor=south] {\tiny \cite{Nek05:ChasingM}};
        \draw[<-, thick, gray] ([xshift=-1.2em]D.south) to[bend left=45] ([xshift=1.5em]C.south);
        \draw[<-, thick, gray] (C.south) to[bend right=55] ([xshift=1.2em]D.south);
        %
        \node[color=gray, rotate=-30] at (-4.2,-3.5) {\tiny Conj.~\ref{conj: rigidity} \& \ref{conj: strip GW formula}};
        \node[color=gray, rotate=-30] at (-3.6,-2.5) {\tiny Conj.~\ref{conj: PT sym} \& \ref{conj: curious CY4 vanishing PT}};
    \end{tikzpicture}
    \caption{Interconnection between different enumerative invariants associated with Calabi--Yau fivefolds and their subvarieties. The gray arrows indicate consequences of \Cref{conj: GW PT}.}
    \label{fig: summary paper}
\end{figure}
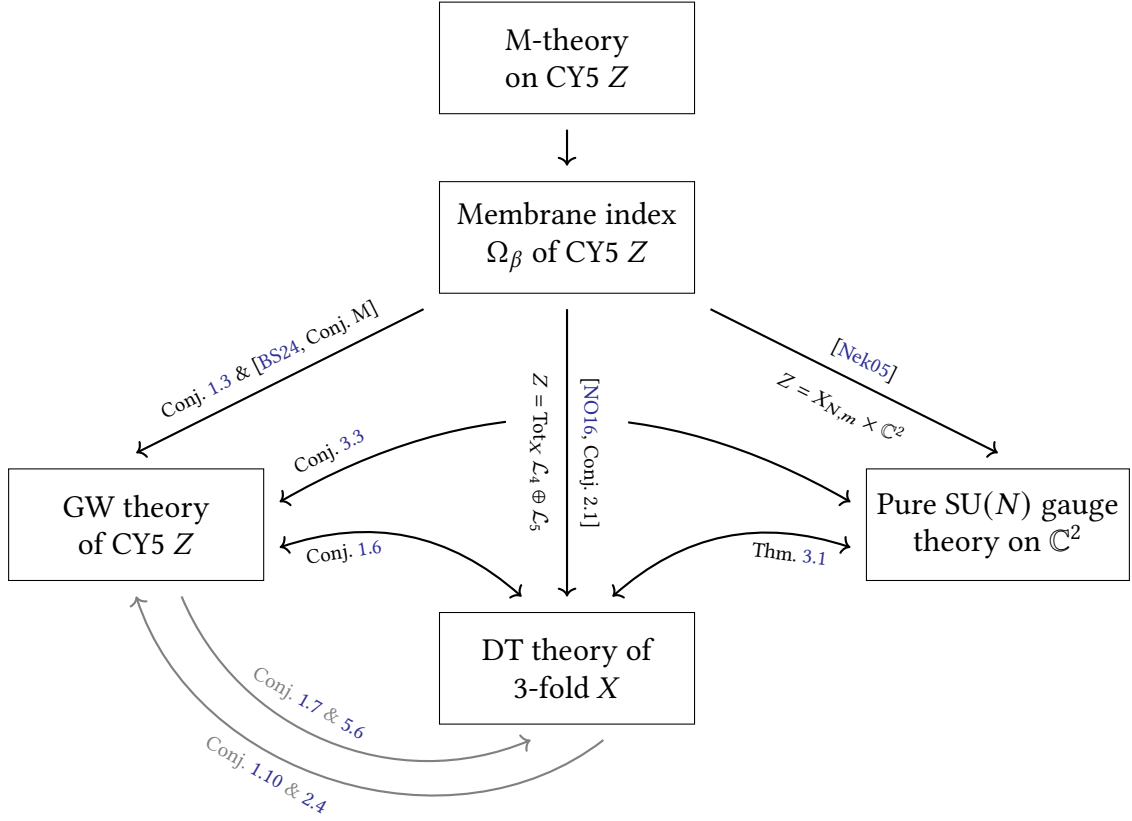

\subsection{K-theoretic maps/pairs correspondence}
One of the main correspondences studied in this paper will be a K-theoretic variant of the MNOP conjecture \cite{MNOPI} (equivalently, the Gromov--Witten/Pandharipande--Thomas correspondence \cite{PT09:StablePairs}) in the context of Calabi--Yau fivefolds.

Gromov--Witten invariants can be defined for targets of any dimension. So for a Calabi--Yau fivefold $\CYfi$ with the action of a torus $\gpT$ we may form the generating series of $\gpT$-equivariant Gromov--Witten invariants of $\CYfi$:
\begin{equation*}
    \GWarg{\CYfi}{\gpT}{}(Q,u;\epsilon_i) =\sum_{\beta \neq 0} Q^{\beta}\sum_{g\geq 0} u^{2g-2} \intEquiv{\mgpT}_{[\Mbar_g(\CYfi,\beta)]^{\vir}_{\gpT}} 1\,.
\end{equation*}
This is a formal series in a variable $Q$ recording the curve class of stable maps and a variable $u$ recording the domain genus. The series takes values in rational functions in equivariant variables $\epsilon_1,\ldots,\epsilon_r \in\LieT^\vee = \mr{Lie}\,\gpT^\vee$ generating $H^*_{\gpT}(\mr{pt})$.

In contrast, in Donaldson--Thomas theory there is no framework that allows the definition of genuine Donaldson--Thomas invariants of Calabi--Yau fivefolds to this day. As a substitute, whenever $\CYfi$ is the total space of two line bundles over a threefold $\CYth$
\begin{equation}
    \label{eq: intro CY5 PT}
    \CYfi = \Tot{}_{\CYth} \, \cL_4 \oplus \cL_5
\end{equation}
with a $\gpT'$-action, Nekrasov and Okounkov propose to study the generating series of equivariant K-theoretic Pandharipande--Thomas invariants of $\CYth$ with a suitable insertion depending on $\cL_4$ and $\cL_5$. (See \Cref{sec: PT setup} for details.) We denote this series by
\begin{equation*}
    \PTarg{}{\CYth}{\CYfi}{\Gm_q}{\gpT'}(Q,q;q_i)\,.
\end{equation*}
Like the \GWs , it depends on a variable $Q$ recording the support of the stable pair and on the generators $q_i = \re^{u \epsilon_i}$ of the representation ring of the torus $\gpT'$.
In addition, the series depends on the so-called box-counting variable $q$. Following to Nekrasov and Okounkov this formal variable should be interpreted as a coordinate on a torus $\Gm_q$ acting on the fibres of $\CYfi \rightarrow \CYth$ with opposite weights $\pm 1$. In other words, $q$ should stand on equal footing with the equivariant variables $q_i$. Indeed, what led Nekrasov and Okounkov to study such generating series of K-theoretic invariants is that, although a priori recording threefold invariants, the series behaves much like one enumerating curves in $\CYfi$. This claim can be made more concrete as follows.

\begin{customconj}{A} (=\Cref{conj: GW PT})
    \label{conj: intro GW PT}
    For all local equivariant Calabi--Yau fivefolds $(\CYfi,\gpT)$ of the form \eqref{eq: intro CY5 PT} we have
    \begin{equation}
        \label{eq: intro GW PT}
        \PTarg{}{\CYth}{\CYfi}{\Gm_q}{\gpT'}(Q,q;q_i) = \exp \, \GWarg{\CYfi}{\Gm_q \times \gpT'}{}(Q,u;\epsilon_i)
    \end{equation}
    under the natural change of variables $q_i = \re^{u \epsilon_i}$, $q = \re^{u c_1(q)}$ relating $\Gm_q \times \gpT'$-equivariant K-theory and cohomology of a point.
\end{customconj}

This conjecture was first mentioned by Brini and the second author in \cite[Conj.~7.19]{BS24:refGW} and coincides with the original Gromov--Witten/Pandharipande--Thomas conjecture when $\CYth$ is a Calabi--Yau threefold and $\gpT'$ fixes the holomorphic threeform \cite[Sec.~7.2.3]{BS24:refGW}.

\textbf{Evidence for \Cref{conj: intro GW PT}.} One of the main contributions of this work is to provide the following evidence for \Cref{conj: intro GW PT} and its consequences:
\begin{enumerate}
    \item\label{item: intro check first} When $\CYth$ is a so-called toric strip geometry and $\cL_4=\cL_5=\cO_{\CYth}$ we prove a closed formula for the \PTs\ (\Cref{prop: strip PT formula}). This is done via the formula for the two-legged capped vertex of Kononov--Okounkov--Osinenko \cite{KOO21:twoLegPT}. We then prove \Cref{conj: intro GW PT} under certain assumptions on the $\gpT$-action (\Cref{prop: strip GW formula limit}) and probe the conjecture numerically without any assumption on the torus action by computing the Gromov--Witten invariants in low genus and curve degree (\Cref{ev: strip GW}).

    \item When $\CYth$ is the so-called closed vertex and $\cL_4=\cL_5=\cO_{\CYth}$ we numerically test \Cref{conj: intro GW PT} in low degree via a conjectural closed formula (\Cref{conj: closed vertex formula}) for the GW and PT series (\Cref{ev: closed vertex formula PT,ev: closed vertex formula GW}).

    \item When $\CYth = X_{N,m}$ is a toric Calabi--Yau threefold engineering pure supersymmetric $\mr{SU}(N)$ gauge theory on $\Aaff{2}$ we probe \Cref{conj: intro GW PT} by matching an explicit formula for the \PTs\ with numerical low-degree and low-genus calculations on the Gromov--Witten side (\Cref{ev: gauge GW correspondence}). The formula for the \PTs\ (\Cref{thm: gauge PT correspondence}) is derived via the refined topological vertex. In physics literature this formula is well-known \cite{Tak08:RefVertInstanton}. Hence, the only task is to carefully check that the mathematical setup indeed reproduces the one appearing in physics literature. This largely builds on work of Arbesfeld \cite{Arb21:KthDTHilbSurf}.

    \item\label{item: intro local Pn} We test \Cref{conj: intro GW PT} for all local projective spaces $\Tot{}_{\bP^n} \,\cO(a_{n+1}) \oplus \ldots \oplus \cO(a_{5})$ with $a_i \leq 0$. In low degree, both sides of the correspondence are calculated numerically and checked to admit a common analytic continuation (\Cref{ev: P2 O min 1,ev: P2 O min 1 PT,ev: local P3,ev: local P3 PT}).

    \item\label{item: intro check last} We claim that for local projective surfaces $\CYfi = \Tot{}_{S} \, \cL_1 \oplus \cL_2 \oplus \cO_{S}$ with $\cL_1$, $\cL_2$ sufficiently negative \Cref{conj: intro GW PT} holds trivially with both sides of the correspondence vanishing identically (\Cref{conj: curious CY4 vanishing,conj: curious CY4 vanishing PT}). We check our claim in numerical experiments for several targets $\CYfi$ in low degree (\Cref{ev: curious CY4 vanishing,ev: curious CY4 vanishing PT}). Under an additional assumption on the torus action we prove that the vanishing on the Gromov--Witten side indeed holds true (\Cref{prop: curious CY4 vanishing limit}).
\end{enumerate}

The need for numerical methods to check \Cref{conj: intro GW PT} arises from the issue that the \PTs\ is usually computable in limits where the \GWs\ is inaccessible and vice versa. For instance when $\CYfi = \CYth \times \Aaff{2}$ the \PTs\ can often be evaluated via the refined topological vertex. On the GW side, however, this corresponds to a limit where some variables $\epsilon_i$ are sent to infinity. Since the \GWs\ is presented as formal series expanded at $\epsilon_i=0$, taking such a limit first requires an analytic continuation which in most cases is out of reach with today's methods. On the other hand, a restriction to locally self-dual torus actions, in the sense of \cite[Def.~1.1]{Sch26:CY5TV}, is quite powerful on the GW side while being an a priori ill-defined operation on the PT side.

Thus, by the lack of methods to control the GW and \PTs\ simultaneously we are currently forced to resort to numerical methods to provide evidence for \Cref{conj: intro GW PT} and the various other conjectures discussed in this paper. We hope that the case studies presented and questions raised will spark new ideas that hopefully lead to the development of novel techniques required to prove \Cref{conj: intro GW PT} and others.

\subsection{Membrane indices}
As already mentioned, many correspondences in enumerative geometry are expected to be explained by M-theory. \Cref{conj: intro GW PT} is no exception. It is the central speculation of the seminal work of Nekrasov and Okounkov \cite{NO14:membranes} that an analytic continuation of the generating series of K-theoretic Pandharipande--Thomas invariants in the box-counting variable $q$ is provided by the index of the yet-to-be-constructed moduli space of M2-branes on $\CYfi$. We will refer to these invariants as \meminds\ for short. A parallel claim in Gromov--Witten theory was made by Brini and the second author in \cite[Conj.~M]{BS24:refGW}. The speculation is that there should exist some moduli space $\MTwoMod_{\beta}$ parametrising M2-branes in curve class $\beta$ on $\CYfi$ equipped with some line bundle $\cE$ (or some suitable element in K-theory). Its equivariant Euler characteristic
\begin{equation*}
    \Omega_\beta \coloneqq \chi_{\gpT} \big(\MTwoMod_{\beta}, \cE \big) \,,
\end{equation*}
viewed as a function on the torus $\gpT = \Gm_q \times \gpT'$, should then provide the analytic continuation from the PT to the GW limit in \Cref{conj: intro GW PT}. This is only after forming the generating series over all curve classes and taking the plethystic exponential:
\begin{equation}
    \label{eq: intro omega gen series}
    \Exp \sum_{\beta \neq 0} Q^{\beta} ~ \Omega_\beta\,.
\end{equation}
To this day this speculation can at most serve as a guiding principle since a proposal for the moduli space of M2-branes is only available in very few cases \cite{NO14:membranes,GPS26:5Hodge}. However, we may still test the numerical implications of the conjecture.

\begin{customconj}{B} (=\Cref{conj: gen GV,conj: PT GV})
    \label{conj: intro GW PT membr}
    Suppose $(\CYfi,\gpT)$ is an equivariantly Calabi--Yau fivefold. There exist rational functions $\Omega_\beta$ on $\gpT$ with constrained poles and coefficients in $\bZ[\frac{1}{2}]$ so that the series \eqref{eq: intro omega gen series} equates to the \GWs\ of $(\CYfi,\gpT)$. Moreover, if the fivefold is of the form \eqref{eq: intro CY5 PT} the series of \meminds\ also equates to the \PTs .
\end{customconj}

When $\CYfi = \CYth \times \Aaff{2}$ and $\gpT$ fixes the holomorphic threeform of $\CYth$ this conjecture reduces to a weak version of the Gopakumar--Vafa conjecture as explained in \cite[Sec.~0.3]{Sch26:CY5TV}.

\textbf{Evidence for \Cref{conj: intro GW PT membr}.} In all cases \labelcref{item: intro check first}--\labelcref{item: intro check last} for which we checked \Cref{conj: intro GW PT} we also verified that the common lift of the PT and \GWs\ indeed satisfies the assertions of \Cref{conj: intro GW PT membr}. We remark that earlier evidence for \Cref{conj: intro GW PT membr} mostly concerned the situation $\CYfi = \CYth \times \Aaff{2}$ or needed strong assumptions on the torus action. In particular, our numerical tests \labelcref{item: intro local Pn} and \labelcref{item: intro check last} go well beyond these limits.

\subsection{Geometric engineering of gauge theories}
A correspondence that is well-known in the physics literature but far less studied in the mathematical one is the concept of geometric engineering of supersymmetric gauge theories on $\Aaff{2}$. In the case of pure supersymmetric $\mr{SU}(N)$ Yang--Mills theory with Chern--Simons level $m\in\bZ$ on $\Aaff{2}\times S^1_u$, the correspondence asserts that the generating series of GW invariants of a suitable toric Calabi--Yau threefold $X_{N,m}$ equates to the generating series $$\ZNek_{N,m}\big(q,q^{-1}\big)$$ of \smash{$\Gm_q$}-equivariant $\hat{A}$-genera of instanton moduli spaces modulo perturbative corrections. Here, the \smash{$\Gm_q$}-action on instanton moduli spaces descends from the anti-diagonal action of torus on the affine plane. In other words, \smash{$\ZNek_{N,m}(q,q^{-1})$} is the (K-theoretic) 5d Nekrasov partition function on the self-dual $\Omega$ background.

In the physics literature it was long expected that there should exist a suitable quantum field theory, usually called the refined topological string, which generalises geometric engineering to the general $\Omega$ background, i.e.~beyond anti-diagonal torus actions on the affine plane. However, an explicit worldsheet formulation of such a theory remained elusive for many years. Only recently Brini and the second author proposed that the equivariant Gromov--Witten theory of $(X_{N,m} \times \Aaff{2}, \gpT)$ should provide this sought-after refinement. Here, the refinement arises from the ability to choose an arbitrary $\gpT$-action on the affine plane and the observation that the \smash{$\Gm_q$}-equivariant Gromov--Witten invariants of $X_{N,m} \times \Aaff{2}$ coincide with the ones of $X_{N,m}$ with the generator of \smash{$\mr{Lie}\,\Gm_q$} effectively acting as a genus counting variable. Let us make our claim more precise.

\begin{customconj}{C} (=\Cref{conj: gauge GW correspondence})
    \label{conj: intro gauge GW correspondence}
    For $|m|<N$ we have
    \begin{equation}
        \label{eq: intro gauge GW correspondence}
        \exp \GWarg{X_{N,m} \times \Aaff{2}}{\gpT}{}(\epsilon_4,\epsilon_5) = \ZNek_{N,m}\big(\re^{u\epsilon_4},\re^{u\epsilon_5}\big)
    \end{equation}
    where $\epsilon_4$ and $\epsilon_5$ are the tangent $\gpT$-weights of $\Aaff{2}$.
\end{customconj}

This conjecture should generalise in several directions: a suitable modification of the threefold ought to accommodate non-trivial fundamental matter, quiver gauge theories and other gauge groups. On the surface side, we expect an extension to quotients $[\Aaff{2}/\Gamma]$ for finite subgroups $\Gamma\subset\SL_2$. We leave the investigation of these generalisations to future work.

\textbf{Evidence for \Cref{conj: intro gauge GW correspondence}.} In the self-dual limit $\epsilon_4 + \epsilon_5 =0$ the correspondence was proven by Iqbal and Kashani-Poor in \cite{IKP06:SUNtopstr}. We provide numerical evidence for the conjecture in low genus and curve degree beyond the self-dual limit in \Cref{ev: gauge GW correspondence}.

Looking at equation \eqref{eq: intro gauge GW correspondence}, note also that in agreement with M-theoretic intuition the genus counting variable $u$ of the \GWs\ can be identified with the radius of the additional $S^1$ providing the elventh space-time dimension. Indeed, this matches the expectation that in the limit where this circle gets contracted, i.e.~we perform a Taylor expansion around $u=0$, M-theory recovers the refined topological string.

We remark that the hypothesis on the Chern--Simons level $m$ in \Cref{conj: intro gauge GW correspondence} is required to ensure that the left-hand side \GWs\ solely depends on the equivariant variables $\epsilon_4$ and $\epsilon_5$. In principle the series may depend on other elements in $\LieT^\vee$ which is what we experimentally observe when $|m|\geq N$. It is natural to ask whether these additional variables admit any interpretation on the left-hand gauge theory side of \Cref{conj: gauge GW correspondence}.

One may interpret the restricted dependence of the \GWs\ on equivariant variables as a consequence of \Cref{conj: intro GW PT}. On the PT side this is a well-known feature called rigidity. It was proven in \cite[Prop.~7.5]{NO14:membranes} that whenever the moduli space of stable pairs on a Calabi--Yau threefold is proper, its K-theoretic PT invariant will only depend on the character $\kappa$ scaling the holomorphic threeform. This means, one may compute these invariants under any restriction of the torus action fixing $\kappa$. A particularly potent restriction is the so-called refined limit where the invariants are governed by the refined topological vertex as explained in \cite{Arb21:KthDTHilbSurf}. This provides enough leverage to equate the \PTs\ of $X_{N,m}$ to Nekrasov's partition function.

\begin{customthm}{D} (=\Cref{thm: gauge PT correspondence})
    \label{thm: intro gauge PT correspondence}
    For $|m|<N$ we have $$\PTarg{}{X_{N,m}}{X_{N,m} \times \Aaff{2}}{\Gm_q}{\gpT'}(q,\kappa) = \ZNek_{N,m}\big(q\kappa^{1/2},q^{-1}\kappa^{1/2}\big)\,.$$
\end{customthm}

We note that this result is well compatible with \Cref{conj: intro GW PT,conj: intro gauge GW correspondence}.

We acknowledge that there are many instances in the physics literature where the refined topological string partition function has been evaluated by means of the refined topological vertex. See for instance \cite{Tak08:RefVertInstanton}. So the only task in proving \Cref{thm: intro gauge PT correspondence} is to demonstrate the compatibility between the vertex formalism in K-theoretic PT theory with the one used in physics literature. We only do this in the particular case of $X_{N,m}$ but we expect that more generally, the vertex rules described by Arbesfeld in \cite{Arb21:KthDTHilbSurf} should indeed simplify to the ones presented in \cite[Sec.~4.1]{BMPTY14:NonLagBrnJunc}. Establishing this identification would therefore immediately generalise \Cref{thm: intro gauge PT correspondence} beyond pure gauge theory to all situations previously studied in physics literature. We leave this bookkeeping exercise for future research.

\subsection{Numerical methods and documentation}
Most evidence presented for \Cref{conj: intro GW PT,conj: intro GW PT membr,conj: intro gauge GW correspondence} in this paper is supported by direct numerical calculations. All code used in these calculations together with extensive documentation can be found in \cite{website}. 

\subsubsection{Maps}
On the side of Gromov--Witten theory, our computations are streamlined by the Julia package \texttt{GW\_CY5.jl}, which
\begin{itemize}
    \item uses \texttt{GKMtools.jl} \cite{Holmes_Muratore_2025a} to compute equivariant Gromov--Witten invariants of the relevant spaces up to a certain degree and genus,
    \item computes the leading terms of the membrane indices $\Omega_\beta$ in the variables $u$ and $\epsilon_i$ using the previously obtained Gromov--Witten invariants, and
    \item automatically compares the result with any predefined predicted value of $\Omega_\beta$.
\end{itemize}
In practice, the prediction comes from explicit computations in Pandharipande--Thomas theory via Conjectures~\ref{conj: intro GW PT}, \ref{conj: intro GW PT membr}, and the explicit expression in \Cref{conj: intro gauge GW correspondence}. We refer the interested reader to our documentation \cite{website} for details and especially to \Cref{ex: A1 x C3 GW,ex: A1 x C3 Omega} for working examples.

\subsubsection{Sheaves}
On the side of Pandharipande--Thomas theory we use a SageMath package \texttt{PT\_CY5.sage} that allows to compute K-theoretic PT invariants of toric varieties via the vertex formalism. This code greatly relies on Liu's implementation of the K-theoretic PT vertex in the SageMath package \texttt{boxcounting} \cite{Liu:boxcounting}. Additionally to Liu's implementation our package supports Nekrasov--Okounkov insertions which are required in the study of fivefolds beyond $\CYth \times \Aaff{2}$. Again, we refer the interested reader to our documentation and to \Cref{ex: running 3,ex: running 4} for two working examples.

\subsection*{Acknowledgements} 
We thank Noah Arbesfeld, Andrea Brini, Henry Liu, Alessandro Giacchetto, Davesh Maulik, Sergej Monavari and Rahul Pandharipande for fruitful discussions and helpful comments. The first author especially thanks Giosu{\`e} Muratore for continued fruitful collaboration on \texttt{GKMtools.jl} \cite{Holmes_Muratore_2025b}.
Finally, we thank the IT Services Group of D-MATH at ETH Z{\"u}rich for continued support in using their Compute Clients.
The first author did this work during his PhD studies at the Institute of Science and Technology Austria (ISTA), supported by
the ERC grant ViaFiPoS, No 101199663 and, from September 2026, by a DOC Fellowship of the Austrian Academy of Sciences.
The second author was supported by the DFG Walter Benjamin Fellowship 576663726 and the SNF grant SNF-200020-219369.

\subsection*{Statement on the use of generative AI}
The mathematical content of the paper is due to the authors. Anthropic's \texttt{Claude} was solely used to improve the wording and grammar of the paper. While most of \texttt{GW\_CY5.jl} is human-written, we used OpenAI's \texttt{Codex} with \texttt{GPT-5.2} to \texttt{GPT-5.6} and Anthropic's \texttt{Claude Code} with \texttt{Opus 4.5} to \texttt{Opus 5} to implement certain parts of it, including the formula for $\ZNek_{N,m}$ in \Cref{conj: gauge GW correspondence}.
We carefully checked and corrected the resulting code.

\section{Setup, methods and central conjectures}
\label{sec: setup}
This section introduces the main enumerative invariants investigated in this note: K-theoretic Pandharipande--Thomas invariants of threefolds and Gromov--Witten invariants and \meminds\ of Calabi--Yau fivefolds. We discuss the methods we will later employ for their computation and mention the central conjectures relating the different invariants.

\subsection{Geometric preliminaries}
First, we recall some necessary geometric background.

\subsubsection{Calabi--Yau GKM spaces}
In the following let $\CYfi$ be a smooth quasi-projective variety with the action of a torus $\gpT\cong\Gm[r]$. We say that $(\CYfi,\gpT)$ is \define{equivariantly Calabi--Yau} if $\CYfi$ admits a nowhere vanishing $\gpT$-invariant holomorphic volume form. In other words, there must be a $\gpT$-equivariant isomorphism $\omega_{\CYfi}\cong \cO_{\CYfi}$. We will also often call such torus actions Calabi--Yau.

We say that $(\CYfi,\gpT)$ is a \define{GKM space} if the $\gpT$-action has finitely many fixed points and one-dimensional torus orbits, and at least one fixed point. We are also going to assume that each one-dimensional torus orbit closure contains at least one torus fixed point. This means all orbit closures are isomorphic to either $\bP^1$ or $\Aone$. We will refer to such a subvariety as a $\gpT$\define{-stable} $\bP^1$ respectively $\Aone$.

We remark that there are varying definitions for GKM spaces in the literature (cf.~\cite{GKM98:EqCohKosLocThm,GKZ24:LowDimGKM}). The characterisation we choose here is what is best suited for our applications.

\subsubsection{GKM graphs}
To a GKM space $(\CYfi,\gpT)$ we associate a combinatorial device recording all geometric data relevant for later applications. This is the so-called \define{GKM graph} $G$ of $(\CYfi,\gpT)$. It is a decorated graph whose vertices $v$ are the $\gpT$-fixed points of $\CYfi$. The edges and leaves $e$ of $G$ are given by the $\gpT$-stable $\bP^1$s respectively $\Aone$s. Every half-edge $h=(e,v)$ is moreover decorated by a $\gpT$-weight $\epsilon_{h}\in \LieT^\vee$. This weight is selected as follows: If $p_v$ denotes the torus fixed point associated to $v$ and $C_e$ is the $\gpT$-stable line corresponding to $e$ then
\begin{equation*}
    \epsilon_{h} = c^{\gpT}_1(T_{p_v}C_e)
\end{equation*}
is its tangent weight at the fixed point.

\begin{example}
    \label{ex: A1 x C3 GKM graph}
    As an illustrating example, let us consider the quotient
    \begin{equation*}
        \CYfi = \Aaff{6}\setminus V(x_0 x_1) \Big/ \Gm
    \end{equation*}
    where $\Gm$ acts on affine six-space via
    \begin{equation*}
    \begin{tikzcd}[row sep=0em]
        \Gm \times \Aaff{6} \ar[r] & \Aaff{6}\,, \\
        \big(t,(x_0,x_1,x_2,x_3,x_4,x_5)\big) \ar[r,maps to] & (t x_0,t x_1,t^{-2} x_2,x_3,x_4,x_5) \,.
    \end{tikzcd}
    \end{equation*}
    We can identify $\CYfi = \Tot_{\bP^1}\!\big(\cO(-2)\big)\times \Aaff{3}$ which means the quotient is Calabi--Yau. The natural $\Gm[6]$-action on $\Aaff{6}$, whose tangent weights at the origin we denote by $\epsilon_0,\ldots,\epsilon_5$, descends to an action on the quotient $\CYfi$. The tangent weights at the two torus fixed points $0,\infty \in \bP^1 \subset \CYfi$ are
    \begin{equation*}
        (\epsilon_1-\epsilon_0,\epsilon_2+2\epsilon_0,\epsilon_3,\epsilon_4,\epsilon_5) \qquad\text{and}\qquad
        (\epsilon_0-\epsilon_1,\epsilon_2+2\epsilon_1,\epsilon_3,\epsilon_4,\epsilon_5) 
    \end{equation*}
    respectively. Note that $(\CYfi,\Gm[6])$ is \textit{not} equivariantly Calabi--Yau since the holomorphic five-form is scaled non-trivially with weight $-\sum_{i=0}^5 \epsilon_i$. However, if we pass to a subtorus $\gpT \cong \Gm[4]$ where this weight vanishes then $(\CYfi,\gpT)$ is indeed equivariantly Calabi--Yau.
    
    We present the embedding of the GKM graph $G$ of $(\CYfi,\Gm[6])$ into $\bR^2$ induced by the moment map of some subtorus $\Gm[2] \subset \Gm[6]$ in \Cref{fig: A1 x C3 GKM graph}. Note that the graph contains a single edge associated to the zero section $\bP^1\subset \CYfi$ connecting the two torus fixed points.
\end{example}

\begin{figure}
\centering
\begin{tikzpicture}[scale=1.1]
    \draw (-1,-1) -- (0,0) -- (3,0) -- (4,-1);
    \node[circle,inner sep=2pt,fill=black] at (0,0){};
    \node[circle,inner sep=2pt,fill=black] at (3,0){};
    \node[below] at (0.6,0) {$\epsilon_1-\epsilon_0$};
    \node[below] at (2.4,0) {$\epsilon_0-\epsilon_1$};
    \node[left] at ($(0,0)!0.6!(-1,-1)$) {$\epsilon_2+2\epsilon_0$};
    \node[right] at ($(3,0)!0.6!(4,-1)$) {$\epsilon_2+2\epsilon_1$};
    \draw (0,0) -- (0.64, 1.26);
    \draw (0,0) -- (-0.44, 1.34);
    \draw (0,0) -- (-1.26, 0.64);
    \node[left] at ($(0,0)!0.65!(0.64, 1.26)$) {$\epsilon_3$};
    \node[left] at ($(0,0)!0.6!(-0.44, 1.34)$) {$\epsilon_4$};
    \node[left] at ($(0,-0.1)!0.4!(-1.26, 0.6)$) {$\epsilon_5$};
    \draw (3,0) -- (3.64, 1.26);
    \draw (3,0) -- (2.56, 1.34);
    \draw (3,0) -- (1.74, 0.64);
    \node[left] at ($(3,0)!0.7!(3.64, 1.26)$) {$\epsilon_3$};
    \node[left] at ($(3,0)!0.6!(2.56, 1.34)$) {$\epsilon_4$};
    \node[left] at ($(3,-0.1)!0.4!(1.74, 0.6)$) {$\epsilon_5$};
\end{tikzpicture}
\caption{The GKM graph of $\CYfi = \Tot_{\bP^1}\!\big(\cO(-2)\big)\times \Aaff{3}$.}
\label{fig: A1 x C3 GKM graph}
\end{figure}

\subsection{Maps}
Next we want to recall features of Gromov--Witten invariants of GKM spaces. Let us first fix their notation.

\subsubsection{Gromov--Witten invariants}
As before, suppose $\CYfi$ is a smooth quasi-projective variety with the action of a torus $\gpT\cong \Gm[r]$. We consider the $\gpT$-equivariant genus-$g$ Gromov--Witten invariant of $\CYfi$ in curve class $\beta$:
\begin{equation*}
    \GWarg{\CYfi}{\gpT}{g,\beta} =\intEquiv{\mgpT}_{[\Mbar_g(\CYfi,\beta)]^{\vir}_{\gpT}} 1\,.
\end{equation*}
In case the moduli space of stable maps to $\CYfi$ is non-proper we define the above intersection number as its $\gpT$-equivariant residue. For this we will always implicitly assume that the locus of $\gpT$-fixed stable maps is proper. This way, equivariant Gromov--Witten invariants are generally rational functions in the equivariant parameters:
\begin{equation*}
    \GWarg{\CYfi}{\gpT}{g,\beta}\in \mr{Frac}\big(H^*_{\gpT}(\mr{pt})\big) \cong \bQ(\epsilon_1,\ldots, \epsilon_r) \eqqcolon Q_{\gpT}\,.
\end{equation*}
Their denominator and numerator are of homogenous degree and the difference of the degrees is precisely given by the virtual dimension of the moduli problem.

We will write 
\begin{equation*}
    \GWarg{\CYfi}{\gpT}{\beta} \coloneqq \sum_{g\geq 0} u^{2g-2} ~ \GWarg{\CYfi}{\gpT}{g,\beta} 
\end{equation*}
for the all-genus generating series of Gromov--Witten invariants where $u$ is a formal variable. We remark that when $\CYfi$ is a Calabi--Yau fivefold --- the main setting of this paper --- this variable is redundant as each invariant $\GWarg{\CYfi}{\gpT}{g,\beta}$ is a rational function in the equivariant parameters of homogenous degree $2g-2$. Nevertheless, we keep the variable for bookkeeping purposes.

Moreover, we form the generating series over all non-zero effective curve classes $\beta\in \Heff{\CYfi}$:
\begin{equation*}
    \GWarg{\CYfi}{\gpT}{} \coloneqq \sum_{\beta \neq 0} Q^{\beta} ~ \GWarg{\CYfi}{\gpT}{\beta}\,.
\end{equation*}
This series takes values in a suitable completion of the semigroup ring of $\Heff{\CYfi}$ with coefficients in $Q_{\gpT} (\!(u)\!)$. We denote this ring as
\begin{equation*}
    Q_{\gpT} (\!(u)\!)\llbracket \Heff{\CYfi} \rrbracket\,.
\end{equation*}


\subsubsection{Vertex formalism and Hodge integrals}
For GKM spaces $(\CYfi,\gpT)$ virtual localisation \cite{GP97:virtloc} decomposes Gromov--Witten invariants as a sum over $\gpT$-fixed loci of the moduli space of stable maps. These fixed loci are labelled by decorated graphs $\Gamma$ together with a morphism of graphs $f:\Gamma \rightarrow G$ to the GKM graph of $(\CYfi,\gpT)$. Vertices $v$ of such a graph $\Gamma$ correspond to components of the domain curve that get contracted to a fixed point. This fixed point is the image of $v$ under $f$. Moreover, $v$ is decorated with an integer $g_v$ recording the arithmetic genus. Edges $e$ correspond to rational components of the domain mapping to $\gpT$-stable rational curves in $\CYfi$. They are decorated with an integer $d_e$ indicating the degree of the cover and $f(e)$ is chosen as the edge in $G$ that is being covered. Virtual localisation then yields the following formula \cite{GP97:virtloc,LS17:GWGKM}:
\begin{equation}
\label{eq: loc formula GW}\GWarg{\CYfi}{\gpT}{g,\beta} = \sum_{\Gamma,f} \frac{1}{|\Aut (\Gamma,f)|} \prod_{e\in E(\Gamma)} \!\!\! \mr{wt}(e,\Gamma) ~\prod_{v\in V(\Gamma)} \int_{\Mbar_{g_v,E(\Gamma)_v}} \!\! \frac{\prod_{h\in H(G)_{f(v)}} \epsilon_{h}^{|E(\Gamma)_v|} \HodgeLambda[g_v]{\epsilon_{h}}}{ \prod_{e\in E(\Gamma)_v}  \big( \frac{\epsilon_{f(e,v)}}{d_e}-\psi_{e}\big)}\,.
\end{equation}
Here, we denote by $E(\Gamma)_v$ the set of all edges adjacent to $v\in V(\Gamma)$ and $H(G)_{\tilde{v}}$ the set of all half-edges adjacent to a vertex $\tilde{v}\in V(G)$. The sum is over all tuples $(\Gamma,f)$ which are compatible with the discrete data $g,\beta$. Edges are weighted by an explicit rational function $\mr{wt}(e,\Gamma)$ in the torus weights and the splitting data of the normal bundle $N_{C_e/Z}$ (cf. \cite[Sec.~4.3.4]{LS17:GWGKM}). Vertices are weighted by Hodge integrals:
\begin{equation*}
    \int_{\Mbar_{g,n}} \frac{\prod_{i=1}^m \HodgeLambda[g]{\epsilon_{i}}}{ \prod_{j=1}^n  \big( z_j^{-1}-\psi_{j}\big)} \,,\qquad \HodgeLambda[g]{\epsilon} \coloneqq \sum_{k \geq 0} \lambda_k (-1)^k \epsilon^{g-k-1}\,.
\end{equation*}
By \cite{Holmes_Muratore_2025b}, the set of all relevant $(\Gamma,f)$ is completely determined by the GKM graph, and $\mr{wt}(e,\Gamma)$ can be computed from the edge weights without precise knowledge of the splitting of $N_{C_e/Z}$.
The Julia package \texttt{GKMtools.jl} \cite{Holmes_Muratore_2025a} supports the resulting computation of equivariant Gromov--Witten invariants from the GKM graph.
Building on this, our Julia package \texttt{GW\_CY5.jl} \cite{website} computes the leading terms of $\GWarg{\CYfi}{\gpT}{\beta}$ as illustrated in the following example. 

\begin{example}
\label{ex: A1 x C3 GW}
Let us continue \Cref{ex: A1 x C3 GKM graph} concerning $\CYfi = \Tot_{\bP^1}\!\big(\cO(-2)\big)\times \Aaff{3}$ and compute $\GWarg{\CYfi}{\gpT}{\beta}$ in curve class $\beta = [\bP^1]$, $2 [\bP^1]$ and $3 [\bP^1]$.
First, we load the required packages, summon the GKM graph of $\CYfi$, and set the maximum genus for the computation to $1$.
\begin{shaded}
\begin{MyVerb}
\textcolor{code}{julia>} using Oscar, GKMtools, GW_CY5
\textcolor{code}{julia>} G = gkm_graph_of_example_1_1();
\textcolor{code}{julia>} beta = curve_class(G, Edge(1, 2));
\textcolor{code}{julia>} gMax = 1;
\end{MyVerb}
\end{shaded}
Then we use \texttt{get\_GW\_beta} to compute $\GWarg{\CYfi}{\gpT}{\beta}$ in degrees $1$, $2$ and $3$. The symbols {\texttt{t1}, \ldots, \texttt{t6}} represent the equivariant parameters $\epsilon_0,\ldots,\epsilon_5$ as in \Cref{ex: A1 x C3 GKM graph}.
\begin{shaded}
\begin{MyVerb}
\textcolor{code}{julia>} GW = get_GW_beta(G, [beta, 2*beta, 3*beta], gMax)
\end{MyVerb}
\end{shaded}
\begin{shaded}
\begin{MyVerb}
\textcolor{code}{julia>} GW[beta]
(1//12*t1*t4*t5*u^2 + 1//12*t1*t4*t6*u^2 + 1//12*t1*t5*t6*u^2 + t1 + 1//12*t2*t4*t5*u^2 + 1//12*t2*t4*t6*u^2 + 1//12*t2*t5*t6*u^2 + t2 + 1//12*t3*t4*t5*u^2 + 1//12*t3*t4*t6*u^2 + 1//12*t3*t5*t6*u^2 + t3)//(t4*t5*t6*u^2)
\textcolor{code}{julia>} GW[2*beta]
(1//24*t1*t4*t5*u^2 + 1//24*t1*t4*t6*u^2 + 1//24*t1*t5*t6*u^2 + 1//8*t1 + 1//24*t2*t4*t5*u^2 + 1//24*t2*t4*t6*u^2 + 1//24*t2*t5*t6*u^2 + 1//8*t2 + 1//24*t3*t4*t5*u^2 + 1//24*t3*t4*t6*u^2 + 1//24*t3*t5*t6*u^2 + 1//8*t3)//(t4*t5*t6*u^2)
\textcolor{code}{julia>} GW[3*beta]
(1//36*t1*t4*t5*u^2 + 1//36*t1*t4*t6*u^2 + 1//36*t1*t5*t6*u^2 + 1//27*t1 + 1//36*t2*t4*t5*u^2 + 1//36*t2*t4*t6*u^2 + 1//36*t2*t5*t6*u^2 + 1//27*t2 + 1//36*t3*t4*t5*u^2 + 1//36*t3*t4*t6*u^2 + 1//36*t3*t5*t6*u^2 + 1//27*t3)//(t4*t5*t6*u^2)
\end{MyVerb}
\end{shaded}
Note that after cancelling common factors of $u$ in the numerator and denominator, the appearing powers of $u$ are precisely $u^{-2}$ and $u^0$, compatibly with our maximum genus $1$.
\end{example}

\subsection{Membranes}
From now on we will always assume that $(\CYfi, \gpT)$ is an equivariantly Calabi--Yau fivefold. In \cite[Conj.~M]{BS24:refGW} Brini and the second author speculate that the equivariant Gromov--Witten invariants of such a target are governed by a smaller set of (almost) integer valued invariants called membrane indices. This conjecture is stated more explicitly in \cite[Conj.~A]{Sch26:CY5TV}. Recall that the all-genus generating series $\GWarg{\CYfi}{\gpT}{\beta}$ is a formal series in the genus variable $u$ and the equivariant parameters $\epsilon_1, \ldots,\epsilon_r\in \LieT^\vee$ generating $H^*_{\gpT}(\mr{pt})$. The conjecture then asserts that this formal series is the Laurent expansion of a rational function in \smash{$\re^{u \epsilon_1/2},\ldots,\re^{u \epsilon_r/2}$} around $u=0$ with constrained poles. Moreover, modulo subtracting multi-covering contributions similar to the definition of Gopakumar--Vafa invariants the coefficients of these rational functions should conjecturally lie in $\mathbb{Z}[\tfrac{1}{2}]$. 

\begin{conj}
\label{conj: gen GV}
    \cite[Conj.~A]{Sch26:CY5TV}
    Suppose $(\CYfi,\gpT)$ is an equivariantly Calabi--Yau fivefold. There exist rational functions
    \begin{equation*}
        \Omega_{\beta}(q_i) \in \bZ\big[\tfrac{1}{2}\big]\left[ q_{1}^{\pm 1/2}, \ldots, q_{r}^{\pm 1/2} , \, \left\{\frac{1}{1-\prod_i q_i^{n_i/2}}\right\}_{\boldsymbol{n}\in \bZ^{r}\setminus \{0\}} \,\right]
    \end{equation*}
    labelled by effective curve classes $\beta$ in $\CYfi$ such that under the change of variables $q_i= \re^{u \epsilon_i}$
    \begin{equation}
        \label{eq: generalised GV}
        \GWarg{Z}{\gpT}{\beta} = \sum_{k \,\mid\, \beta} \frac{1}{k} \Omega_{\beta/k}(q_i^k)
    \end{equation}
    where the sum is over all integers $k>0$ dividing $\beta$, i.e.~$k\beta' = \beta$ for some $\beta'\in\Heff{\CYfi}$.
\end{conj}

We will refer to the invariant $\Omega_{\beta}$ as the \define{\memind} of $\CYfi$ in curve class $\beta$. 

One of the objectives of this work is to provide numerical evidence for \Cref{conj: gen GV}. The conjecture was earlier investigated in \cite{Sch26:CY5TV} in the context of GKM spaces under a restrictive hypothesis on the torus action. This hypothesis essentially ensured that all quintuple Hodge integrals in the localisation formula \eqref{eq: loc formula GW} reduce to triple Hodge integrals which are well-known in the literature and for instance governed by the topological vertex \cite{OP04:HodgeIntUnknot,LLLZ09:TV,LLZ03:MarinoVafaFormula}. In this paper we will go beyond such restricted torus actions and embrace honest quintuple Hodge integrals in concrete numerical experiments.

Regarding these experiments, note that equation \eqref{eq: generalised GV} can be inverted to determine the expansion of $\Omega_{\beta}$ around $u=0$ from the collection of all \GWs\ $\GWarg{Z}{\gpT}{\beta/k}$ where $k$ divides $\beta$.
We implemented this in our Julia package \texttt{GW\_CY5.jl}.
In our numerical experiments, we will compute the true values of $\Omega_\beta$ up to some order in $u$ and compare the results to some educated guesses for their analytic continuation.
A full documentation of the workflow using \texttt{GW\_CY5.jl} is available online at \cite{website}. Let us consider an example.

\begin{example}
    \label{ex: A1 x C3 Omega}
    We continue \Cref{ex: A1 x C3 GW} concerning $\CYfi = \Tot_{\bP^1}\!\big(\cO(-2)\big)\times \Aaff{3}$.
    The function \texttt{get\_Omega\_beta} internally calls \texttt{get\_GW\_beta} to compute the required terms of $\GWarg{Z}{\gpT}{\beta/k}$.
    It then uses equation \eqref{eq: generalised GV} to compute $\Omega_\beta$ up to the desired power of $u$.
    The following code continues to use the variables set in \Cref{ex: A1 x C3 GW}.
    \begin{shaded}
    \begin{MyVerb}
\textcolor{code}{julia>} Omega = get_Omega_beta(G, [beta, 2*beta, 3*beta], gMax)
\textcolor{code}{julia>} Omega[beta]
(1//12*t1*t4*t5*u^2 + 1//12*t1*t4*t6*u^2 + 1//12*t1*t5*t6*u^2 + t1 + 1//12*t2*t4*t5*u^2 + 1//12*t2*t4*t6*u^2 + 1//12*t2*t5*t6*u^2 + t2 + 1//12*t3*t4*t5*u^2 + 1//12*t3*t4*t6*u^2 + 1//12*t3*t5*t6*u^2 + t3)//(t4*t5*t6*u^2)
\textcolor{code}{julia>} Omega[2*beta]
0
\textcolor{code}{julia>} Omega[3*beta]
0
    \end{MyVerb}
    \end{shaded}
    That is, we have computed $\Omega_{[\bP^1]}$, $\Omega_{2[\bP^1]}$ and $\Omega_{3[\bP^1]}$ up to genus $1$.
    In fact, the result matches the following formula, which is a special case of \Cref{conj: strip GW formula}
    \begin{equation*}
        \Omega_{d [\bP^1]} = -\delta_{d,1} \, \frac{2 \sinh \frac{u(\epsilon_3 + \epsilon_4 + \epsilon_5)}{2}}{\prod_{i=3}^5 2 \sinh \frac{u\epsilon_i}{2}}.
    \end{equation*}
    We have therefore \emph{verified the conjecture up to genus $g=1$} in this particular example.
    Our numerical experiments on the Gromov--Witten side will follow the same structure, with the only difference that \texttt{GW\_CY5.jl} supports automatic comparison to the conjectural values of $\Omega_\beta$.
\end{example}


\subsection{Sheaves}
In this section we consider the situation where $(\CYfi,\gpT)$ is an equivariantly Calabi--Yau fivefold together with the action of a one-dimensional subtorus $\Gm_q \subseteq \gpT$ whose fixed locus
\begin{equation*}
     \CYth = \CYfi^{\Gm_q}
\end{equation*}
is a threefold. In \cite{NO14:membranes} Nekrasov and Okounkov conjecture that the K-theoretic Donaldson--Thomas invariants of $\CYth$ with a suitable insertion depending on the normal bundle $N_{\CYth} \CYfi$ equate to the \memind\ of $\CYfi$. We begin by recalling their setup. Here, we are going to restrict ourselves to the situation where $\CYth$ is connected. (See \cite[Sec.~3.2 \& 5.5]{NO14:membranes} and \cite{dZNPZ22:PlayingMth} for a discussion of possibly disconnected threefolds.) This means, $\CYfi$ must be a total space
\begin{equation}
    \label{eq: DT CY5}
    \CYfi = \Tot{}_{\CYth} \,\cL_4 \oplus \cL_5\,.
\end{equation}
of two line bundles $\cL_4$ and $\cL_5$ on $\CYth$. We assume that the torus \smash{$\Gm_q$} acts with positive weight $q$ on $\cL_4$ and negative weight $q^{-1}$ on $\cL_5$.

\subsubsection{Pandharipande--Thomas invariants}
\label{sec: PT setup}
Let us denote by $P_n(\CYth,\beta)$ the moduli space of stable pairs on a threefold $\CYth$ with support $\beta$ and Euler characteristic $n$. This moduli space admits a universal stable pair $(\bF,s)$ on the product $P_n(X,\beta)\times X$. We denote by $\pi$ and $\rho$ the projections to the first respectively second factor.

Now suppose $\CYfi$ is as in \eqref{eq: DT CY5} and we chose an isomorphism $\gpT \cong \Gm_q \times \gpT'$. Then following \cite[Def.~3.1]{NO14:membranes} we consider the Euler characteristic of the virtual structure sheaf of the moduli space of stable pairs twisted by a certain square root:
\begin{equation}
    \widetilde{\cO}^{\vir} = \cO^{\vir} \otimes \left( K_{\vir} \otimes \det \big(\mathbf{R}\pi_{*} \,\bF \otimes \rho^*(\cL_4-\cL_5)\big) \right)^{1/2}\,.
\end{equation}
In the following we will often call the twist by the above determinant a Nekrasov--Okoukov insertion. 

We form the generating series of these K-theoretic Pandharipande--Thomas invariants:
\begin{equation}
\label{eq: defn PT series}
\begin{split}
    \PTarg{\beta}{\CYth}{\CYfi}{\Gm_q}{\gpT'} & \coloneqq \sum_{n\in \bZ} (-1)^{n+(\cL_4,\beta)}q^{n+\frac{K_X\cdot\beta}{2}} ~\chi_{\gpT'}\left(P_n(X,\beta),\widetilde{\cO}^{\vir}\right) \\
    \PTarg{}{\CYth}{\CYfi}{\Gm_q}{\gpT'} & \coloneqq 1+ \sum_{\beta \neq 0} Q^\beta ~ \PTarg{\beta}{\CYth}{\CYfi}{\Gm_q}{\gpT'} \,.
\end{split}
\end{equation}

\subsubsection{K-theoretic maps/pairs correspondence}
By definition the generating series of K-theoretic PT invariants is a formal Laurent series in $q^{1/2}$. However conjecturally, it is the Laurent expansion of a rational function in $q^{1/2}$ with constrained poles.

\begin{conj}
    \label{conj: PT GV}
    \cite[Conj.~2.1]{NO14:membranes} Suppose we are in the situation \eqref{eq: DT CY5}. Then there exist rational functions
    \begin{equation*}
        \Omega_{\beta} \in \bZ\big[\tfrac{1}{2}\big]\left[ q^{\pm 1/2},q_2^{\pm 1/2} \ldots, q_{r}^{\pm 1/2} , \, \left\{\frac{1}{1-q^{n_1/2}\prod_i q_i^{n_i/2}}\right\}_{\boldsymbol{n}\in \bZ^{m}\setminus \{0\}} \,\right]
    \end{equation*}
    labelled by effective curve classes $\beta$ in $\CYth$ such that
    \begin{equation}
        \label{eq: PT rational lift}
        \PTarg{}{\CYth}{\CYfi}{\Gm_q}{\gpT'} = \Exp \sum_{\beta \neq 0} Q^{\beta}~ \Omega_{\beta}\,.
    \end{equation}
\end{conj}
In this \namecref{conj: PT GV}, $\Exp$ denotes the plethystic exponential which acts on a formal series $f(q_i,q,Q)$ with $f(q_i,q,0)=0$ as
\begin{equation*}
    \Exp f\big(q_i,q,Q\big) = \exp\left(\sum_{k>0} \frac{1}{k}\, f\big(q_i^k,q^k,Q^k\big)\right)\,.
\end{equation*}

It is no coincidence that we used the same notation for the rational lifts in \Cref{conj: gen GV,conj: PT GV}: It is the central speculation of the seminal paper of Nekrasov and Okounkov that the K-theoretic \PTs\ equates to the index of M2-branes on $\CYfi$ \cite{NO14:membranes}. Thus, performing a change of variables
\begin{equation}
    \label{eq: GW PT change of vars}
    q = \re^{u \epsilon} \qquad \text{and} \qquad q_i = \re^{u \epsilon_i}
\end{equation}
and expanding around $u=0$ we should expect a match with the \GWs\ of $\CYfi$ by \Cref{conj: gen GV}. Here, we wrote \smash{$\epsilon=c_1^{\Gm}(q)$} for the generator of $H_{\Gm_q}^*(\mr{pt})$.

\begin{conj}[\GWPT]
    \label{conj: GW PT}
    \cite[Conj.~7.19]{BS24:refGW} Suppose we are in situation \eqref{eq: DT CY5}. Then under the lift described in \Cref{conj: PT GV} and under the change of variables \eqref{eq: GW PT change of vars} we have
    \begin{equation*}
        \PTarg{}{\CYth}{\CYfi}{\Gm_q}{\gpT'} = \exp \, \GWarg{\CYfi}{\Gm_q \times \gpT'}{}\,.
    \end{equation*}
\end{conj}

\subsubsection{Consequence I: Hidden Symmetries}
\Cref{conj: GW PT} has several consequences which may be investigated in their own right. For instance, note that the \GWs\ does not depend on the choice of a one-dimensional subtorus $\Gm_q \subseteq \gpT$ up to an identification of equivariant variables. Hence, the same must be true for the other side of the correspondence.

\begin{conj}
    \label{conj: PT sym}
    After analytic continuation in $q$, a suitable identification of equivariant variables and possibly restricting the curve class variables, the series
    \begin{equation*}
        \PTarg{}{\CYth}{\CYfi}{\Gm_q}{\gpT'}
    \end{equation*}
    is independent of the choice $\Gm_q \subset \gpT$. 
\end{conj}

\begin{rmk}
    If $\CYth$ is a Calabi--Yau threefold, the $\gpT'$-action on $\CYth$ is Calabi--Yau and $\cL_4 = \cL_5 = \cO_{\CYth}$ then the \GWPT\ reduces to the original GW/PT correspondence \Cref{conj: GW PT} in cohomology \cite[Conj.~3.3]{PT09:StablePairs}. See \cite[Sec.~7.2.3]{BS24:refGW}. In this case \Cref{conj: PT sym} coincides with the well-known $q \mapsto q^{-1}$ symmetry of the \PTs\ in cohomology \cite{Tod10:PTsym,Bri11:HallAlgCurveCountInv}.
\end{rmk}

\begin{rmk}
    In the context of local curves, we remark that this unexpected symmetry is related to 3d mirror symmetry of the Hilbert scheme of points on $\Aaff{2}$ as explained in \cite[Sec.~5.5]{Oko19:Takagi}.
\end{rmk}

We will numerically investigate \Cref{conj: PT sym} for a local curve in \Cref{ex: running 3,ex: running 4}. We will study the symmetry in several other examples in \Cref{sec: local proj spaces,sec: CY4 vanishing} and \Cref{rmk: sym PT strip}.

\subsubsection{Consequence II: Rigidity}
Another consequence of \Cref{conj: GW PT} is \textit{rigidity} of the \GWs\ as we are going to recall now. Suppose the fivefold is a product of a Calabi--Yau threefold and the affine plane:
\begin{equation*}
    \CYfi = \CYth \times \Aaff{2}\,.
\end{equation*}
On the sheaf side, it was proven in \cite[Prop.~7.5]{NO14:membranes} that whenever the moduli space of stable pairs $P_n(X,\beta)$ is proper, the Pandharipande--Thomas invariant
\begin{equation*}
    \chi_{\gpT'}\left(P_n(X,\beta),\cO^{\vir}\otimes K_{\vir}^{1/2}\right)
\end{equation*}
only depends on a single equivariant variable $\kappa$ (or better its square root): the character of the induced $\gpT'$ action on the holomorphic three-form, ie.~$\omega_X \cong \cO_X\otimes \kappa$. This restricted dependency is what is usually called rigidity. Note that by \Cref{conj: GW PT} we should find the same feature on the Gromov--Witten side.

\begin{conj}
    \label{conj: rigidity} \cite[Conj.~2.13]{BS24:refGW}
    Suppose $(X\times \Aaff{2},\gpT)$ is equivariantly Calabi--Yau and $\beta$ is a curve class in $\CYth$ for which $\Mbar_g(X,\beta)$ is proper for all $g\geq 0$. Then
    \begin{equation*}
        \epsilon_4 \epsilon_5 \cdot \GW_{\beta}(X \times \Aaff{2},\gpT)
    \end{equation*}
    is a power series in $u$ with coefficients in $\bQ[\epsilon_4,\epsilon_5]$ where $\epsilon_4$ and $\epsilon_5$ are the tangent $\gpT$-weights of the affine plane.
\end{conj}

The conjecture holds when $\epsilon_4 + \epsilon_5=0$ because of the following dimensional reduction as explained for instance in \cite[Cor.~2.7]{BS24:refGW}:
    \begin{equation}
        \label{eq: GW 5 to 3 reduction strip}
        \GWarg{X\times\Aaff{2}}{\gpT}{} = \GWarg{X }{\gpT}{} \, \Big|_{u\mapsto \ri u \epsilon_4}\,.
    \end{equation}
A proof of the \namecref{conj: rigidity} beyond this limit is desirable as it would provide great leverage for the calculation of Gromov--Witten invariants of toric Calabi--Yau fivefolds. We will present numerical evidence for the \namecref{conj: rigidity} beyond the limit $\epsilon_4 + \epsilon_5 = 0$ in \Cref{ex: chain of O(-1)s} and \Cref{ev: gauge GW correspondence}.

\subsubsection{The K-theoretic vertex}
\label{sec: Kth vertex}
Now suppose that $\CYth$ is toric and $\gpT'$ acts on $\CYth$ through its dense open torus in a way that $(\CYth,\gpT')$ is a GKM space. In this setting, virtual localisation decomposes the \PTs\ as a weighted sum indexed by labellings of the edges of the GKM graph by integer partitions. We will recall the formula in this section.

Following \cite[Sec.~2]{PT09:vertex}, to each $\gpT'$-fixed stable pair $[\cO_X \rightarrow F] \in P_n(X,\beta)^{\gpT'}$ we can associate a collection of partitions $\boldsymbol{\mu}=(\mu_h)_{h\in H(G)}$ indexed by the half-edges of the GKM graph $G$ of $\CYth$. These partitions record the asymptotics of the scheme theoretic support of the restriction of $F$ to an affine chart around each torus fixed point. For this assignment we fix a cyclic order of the three half-edges adjacent to every vertex in a way that $\mu_h = \mu_{h'}^{\mr{t}}$ whenever two half-edges $h$ and $h'$ belong to the same edge. Via a \v{C}ech complex calculation, as an element in K-theory, the virtual normal bundle of the $\gpT'$-fixed locus can be decomposed into contributions labelled by edges and vertices of the GKM graph \cite[Thm.~2]{PT09:vertex}. In the situation where $\cL_4=\cL_5=\cO_{X}$ this decomposition leads to the following formula.

\begin{thm}
\label{thm: PT vertex formula no ins}
If $\cL_4=\cL_5=\cO_X$, we have
    \begin{equation}
    \label{eq: PT vertex formula no ins}
          \PTarg{}{\CYth}{X\times \Aaff{2}}{\Gm_q}{\gpT'} = \sum_{\boldsymbol{\mu}} \prod_{e=(h,h')}Q^{|\mu_{h}|[C_e]} \, \hat{E}_{\mu_{h}}\big(\boldsymbol{a}_h;\boldsymbol{q}_h;q\big) ~ \prod_{v} \hat{V}^{\PT}_{\boldsymbol{\mu}_v}\big(\boldsymbol{q}_v;q\big)\,.
    \end{equation}
\end{thm}

Here, $\hat{E}$ and $\hat{V}^{\PT}$ denote the so-called K-theoretic edge and vertex weights. An explicit formula for the edge weight can be found for instance in \cite[Sec.~4.2]{Arb21:KthDTHilbSurf}. It is evaluated at a partition $\mu_{h}$ labelling a half-edge of $e=(h,h')$ and depends on the degrees $\boldsymbol{a}_h=(a_2,a_3)$ of a splitting of the normal bundle of the $\gpT'$-stable curve $C_e$:
\begin{equation*}
    N_{C_e} \CYth \cong \cO_{\bP^1}(a_2) \oplus \cO_{\bP^1}(a_3)\,.
\end{equation*}
Moreover, it takes as an argument the K-theoretic tangent $\gpT'$-weights $\boldsymbol{q}_h = (q_h, q_1, q_2)$ at the fixed point associated to $h$ where all tuples must respect the cyclic ordering we chose earlier at every fixed point. 

Similarly, the vertex weight \smash{$\hat{V}^{\PT}_{\boldsymbol{\mu}_v}(\boldsymbol{q}_v;q)$} is evaluated at the collection of partitions $\boldsymbol{\mu}_v = (\mu_{h_1},\mu_{h_2},\mu_{h_3})$ incident to a fixed point and at its K-theoretic tangent weights $\boldsymbol{q}_v = (q_{h_1},q_{h_2},q_{h_3})$. We remark that each $q_{h_i}$ is the lift of the cohomology weight $\epsilon_{h_i}$ under the Chern character:
\begin{equation*}
    q_{h_i} = \re^{u \epsilon_{h_i}}\,.
\end{equation*}

See \cite[Eq.~(4.17)]{Arb21:KthDTHilbSurf} for a derivation of an analogous formula to \eqref{eq: PT vertex formula no ins} in the context of Donaldson--Thomas theory. This calculation is based on the analysis of the $\gpT'$-fixed locus of the moduli space of ideal sheaves and its virtual normal bundle by Maulik, Nekrasov, Okounkov and Pandharipande \cite{MNOPI}. Formula \eqref{eq: PT vertex formula no ins} then follows by repeating the same steps in Pandharipande--Thomas theory building on \cite{PT09:vertex} or alternatively from the fact that the K-theoretic Donaldson--Thomas and Pandharipande--Thomas vertex agree up to an explicit factor \cite[Thm.~1.3.1]{KLT26:DTPTvert}.

Both the K-theoretic edge and vertex weight have been implemented as part of the SageMath package \texttt{boxcounting} by Henry Liu \cite{Liu:boxcounting}. All numerical experiments with Pandharipande--Thomas invariants presented in this paper will rely on this code. Before showcasing a concrete example, let us explain how \Cref{thm: PT vertex formula no ins} must be modified in order to allow for a Nekrasov--Okounkov insertion coming from non-trivial ambient line bundles $\cL_4$ and $\cL_5$. It turns out, such an insertion only leads to a shift in $q$ at each vertex and a minor modification of the edge weights.

To describe the modification, note that we obtain a one-dimensional $\gpT'$-representation by restricting $\cL_i$ to a fixed point labelled by a vertex $v$ of the GKM graph. We denote this representation by $q_{v,i}$, $i\in\{4,5\}$, and similarly write $q_{h,i} \coloneqq q_{v,i}$ if $h$ is a half-edge adjacent to $v$. Moreover, for every edge $e=(h,h')$ of the GKM graph let $a_{h,i}$, $i\in\{4,5\}$, denote the degree of the restriction of the line bundle to the associated $\gpT'$-stable $\bP^1$:
\begin{equation*}
    \cL_i \big|_{C_e} \cong \cO(a_{h,i})\,.
\end{equation*}
For every edge $e=(h,h')$ of the GKM graph we introduce the modified edge weight
\begin{equation}
    \label{eq: new PT edge}
    \begin{split}
        \hat{E}^{\NO}_{\mu_h}\!\left(\boldsymbol{a}_h,a_{4},a_{5};\boldsymbol{q}_h,q_{h,4},q_{h,5};q\right) \coloneqq (-1)^{a_{4} |\mu_{h}|}  \, \hat{E}_{\mu_{h}}\!\left(\boldsymbol{a}_h;\boldsymbol{q}_h;q\right) q^{k_0(\mu_h)/2} \prod_{i=1}^5 q_{h,i}^{k_i(\mu_h)/2}\,.
    \end{split}
\end{equation}
With the notation $|\mu|\coloneqq \sum_i \mu_i$ and $n_{\mu} \coloneqq \sum_{(i,j)\in \mu} i$ the exponents in the last equation read
\begin{gather*}
    k_0(\mu) = (a_{4}+a_{5})|\mu|\,, \\
    k_1(\mu) = (a_{5}-a_{4}) \left(\tfrac{a_{4}+a_{5}+1}{2}|\mu| -a_{2} n_{\mu} -a_{3} n_{\mu^{\mr{t}}} \right) \,,\\[0.1em] %
    k_{2}(\mu) = (a_{5}-a_{4}) n_\mu \,, \qquad %
    k_{3}(\mu) = (a_{5}-a_{4}) n_{\mu^{\mr{t}}} \,,\\[0.3em]
    k_{4}(\mu) = (a_{4}+1)|\mu| -a_{2} n_{\mu} -a_{3} n_{\mu^{\mr{t}}} \,, \qquad %
    k_5(\mu) = -(a_{5}+1)|\mu| +a_{2} n_{\mu} +a_{3} n_{\mu^{\mr{t}}} \,.
\end{gather*}
Regarding the definition of $n_{\mu}$ we remark that we adopt the convention where the box shared by every non-empty partition $\mu$ has coordinate $(0,0)$.

\begin{prop}
    \label{prop: PT vertex formula with ins}
    We have
    \begin{equation*}
    \begin{split}
        &\qquad \PTarg{}{\CYth}{\CYfi}{\Gm_q}{\gpT'} = \\
        &\sum_{\boldsymbol{\mu}} \prod_{e=(h,h')}Q^{|\mu_{h}|[C_e]}\,\hat{E}^{\NO}_{\mu_h}\big(\boldsymbol{a}_h,a_{h,4},a_{h,5};\boldsymbol{q}_h,q_{h,4},q_{h,5};q\big) ~ \prod_{v} \hat{V}^{\PT}_{\boldsymbol{\mu}_v}\big(\boldsymbol{q}_v;\,q\,(q_{v,4}/q_{v,5})^{1/2}\big) \,.
    \end{split}
    \end{equation*}
\end{prop}

\begin{proof}
    In \cite[Sec.~3.1]{PT09:vertex} it was shown that a connected component $\mathsf{Q}$ of the $\gpT'$-fixed locus of $P_n(X,\beta)$ is a product $\mathsf{Q} = \prod_{v\in G} \mathsf{Q}_{\,v}$ where each $\mathsf{Q}_{\,v}$ parametrises a so-called labelled box configuration at a fixed point associated with $v$. Each such box configuration has a fixed asymptotic profile recorded by a triple of partitions $\boldsymbol{\mu}_v$. Recall that each term in \eqref{eq: PT vertex formula no ins} corresponds to the contribution of all fixed components with prescribed asymptotic profile recorded by a collection of partitions $\boldsymbol{\mu}$.

    Now let $\alpha\in K_{\gpT'}(X)$. Via a \v{C}ech complex calculation similar to the evaluation of the virtual normal bundle in \cite[Thm.~3]{PT09:vertex}, one can show that the restriction of $\mathbf{R}\pi_* ( \bF \otimes \rho^*\alpha )$ to a fixed component $Q$ decomposes as a K-theory class into
    \begin{equation*}\mathbf{R}\pi_* ( \bF \otimes \rho^*\alpha ) = \sum_{v\in V(G)} \alpha|_v \, A_v+\sum_{e=(h,h')\in E(G)}  \frac{\alpha|_h \, B_{\mu_{h}}(q_{h,2},q_{h,3})}{1-q_{h}^{-1}} + \frac{\alpha|_{h'} \, B_{\mu_{h'}}(q_{h',2},q_{h',3})}{1-q_{h'}^{-1}} 
    \end{equation*}
    where $A_v \in K_{\gpT'}(\mathsf{Q}_{\,v})$, we write $B_{\mu}(q_1,q_2) \coloneqq \sum_{(i,j)\in\mu} q_1^{-i} q_2^{-j}$ and we denote by $\alpha|_{v}$ and $\alpha|_{h}$ the virtual $\gpT'$-representation obtained by restricting $\alpha$ to the fixed point $v$ respectively to the fixed point adjacent to the half-edge $h$. Specialising this formula to the rank-zero class $\alpha = \cL_4 - \cL_5$ and taking the determinant thus yields
    \begin{equation}
    \label{eq: NO insertion distributed}
    \begin{split}
        &\qquad \det\left(\mathbf{R}\pi_* \big( \bF \otimes \rho^*\alpha \big)\right) = \\
        & \prod_{v\in V(G)}\left(\frac{q_{v,4}}{q_{v,5}}\right)^{\chi (A_v)} \!\! \prod_{e\in E(G)}  \!\!\det \left( \frac{ (q_{h,4} - q_{h,5}) \, B_{\mu_{h}}(q_{h,2},q_{h,3})}{1-q_{h}^{-1}} + \frac{(q_{h',4} - q_{h',5}) \, B_{\mu_{h'}}(q_{h',2},q_{h',3})}{1-q_{h'}^{-1}}  \right)\,.
    \end{split}
    \end{equation}
    By a direct calculation one can show that the edge factors simplify to
    \begin{equation*}
        \det \left( \frac{\alpha|_v \, B_{\mu_{v,e}}(q_{h_2},q_{h_ 3})}{1-q_{v,e}^{-1}} + \frac{\alpha|_{v'} \, B_{\mu_{v,e}}(q_{h_3'},q_{h_2'})}{1-q_{v,e}} \right) = \prod_{i=1}^5 q_{h,i}^{k_i(\mu)}
    \end{equation*}
    justifying the modification \eqref{eq: new PT edge}. We only remark that in order to prove the identity one has to use the relations
    \begin{equation*}
        q_{h'} = q_{h}^{-1}\,,\quad q_{h',2} = q_{h,3} q_h^{-a_{h,2}}\,,\quad q_{h',3} =  q_{h,2} q_h^{-a_{h,3}}\,,\quad q_{h',4} = q_{h,4}q_{h}^{-a_{h,4}}\,,\quad q_{h',5} = q_{h,5}q_{h}^{-a_{h,5}}\,.
    \end{equation*}

    Regarding vertex terms, note that the coefficient of $q^n$ in \smash{$\hat{V}^{\PT}_{\boldsymbol{\mu}_v}$} comes from factors $\mathsf{Q}_{\,v}$ of connected components $\mathsf{Q}=\bigsqcup_{v\in V(G)}\mathsf{Q}_{\,v}$ where $\chi (A_v) = n$. Hence, the first factor on the right-hand side of \eqref{eq: NO insertion distributed} can be absorbed by shifting \smash{$q\mapsto q\, (q_{v,4}/q_{v,5})^{1/2}$} in the vertex weight.
\end{proof}

\subsubsection{Numerical implementation}
\label{sec: PT numerics}
\Cref{prop: PT vertex formula with ins} together with Liu's implementation of the K-theoretic vertex and edge weight \cite{Liu:boxcounting} allows us to compute the \PTs\ of a toric threefold in a fivefold in low degree $\beta$ to the first orders in $q$. We illustrate this with our running example. For more details consult our documentation \cite{website}.

\begin{example}
\label{ex: running 3}
We continue our analysis of $Z= \Tot{}_{\bP^1} \, \cO(-2) \oplus \cO^{\oplus 3}$. First, we consider the \PTs\ with respect to the one-dimensional subtorus
\begin{equation*}
\begin{tikzcd}[row sep=0em]
    \Gm_q \times \gpT'_1 \ar[r] & \Gm[6]\\
    \big(q,(q_0,q_1,q_2,q_3)\big) \ar[r, maps to] & \big(q_0,q_1,q_2,q_3,q,(q_0 q_1 q_2 q_3q)^{-1}\big)\,.
\end{tikzcd}
\end{equation*}
With this choice we have
\begin{equation*}
    X_1 \coloneqq Z^{\Gm_q} = \Tot{}_{\bP^1} \, \cO(-2) \oplus \cO
    \,, \qquad \cL_4 = \cL_5 = \cO_{X_1}\,.
\end{equation*}

The coeffiecients of the \PTs\ will be rational functions in the variables $q_0,q_1,q_2,q_3$. After loading the package using
\begin{shaded}
\begin{MyVerb}
\textcolor{code}{sage>} load(PT_CY5.sage)
\end{MyVerb}
\end{shaded}
we set up the ring for our calculation as follows:
\begin{shaded}
\begin{MyVerb}
\textcolor{code}{sage>} setup_ring('q0,q1,q2,q3')
\end{MyVerb}
\end{shaded}
Now for example, if the unique edge of the GKM graph of $X_1$ is decorated with the partition $(1)$ its edge weight can be determined via
\begin{shaded}
\begin{MyVerb}
\textcolor{code}{sage>} E([1], [-2, 0, 0, 0], [q1/q0, q2*q0^2, q3, 1, 1/(q0*q1*q2*q3)])
((q0^2*q1^2*q2^2*q3^2 - q3^2)/(-q3^2 + 1))*q^2
\end{MyVerb}
\end{shaded}
Here, we use the same convention for the tangent weights as illustrated in \Cref{fig: A1 x C3 GKM graph} where we identify $q_i=\re^{u \epsilon_i}$. However, we should stress that in the output of function \texttt{E} all variables $q,q_0,\ldots,q_3$ are squared. Hence, the above output has to be understood as
\begin{equation*}
    \hat{E}^{\NO}_{(1)}\big(-2,0,0,0;q_1 q_0^{-1},q_2 q_0^2, q_3,1,(q_0 q_1 q_2 q_3)^{-1};q\big) = \frac{q_0 q_1 q_2 q_3 - q_3}{1 - q_3}\,q\,.
\end{equation*}
Similarly, the vertex weight at the fixed point $0\in\bP^1$ can be evaluated via
\begin{shaded}
\begin{MyVerb}
\textcolor{code}{sage>} V([[1], [], []], [q1/q0, q2*q0^2, q3, 1, 1/(q0*q1*q2*q3)], 3)
1 + ((q0^4*q1^2*q2^2*q3^2 - q1^2)/(q0^2 - q1^2))*q^2 + ((q0^10*q1^4*q2^4*q3^4 - q0^6*q1^4*q2^2*q3^2 - q0^4*q1^6*q2^2*q3^2 + q1^6)/(q0^6 - q0^4*q1^2 - q0^2*q1^4 + q1^6))*q^4 + O(q^6)
\end{MyVerb}
\end{shaded}
Here, we use the same convention for the $\gpT'$-weights as in \Cref{ex: A1 x C3 GKM graph}.

We may evaluate the coefficients of the \PTs\ in  curve class $\beta = d[\bP^1]$ by summing over all partitions of $d$. For instance for $d=1$ we find:
\begin{shaded}
\begin{MyVerb}
\textcolor{code}{sage>} (
\textcolor{code}{sage>}   V([[1], [], []], [q1/q0, q2*q0^2, q3, 1, 1/(q0*q1*q2*q3)], 3)
\textcolor{code}{sage>}   * E([1], [-2, 0, 0, 0], [q1/q0, q2*q0^2, q3, 1, 1/(q0*q1*q2*q3)])
\textcolor{code}{sage>}   * V([[1], [], []], [q0/q1, q2*q1^2, q3, 1, 1/(q0*q1*q2*q3)], 3)
\textcolor{code}{sage>} )
((q0^2*q1^2*q2^2*q3^2 - q3^2)/(-q3^2 + 1))*q^2 + ((q0^4*q1^4*q2^4*q3^4 - q0^2*q1^2*q2^2*q3^4 + q0^2*q1^2*q2^2*q3^2 - q3^2)/(-q3^2 + 1))*q^4 + ((-q0^6*q1^6*q2^6*q3^6 + q0^4*q1^4*q2^4*q3^6 - q0^4*q1^4*q2^4*q3^4 + q0^2*q1^2*q2^2*q3^4 - q0^2*q1^2*q2^2*q3^2 + q3^2)/(q3^2 - 1))*q^6 + O(q^8)
\end{MyVerb}
\end{shaded}
Consistent with \Cref{conj: PT GV} these are the first terms of a Laurent expansion of a rational function in $q$:
\begin{equation}
    \label{eq: formula PT d1 KP1xC3}
    \PTarg{[\bP^1]}{X_1}{Z}{\Gm_q}{\gpT_1'} =  -\frac{[(q_0 q_1 q_2)^{-1}]}{[q_3] \, [q] \, [(q_0 q_1 q_2 q_3q)^{-1}]} +O\big(q^4\big)\,.
\end{equation}
Here, we write $[q]\coloneqq q^{1/2} - q^{-1/2}$. We remark that in \cite[Cor.~5.1.19]{Ok15:lectures} Okounkov indeed proves that the above formula holds to any power in $q$.

Now note that under the change of variables
\begin{equation*}
    q_0=\re^{u \epsilon_0} \,, \qquad q_1=\re^{u \epsilon_1}  \,, \qquad q_2=\re^{u \epsilon_2} \,, \qquad q_3=\re^{u \epsilon_3} \,, \qquad q=\re^{u \epsilon_4}
\end{equation*}
formula \eqref{eq: formula PT d1 KP1xC3} perfectly agrees with our earlier calculation of Gromov--Witten invariants in \Cref{ex: A1 x C3 Omega}. This provides some first non-trivial evidence for \Cref{conj: GW PT}. Later in \Cref{prop: strip PT formula}, we will generalise formula \eqref{eq: formula PT d1 KP1xC3} from degree one to arbitrary degree.
\end{example}

\begin{example}
\label{ex: running 4}
Note that in the last example we could have as well made the choice
\begin{equation*}
\begin{tikzcd}[row sep=0em]
    \Gm_{q'} \times \gpT'_1 \ar[r] & \Gm[6]\\
    \big(q,(q_0,q_1,q_4,q_5)\big) \ar[r, maps to] & \big(q_0,q_1,q',(q_0 q_1 q_4 q_5q')^{-1},q_4,q_5\big)\,.
\end{tikzcd}
\end{equation*}
In this case the $\Gm_{q'}$-fixed locus of $\CYfi$ is
\begin{equation*}
    X_2 = \bP^1 \times \Aaff{2} 
    \qquad \text{and} \qquad \cL_4 = p^*\cO_{\bP^1}(-2)\,, \qquad \cL_5 = \cO_{X_2}
\end{equation*}
where $p$ is the projection $X_2 \rightarrow \bP^1$.
Note that \Cref{conj: PT sym} predicts an equality
\begin{equation}
    \label{eq: A1 x C3 PT sym}
    \PTarg{}{X_2}{Z}{\Gm_{q'}}{\gpT_2'} = \PTarg{}{X_1}{Z}{\Gm_q}{\gpT_1'}
\end{equation}
under a suitable change of variables.

Let us check this claim and evaluate the Pandharipande--Thomas invariants in curve class degree one:
\begin{shaded}
\begin{MyVerb}
\textcolor{code}{sage>} setup_ring('q0,q1,q4,q5')
\textcolor{code}{sage>} (
\textcolor{code}{sage>}   V([[1], [], []], [q1/q0, q4, q5, q0^2, 1/(q0*q1*q4*q5)], 3)
\textcolor{code}{sage>}   * E([1], [0, 0, -2, 0], [q1/q0, q4, q5, q0^2, 1/(q0*q1*q4*q5)])
\textcolor{code}{sage>}   * V([[1], [], []], [q0/q1, q4, q5, q1^2, 1/(q0*q1*q4*q5)], 3)
\textcolor{code}{sage>} )
((-q4^2*q5^2)/(q4^2*q5^2 - q4^2 - q5^2 + 1)) + ((q0^2*q1^2*q4^4*q5^4 - q0^2*q1^2*q4^2*q5^2)/(-q4^2*q5^2 + q4^2 + q5^2 - 1))*q^2 + ((q0^4*q1^4*q4^6*q5^6 - q0^4*q1^4*q4^4*q5^4)/(-q4^2*q5^2 + q4^2 + q5^2 - 1))*q^4 + O(q^6)
\end{MyVerb}
\end{shaded}
Indeed, these are the first terms of the Laurent expansion of the following rational function:
\begin{equation*}
    \PTarg{[\bP^1]}{X_1}{Z}{\Gm_{q'}}{\gpT_1'} = -\frac{[(q_0 q_1q')^{-1}]}{[(q_0 q_1 q_4 q_5q')^{-1}] \, [q_4] \, [q_5]} +O\big(q'^3\big)
\end{equation*}
Again, this matches our Gromov--Witten calculation in \Cref{ex: A1 x C3 Omega} consistently with \Cref{conj: GW PT} under the change of variables
\begin{equation*}
    q_0=\re^{u \epsilon_0} \,, \qquad q_1=\re^{u \epsilon_1}  \,, \qquad q_4=\re^{u \epsilon_4} \,, \qquad q_5=\re^{u \epsilon_5} \,, \qquad q'=\re^{u \epsilon_2}
\end{equation*}
which indeed confirms the symmetry \eqref{eq: A1 x C3 PT sym}. 
\end{example}

\section{Strip geometries}
\label{sec: strips}
The first class of geometries for which we want to investigate the web of correspondences between enumerative invariants (\Cref{fig: summary paper}) are so-called strip geometries.

\subsection{Geometric setup}
\label{sec: strips setup}
By a strip geometry we understand a toric Calabi--Yau threefold $\CYth$ whose fan is the cone over a triangulated rectangle placed at height one:
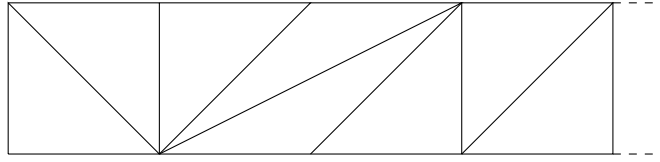
\begin{equation*}
\begin{tikzpicture}
    \draw (0,0) -- (2,0) -- (0,2) -- (0,0);
    \draw (2,0) -- (2,2) -- (0,2);
    \draw (2,0) -- (4,2) -- (2,2);
    \draw (2,0) -- (6,2) -- (4,2);
    \draw (2,0) -- (4,0) -- (6,2);
    \draw (4,0) -- (6,0) -- (6,2);
    \draw (6,0) -- (8,2) -- (6,2);
    \draw (6,0) -- (8,0) -- (8,2);
    \draw[dashed] (8,0) -- (8.7,0);
    \draw[dashed] (8,2) -- (8.7,2);
\end{tikzpicture}
\end{equation*}
We require that each triangle has a horizontal base of length one. We remark that all faces of the fan associated with segments at the top (respectively bottom) of the strip correspond to torus stable $\Aone$s acted on with the same torus weight which we denote by $\epsilon_2$ (respectively $\epsilon_3$). We promote such a threefold $\CYth$ to a fivefold by taking the product with the affine plane:
\begin{equation*}
    \CYfi = \CYth \times \Aaff{2}\,.
\end{equation*}

For strip geometries $\CYth$ the projection of all one-dimensional torus orbit closures under the moment map associated to the two-dimensional Calabi--Yau torus takes the following shape:
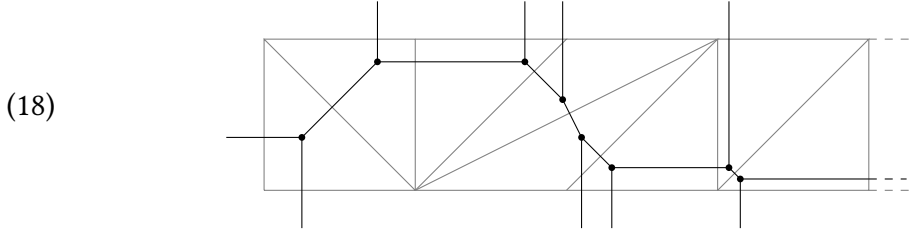
\begin{equation}
\label{eq: strip toric graph}
\begin{tikzpicture}[baseline={(current bounding box.center)}]
    \draw[gray] (0,0) -- (2,0) -- (0,2) -- (0,0);
    \draw[gray] (2,0) -- (2,2) -- (0,2);
    \draw[gray] (2,0) -- (4,2) -- (2,2);
    \draw[gray] (2,0) -- (6,2) -- (4,2);
    \draw[gray] (2,0) -- (4,0) -- (6,2);
    \draw[gray] (4,0) -- (6,0) -- (6,2);
    \draw[gray] (6,0) -- (8,2) -- (6,2);
    \draw[gray] (6,0) -- (8,0) -- (8,2);
    \draw[gray, dashed] (8,0) -- (8.7,0);
    \draw[gray, dashed] (8,2) -- (8.7,2);
    \foreach \pos in {{0.5,0.7},{1.5,1.7},{3.45,1.7},{3.95,1.2},{4.2,0.7},{4.6,0.3},{6.15,0.3},{6.3,0.15}}{
		\node[circle,inner sep=0.9pt,fill=black] at (\pos){};
	}
    \draw (-0.5,0.7) -- (0.5,0.7) -- (1.5,1.7) -- (3.45,1.7) -- (3.95,1.2) -- (4.2,0.7) -- (4.6,0.3) -- (6.15,0.3) -- (6.3,0.15) -- (8,0.15);
    \draw[dashed] (8,0.15) -- (8.5,0.15);
    \draw (0.5,0.7) -- (0.5,-0.5);
    \draw (1.5,1.7) -- (1.5,2.5);
    \draw (3.45,1.7) -- (3.45,2.5);
    \draw (3.95,1.2) -- (3.95,2.5);
    \draw (4.2,0.7) -- (4.2,-0.5);
    \draw (4.6,0.3) -- (4.6,-0.5);
    \draw (6.15,0.3) -- (6.15,2.5);
    \draw (6.3,0.15) -- (6.3,-0.5);
\end{tikzpicture}
\end{equation}
The normal bundle of every torus stable $\bP^1$ splits as either
\begin{equation*}
    \cO_{\bP^1}(-1) \oplus \cO_{\bP^1}(-1) \qquad \text{or} \qquad \cO_{\bP^1}(-2) \oplus \cO_{\bP^1} \,.
\end{equation*}
Our previous \Cref{ex: A1 x C3 GKM graph} falls into this class of examples with a single $\gpT'$-stable $\bP^1$ of the second type.

\subsection{The PT series}
For strip geometries the \PTs\ can be determined in closed-form via the capped two-leg vertex studied in \cite{KOO21:twoLegPT}. To state the formula, however, we first need to introduce more notation: We label the torus fixed points from left to right by $v_1,\ldots,v_{N+1}$ and denote by $\epsilon_{v_i}\in\{\epsilon_2,\epsilon_3\}$ the tangent weight of the vertical torus stable $\Aone$ adjacent to $v_i$, $i\in\{1,\ldots,N+1\}$. As always we denote the lift of this weight to K-theory by $q_{v_i} = \re^{u\epsilon_{v_i}}$. Finally, let us write $Q_i = Q^{[C_{i}]}$ for the variable that records the curve degree of the line $C_{i}$ connecting $v_{i}$ and $v_{i+1}$.

Before we can state a formula for the \PTs , recall that the definition of this series requires the choice of a \smash{$\Gm_q$}-action on $\CYfi = \CYth \times \Aaff{2}$ with three-dimensional fixed locus. We do the obvious choice where \smash{$\Gm_q$} acts on the two affine directions with opposite characters $q$ and $q^{-1}$ while fixing $\CYth$. Further, let $\gpT'=\Gm[3]$ be the dense open torus of $\CYth$. We extend the action of $\gpT'$ to the fivefold so that the tangent space of the affine plane decomposes as a $\Gm_q\times\gpT'$-representation into
\begin{equation*}
    \kappa^{1/2}q  \qquad \text{and} \qquad \kappa^{1/2}q^{-1} \,.
\end{equation*}
Here, $\kappa$ denotes the character by which $\gpT'$ scales the holomorphic three-form of $\CYth$. This choice makes $(\CYfi, \Gm_q\times\gpT')$ equivariantly Calabi--Yau.

\begin{prop}
    \label{prop: strip PT formula}
    We have
    \begin{equation}
	\label{eq: strip PT formula}
		\PTarg{}{\CYth}{\CYfi}{\Gm_q}{\gpT'} = \Exp \left(-\sum_{1 \leq m < n \leq N+1} \frac{[\kappa q_{v_m}]}{[q_{v_n}] \, [\kappa^{1/2}q] \, [\kappa^{1/2}q^{-1}]} \prod_{i=m}^{n-1} Q_i \right)
	\end{equation}
    where we use the notation $[q]\coloneqq q^{1/2} - q^{-1/2}$.
\end{prop}

We defer the proof of \Cref{prop: strip PT formula} to \Cref{sec: strip PT formula proof}.

\begin{rmk}
    Note that in \Cref{ex: running 3} we numerically tested the above formula when $N=1$, that is for a single torus stable $\bP^1$ with $q_{v_1} = q_{v_2}=q_2$.
\end{rmk}

\begin{rmk}
    \label{rmk: sym PT strip}
    Note that formula \eqref{eq: strip PT formula} exhibits all discrete symmetries of the ambient fivefold $\CYfi$. This is clear evidence for \Cref{conj: PT sym}.
\end{rmk}


\subsection{The GW series}
Before proving \Cref{prop: strip PT formula} let us explore its consequences. Note that together with the \GWPT\ (\Cref{conj: GW PT}) the \namecref{prop: strip PT formula} predicts the following closed-form expression for the \GWs\ of $X\times\Aaff{2}$.

\begin{conj}
    \label{conj: strip GW formula}
    Suppose $\CYth$ is a strip geometry and $(X\times\Aaff{2},\gpT)$ is equivariantly Calabi--Yau. Then if we denote the tangent $\gpT$-weights of the affine plane by $\epsilon_4,\epsilon_5$ we have
    \begin{equation*}
    \label{eq: strip GW formula}
    \GWarg{X\times\Aaff{2}}{\gpT}{} = \Exp \left(-\sum_{1 \leq m < n \leq N+1} \frac{[\kappa q_{v_m}]}{[q_{v_n}] \, [q_4] \, [q_5]} \prod_{i=m}^{n-1} Q_i \right)
	\end{equation*}
    under the change of variables $q_i = \re^{u \epsilon_i}$.
\end{conj}

We can prove this formula for anti-diagonal torus actions.

\begin{prop}
    \label{prop: strip GW formula limit}
    \Cref{conj: strip GW formula} holds when $q_i q_j = 1$ for some $i,j\in\{2,3,4,5\}$.
\end{prop}

\begin{proof}
    When $\{i,j\}=\{4,5\}$ the fivefold invariants reduce to the threefold invariants of $\CYth$ by relation \eqref{eq: GW 5 to 3 reduction strip}. Via the two-leg topological vertex \cite{LLLZ09:TV} and the methods developed in \cite[Sec.~2.2]{IKP06:strip} the resulting threefold generating series can be determined as
    \begin{equation*}
    	\exp \, \GWarg{X\times\Aaff{2}}{\gpT}{} = \Exp \left(\sum_{1 \leq m < n \leq N+1} \frac{(-1)^{\delta_{\epsilon_{v_m},\epsilon_{v_n}}}}{\big(2 \sin \tfrac{u}{2}\big)^2} \prod_{i=m}^{n-1} Q_i\right)\,.
    \end{equation*}
    This is in agreement with the formula stated in \Cref{conj: strip GW formula}. By the Calabi--Yau condition on the torus weights it remains to check the formula when $i\in\{2,3\}$ and $j\in \{4,5\}$. In this case the formula claimed has been established in \cite[Eq.~(16)]{Sch26:CY5TV} using that a dimensional reduction similar to \eqref{eq: GW 5 to 3 reduction strip} appears locally at every vertex contributing to the fixed point formula \eqref{eq: loc formula GW}.
\end{proof}

The key to prove \Cref{conj: strip GW formula} for anti-diagonal torus actions is that every quintuple Hodge integral occurring in the fixed point formula \eqref{eq: loc formula GW} reduces to a triple Hodge integral for which explicit formulae are known. Beyond anti-diagonal torus actions such a reduction does not occur. This means one has to embrace honest quintuple Hodge integrals which at the moment we cannot control sufficiently.

However, we can still provide numerical evidence for \Cref{conj: strip GW formula} beyond anti-diagonal torus actions, which includes the following two examples.

\begin{example}\label{ex: Ar x C}
Let $\cA_N$ be a minimal crepant resolution of the $A_N$-Du Val surface singularity. The exceptional locus of the resolution consists of a chain of $N$ rational curves and $X=\cA_N \times \Aone$ falls into the class of strip geometries. We display the GKM graph of this threefold in \Cref{fig: GKM graphs strip examples} (A).
As indicated we denote the tangent $\gpT$-weight at the origin of $\Aone$ by $\epsilon_3$.
To obtain the picture from \eqref{eq: strip toric graph}, we simply choose the triangulation of the rectangular strip so that all resulting vertical half-edges point down.

For this geometry, \Cref{conj: strip GW formula} predicts the following formula, which we can test using \texttt{GW\_CY5.jl}:
\begin{equation*}
    \exp \GWarg{\cA_N \times \Aaff{3}}{\gpT}{} = \Exp \left(-\frac{[q_3q_4q_5]}{[q_3][q_4][q_5]}\sum_{1\leq m \leq n\leq N} \prod_{k=m}^n Q_k \right)\,.
\end{equation*}
For $N=1$ this formula was first conjectured in \cite[Conj.~3.4]{BS24:refGW}.
\end{example}

\begin{figure}
    \centering
    \begin{subfigure}[t]{0.47\textwidth}
        \centering
        \begin{tikzpicture}[xscale=1.8]
        \foreach \pos in {{2,0},{3,1},{3.5,2},{3.83,3}}{
            \node[circle,inner sep=1.5pt,fill=black] at (\pos){};
        }
        \draw (1.2,0) -- (2,0) -- (3,1) -- (3.5,2) -- (3.61,2.33);
        \draw[dotted] (3.61,2.33) -- (3.72,2.66);
        \draw (3.72,2.66) -- (3.83,3) -- (3.955,3.5);
        \draw  (2,0) -- (2,-1);
        \draw  (3,1) -- (3,-1);
        \draw  (3.5,2) -- (3.5,-1);
        \draw  (3.83,3) -- (3.83,-1);
        \node[below] at (2.2,0) {\scriptsize $\epsilon_3$};
        \node[below] at (1.5,0.5) {\scriptsize $\epsilon_1$};
        \node[below] at (3.2,1) {\scriptsize $\epsilon_3$};
        \node[below] at (4.03,3) {\scriptsize $\epsilon_3$};
        \end{tikzpicture}
        \caption{$\cA_N \times \Aaff{3}$.}
        \label{fig: Ar x C}
    \end{subfigure}
    \hfill
    \begin{subfigure}[t]{0.47\textwidth}
        \centering
        \begin{tikzpicture}[scale=0.9]
        \foreach \pos in {{1,0},{2,1},{3.4,1},{4.4,2},{5.8,2}}{
            \node[circle,inner sep=1.5pt,fill=black] at (\pos){};
        }
        \draw (-0.4,0) -- (1,0) -- (2,1) -- (3.4,1) -- (3.7,1.3);
        \draw[dotted] (3.7,1.3) -- (4.1,1.7);
        \draw (4.1,1.7) -- (4.4,2) -- (5.8,2) -- (6.8,3);
        \draw  (1,0) -- (1,-1.4);
        \draw  (2,1) -- (2,2.4);
        \draw  (3.4,1) -- (3.4,-0.4);
        \draw  (4.4,2) -- (4.4,3.4);
        \draw  (5.8,2) -- (5.8,0.6);
        \node at (1.2, -0.8) {\scriptsize $\epsilon_3$};
        \node at (3.6, 0.4) {\scriptsize $\epsilon_3$};
        \node at (6, 1.4) {\scriptsize $\epsilon_3$};
        \node at (2.2, 1.5) {\scriptsize $\epsilon_2$};
        \node at (4.6, 2.6) {\scriptsize $\epsilon_2$};
        \end{tikzpicture}
        \caption{$X^{(-1,-1)}_N$.}
        \label{fig: chain of O(-1)s}
    \end{subfigure}
    \caption{GKM graphs of the toric geometries from \Cref{ex: Ar x C,ex: chain of O(-1)s}.}
    \label{fig: GKM graphs strip examples}
\end{figure}

\begin{example}\label{ex: chain of O(-1)s}
    Now let us consider a triangulation of the rectangular strip \eqref{eq: strip toric graph} so that the vertical half-edges alternate between pointing down and up as displayed in \Cref{fig: GKM graphs strip examples} (B).
    This is a chain of $\cO_{\bP^1}(-1) \oplus \cO_{\bP^1}(-1)$ of length $N$. We denote this threefold by \smash{$X_{\mathrlap{\raisebox{-1.2pt}{\tiny$N$}}}{\mathstrut}^{(-1,-1)}$}.
    For this threefold, \Cref{conj: strip GW formula} specialises to the following formula:
    \begin{equation*}
    \begin{split}
        &\exp \GWarg{X^{(-1,-1)}_N \times \Aaff{2}}{\gpT}{} = \Exp \Bigg(\frac{1}{[q_4][q_5]}\sum_{\substack{1\leq m\leq n\leq N \\ n-m\text{ even}}} \prod_{k=m}^n Q_k \Bigg)\\
        & \qquad \times \Exp \Bigg( - \frac{[q_2q_4q_5]}{[q_2][q_4][q_5]}\sum_{\substack{1\leq m\leq n\leq N \\ m\text{ odd}\\ n\text{ even} }} \prod_{k=m}^n Q_k - \frac{[q_3q_4q_5]}{[q_3][q_4][q_5]}\sum_{\substack{1\leq m\leq n\leq N \\ m\text{ even}\\ n\text{ odd}}} \prod_{k=m}^n Q_k \Bigg)\,.
    \end{split}
    \end{equation*}
    For $N=1$ this formula was first conjectured in \cite[Rmk.~3.2]{BS24:refGW}. We remark that consistent with \Cref{conj: rigidity},
    the formula predicts that the \GWs\ for fixed curve class $k[C_m]+\ldots+k[C_n]$ solely depends on the equivariant weights $\epsilon_4$ and $\epsilon_5$ when $m-n$ is even.
\end{example}

We are now ready to present evidence for \Cref{eq: strip GW formula} beyond \Cref{prop: strip GW formula limit}.

\begin{evidence}\label{ev: strip GW}
    Let $N=3$ and let $\beta_1,\beta_2,\beta_3$ be the curve classes of the torus stable $\bP^1$s from left to right in picture \eqref{eq: strip toric graph}, writing $(d_1,d_2,d_3):=d_1\beta_1+d_2\beta_2+d_3\beta_3$.
    We verified Conjecture~\ref{conj: strip GW formula} in the examples listed in \Cref{tab: conj strip GW evidence} up to genus $3$, for all curve classes of which some positive integer multiple appears in the given list.
    Here, the \emph{mixed example} refers to the threefold $\CYth$ obtained from picture \eqref{eq: strip toric graph} with $N=3$ torus stable $\bP^1$s where the vertical flags point up, up, down, up (from left to right).
    The full Julia code to reproduce \Cref{tab: conj strip GW evidence} may be found in \cite{website}.
    \begin{table}[H]
        \centering
        \def\arraystretch{1.5}
        \begin{tabular}{c|c}
        \textbf{Choice of $\boldsymbol{X}$} & \textbf{Verified curve classes} \\ \hline\hline
        \multirow{2}{*}{$\cA_3\times\bC$ (see \Cref{ex: Ar x C})} & $(6, 0, 0), (5, 0, 0), (4,0,0), (3,3,3), (4,2,4),$
        \\
        &
        $(3, 3, 0), (4, 2, 0), (2, 2, 2), (2, 1, 1), (1, 2, 1)$
    \\\hline
        \multirow{2}{*}{$X^{(-1,-1)}_3$ (see \Cref{ex: chain of O(-1)s})}
        &
        $(6,0,0), (5,0,0), (4,0,0), (3, 3, 0), (4, 2, 0),$
        \\
        &$(2, 2, 2), (2, 1, 1), (1, 2, 1), (3,3,3)$
        \\\hline
        \multirow{2}{*}{mixed example (see above)}
        &
        $(4,0,0), (0,4,0), (0,0,4), (2,2,0), (0, 2, 2),$
        \\
        &
        $(2, 2, 2), (3, 3, 3)$
        \end{tabular}
        \caption{Evidence for \Cref{conj: strip GW formula}}
        \label{tab: conj strip GW evidence}
    \end{table}
\end{evidence}

\begin{nonexample}
    Going beyond \Cref{conj: strip GW formula}, one might be tempted to speculate that for any chain of rational curves whose normal bundles are all of the form $\cO_{\bP^1}(-1) \oplus \cO_{\bP^1}(-1)$ or $\cO_{\bP^1}(-2) \oplus \cO_{\bP^1}$ we have
    \begin{equation}
        \label{eq: strip vanishing speculation}
        \Omega_{\beta} \neq 0 \qquad \text{only when} \qquad \beta = (0,\ldots,0,1,\ldots,1,0,\ldots,0)\,.
    \end{equation}
    Especially, we may allow the movement of the chain not be confined to a threefold $\CYth \subset \CYfi$. However, one already finds a counterexample to this speculation in the situation $\CYfi = \CYth \times \Aaff{2}$. More concretely, let $\CYth$ be the result of deleting toric strata from $K_{\bP^1 \times \bP^1}$ until there are only two torus stable $\bP^1$ meeting in a point. Then the normal bundle of both curves is $\cO_{\bP^1}(-2) \oplus \cO_{\bP^1}$ but numerically one finds that
    \begin{equation*}
        \Omega_{(2,2)} \neq 0\,.
    \end{equation*}
    In particular, this geometry violates \eqref{eq: strip vanishing speculation}. However experimentally, we do observed other examples beyond $\CYfi = \CYth \times \Aaff{2}$ where the vanishing \eqref{eq: strip vanishing speculation} appears to occur. This raises the following question which we leave for future research.
\end{nonexample}

\begin{question}
    Is there a weaker assumption on chains of rational curves in fivefolds than the one in \Cref{conj: strip GW formula} which ensures the vanishing \eqref{eq: strip vanishing speculation}?
\end{question}

\subsection{Interlude: The closed vertex}
Let us quickly discuss a geometry which is closely related but falls outside the class of strip geometries: the closed vertex. By this we mean the toric Calabi--Yau threefold $\CYth$ whose GKM graph under the embedding induced by the moment map associated with the two-dimensional Calabi--Yau torus of $\CYth$ is displayed in \Cref{fig: closed vertex}.
\begin{figure}
    \centering
    \begin{tikzpicture}[baseline={(current bounding box.center)}]
    \foreach \pos in {{1,0},{2,1},{3.4,1},{2,2.4}}{
		\node[circle,inner sep=1.5pt,fill=black] at (\pos){};
	}
    \draw (-0.4,0) -- (1,0) -- (2,1) -- (3.4,1) -- (4.4,2);
    \draw  (2,1) -- (2,2.4) -- (3,3.4);
    \draw  (1,0) -- (1,-1.4);
    \draw  (3.4,1) -- (3.4,-0.4);
    \draw  (2,2.4) -- (0.6,2.4);
    \node[right] at (2,1.7) {$Q_1$};
    \node[below] at (2.7,1) {$Q_2$};
    \node[above] at (1.2,0.4) {$Q_3$};
    \node[above] at (0.3,0) {\scriptsize$\epsilon_2$};
    \node[above] at (1.3,2.4) {\scriptsize$\epsilon_2$};
    \node[above] at (2.3,2.8) {\scriptsize$\epsilon_3$};
    \node[above] at (3.7,1.4) {\scriptsize$\epsilon_3$};
    \node[right] at (1,-0.7) {\scriptsize$\epsilon_1$};
    \node[right] at (3.4,0.3) {\scriptsize$\epsilon_1$};
\end{tikzpicture}
    \caption{The GKM graph of the closed vertex.}
    \label{fig: closed vertex}
\end{figure}
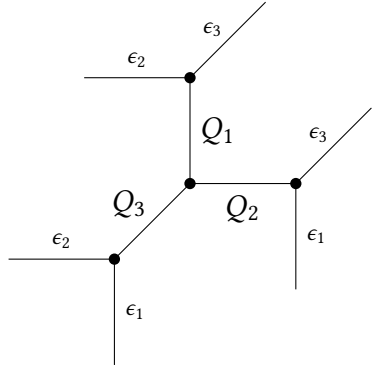
More precisely, $\CYth$ is obtained by compactifying the three coordinate lines of $\Aaff{3}$ to a $\bP^1$ in a way that the normal bundle of each line is $\cO_{\bP^1}(-1)\oplus \cO_{\bP^1}(-1)$.

We expect that the PT and \GWs\ of $\CYth$, respectively $\CYth \times \Aaff{2}$, admit a closed formula similar to the one we presented for strip geometries.

\begin{conj}\label{conj: closed vertex formula}
    Suppose $(X\times \Aaff{2},\Gm_q \times \gpT')$ is equivariantly Calabi--Yau with $\gpT'$-weights and curve class variables as indicated in \Cref{fig: closed vertex}. Then
    \begin{equation}
    \label{eq: closed vertex formula}
    \begin{split}
        & \qquad \PTarg{}{\CYth}{\CYth \times \Aaff{2}}{\Gm_q}{\gpT'} = \exp\,\GWarg{\CYth \times \Aaff{2}}{\gpT}{} = \\
        \Exp & \left( \quad \frac{1}{[q_4][q_5]}\left(Q_1+Q_2+Q_3\right) -\frac{[q_3q_4q_5]}{[q_3][q_4][q_5]} Q_1 Q_2 - \frac{[q_2q_4q_5]}{[q_2][q_4][q_5]} Q_1 Q_3  \right.  \\
          & \quad\left.-\frac{[q_1q_4q_5]}{[q_1][q_4][q_5]} Q_2 Q_3 + \frac{1}{[q_4][q_5]}Q_1Q_2Q_3 \right)
    \end{split}
    \end{equation}
\end{conj}

\begin{evidence}
    \label{ev: closed vertex formula PT}
    Numerically, we checked that the $Q_1^{d_1}Q_2^{d_2}Q_3^{d_3}$-coefficients of the \PTs\ indeed match the right-hand side of equation \eqref{eq: closed vertex formula} for $d_1+d_2+d_3 \leq 5$ up to $6$ 
    orders beyond the leading order in $q$. The calculation was done via the K-theoretic vertex as explained in \Cref{sec: PT numerics} and can be reproduced using the code documented in \cite{website}.
\end{evidence}

\begin{evidence}
    \label{ev: closed vertex formula GW}
    For the \GWs , we verified \Cref{conj: closed vertex formula} up to genus $3$ for all curve classes that have a positive integer multiple contained in the following list.
    Here, we write $(d_1,d_2,d_3)$ for the curve class associated to $Q^{\mathrlap{d_1}}{\mathstrut}_{1} \, Q^{\mathrlap{d_2}}{\mathstrut}_{2} \,Q^{\mathrlap{d_3}}{\mathstrut}_{3}$.
    \[
    (6,0,0), (5,0,0), (4,0,0), (3, 3, 0), (2, 2, 0), (2, 2, 2), (4, 2, 0), (1, 2, 1), (1,3,1)
    \]
    This list may be reproduced using the code documented in \cite{website}.
\end{evidence}

\subsection{Proof of \texorpdfstring{\Cref{prop: strip PT formula}}{Proposition 2.1}}
\label{sec: strip PT formula proof}
We will prove formula \eqref{eq: strip PT formula} for the \PTs\ of a strip geometry via a formula for the so-called \textit{capped} K-theoretic vertex proven by Kononov--Okounkov--Osinenko \cite{KOO21:twoLegPT}. More generally, we are actually going to prove a slight generalisation of this formula for relative PT invariants of a partial compactification of a strip geometry.

\subsubsection{Geometric setup}
Let us first specify the partial compactification. For this suppose $\CYth$ is a strip geometry as introduced in \Cref{sec: strips setup}. It constitutes of a chain of $N$ torus stable $\bP^1$s. Adjacent to the two torus fixed points at the ends of this chain are two torus stable $\Aone$s respectively: One has torus weight $\epsilon_2$ or $\epsilon_3$ and is represented by an upward or downward pointing ray in the GKM graph \eqref{eq: strip toric graph}. The other is represented by a ray with non-zero horizontal slope. We compactify a torus stable $\Aone$ of the latter type to a $\bP^1$ by adding an affine plane $\Aaff{2}$ at infinity in a way that the torus acts with tangent weights $\epsilon_2$, $\epsilon_3$ on this plane. We denote this partial compactification of $\CYth$ by \smash{$\ol{\CYth}$} and write \smash{$D_1 \sqcup D_2 = \ol{\CYth} \setminus \CYth$}. See \Cref{fig: strip compactified} for an illustration of a GKM graph of such a partial compactification. We remark that by our choice of torus weights on the divisors $D_i$, the normal bundles of the two compactified ends $C_0$ and $C_{N+1}$ must be isomorphic to $\cO_{\bP^1}(-1) \oplus \cO_{\bP^1}$. This means \smash{$(\ol{\CYth} \mid D_1+D_2)$} is logarithmically Calabi--Yau and so is the fivefold
\begin{equation*}
    (\ol{\CYfi} \mid \ol{D}_1+\ol{D}_2) = (\ol{\CYth} \times \Aaff{2} \mid D_1 \times \Aaff{2}+D_2 \times \Aaff{2})\,.
\end{equation*}

\begin{figure}
    \centering
    \begin{tikzpicture}[baseline={(current bounding box.center)}]
    \foreach \pos in {{0.5,0.7},{1.5,1.7},{3.45,1.7},{3.95,1.2},{4.2,0.7},{4.6,0.3},{6.15,0.3},{6.3,0.15}}{
		\node[circle,inner sep=1.3pt,fill=black] at (\pos){};
	}
    \draw (-0.8,0.7) -- (0.5,0.7) -- (1.5,1.7) -- (3.45,1.7) -- (3.95,1.2) -- (4.2,0.7) -- (4.6,0.3) -- (6.15,0.3) -- (6.3,0.15) -- (7.8,0.15);
    \draw (0.5,0.7) -- (0.5,-0.5);
    \draw (1.5,1.7) -- (1.5,2.5);
    \draw (3.45,1.7) -- (3.45,2.5);
    \draw (3.95,1.2) -- (3.95,2.5);
    \draw (4.2,0.7) -- (4.2,-0.5);
    \draw (4.6,0.3) -- (4.6,-0.5);
    \draw (6.15,0.3) -- (6.15,2.5);
    \draw (6.3,0.15) -- (6.3,-0.5);
    \draw[decorate, decoration={snake, segment length=0.7em, amplitude=0.1em}] (-0.8,2.5) -- (-0.8,-0.5) node[below] {$D_1$};
    \draw[decorate, decoration={snake, segment length=0.7em, amplitude=0.1em}] (7.8,2.5) -- (7.8,-0.5) node[below] {$D_2$};
    \node[above] at (-0.15,0.7) {$C_0$};
    \node[above] at (7.05,0.15) {$C_{N+1}$};
\end{tikzpicture}
    \caption{The GKM graph of the partial compactification $\ol{X}$. The squiggly edges represent the image of the relative divisors $D_1$ and $D_2$ under the moment map.}
    \label{fig: strip compactified}
\end{figure}

\subsubsection{Relative PT invariants}
The moduli space of relative stable pairs $P_n(\ol{\CYth} \mid D_1+D_2, \beta)$ comes equipped with a proper evaluation morphism
\begin{equation*}
    \mr{ev}: P_n(\ol{\CYth} \mid D_1+D_2, \beta) \longrightarrow \mr{Hilb} (D_1) \times \mr{Hilb} (D_2)
\end{equation*}
to the Hilbert scheme of points on $D_1 \sqcup D_2\eqqcolon D$. Following \cite[Sec.~2.1.6]{KOO21:twoLegPT} we twist the virtual structure sheaf of this moduli space
\begin{equation*}
     \widetilde{\cO}^{\vir}_{P_n(\ol{\CYth} \mid D, \beta)} \coloneqq  \cO^{\vir}_{P_n(\ol{\CYth} \mid D, \beta)} \otimes \left(K_{\vir}\otimes \det \mathbf{R}\pi_* \,\mathbb{F}|_D\right)^{1/2}
\end{equation*}
and consider its push-forward along the evaluation morphism:
\begin{equation}
    \label{eq: defn relative PT}
    \mr{ev}_* \, \widetilde{\cO}^{\vir}_{P_n(\ol{\CYth} \mid D, \beta)}\,.
\end{equation}
This push-forward takes values in the equivariant K-theory of the Hilbert scheme of points on the affine plane which can be identified with the ring of symmetric functions $\Lambda$ with values in $\Rep \gpT'$:
\begin{equation*}
    K_{\gpT'}\big(\mr{Hilb} (D_i)\big) \cong \Lambda \otimes \Rep \gpT'\,, \qquad i\in\{1,2\}\,.
\end{equation*}
Writing $\cev{p}_k$ and $\vec{p}_k$ for the $k$th power sum in the K-theory of the Hilbert scheme of $D_1$ respectively $D_2$, we can prove the following formula for the generating series of the invariants \eqref{eq: defn relative PT} which we form exactly as in \eqref{eq: defn PT series}:
\begin{equation*}
    \PTarg{}{\CYth}{\ol{\CYfi} \mid \ol{D}}{\Gm_q}{\gpT'} \coloneqq \!\!\!\! \sum_{\substack{\beta \neq 0\\ \beta = \sum_{i=0}^{n+1} d_i [C_i]}} \!\!\!\prod_{i=1}^n Q_i^{d_i} \sum_{n\in \bZ} (-q)^n ~ \mr{ev}_* \, \widetilde{\cO}^{\vir}_{P_n(\ol{\CYth} \mid D, \beta)}\,.
\end{equation*}

\begin{prop}
    \label{prop: relative strip PT formula}
    Writing \smash{$q_4 = \kappa^{1/2}q$} and \smash{$q_5 = \kappa^{1/2}q^{-1}$}, for $N\geq 0$ we have
    \begin{equation}
    \label{eq: relative strip PT formula}
    \begin{split}
        & \qquad \PTarg{}{\ol{\CYth}}{\ol{\CYfi} \mid \ol{D}}{\Gm_q}{\gpT'} \\
        =& \Exp \left(-\sum_{1 \leq m < n \leq N+1} \frac{[q_{v_m} q_4 q_5]}{[q_{v_n}] \, [q_4] \, [q_5]} \prod_{i=m}^{n-1} Q_i \right) \\[1em]
        \times & \Exp \left( \frac{q_4 [q_4]}{[q_2][q_3][q_5]} \,\cev{p}_1 \, \vec{p}_1 \prod_{i=1}^N Q_i - \sum_{m = 1}^{N+1} \frac{q_4^{1/2}}{[q_{v_m}][q_5]} \, \cev{p}_1 \prod_{i=1}^{m-1} Q_i - \sum_{m = 1}^{N+1} \frac{q_4^{1/2}}{[q_{v_{m}}][q_5]} \, \vec{p}_1 \prod_{i=m}^{N} Q_i \right)\,.
    \end{split}
    \end{equation}
\end{prop}
Formula \eqref{eq: strip PT formula} for the \PTs\ of $\CYth$ then follows by specialising the above formula to curve classes $\beta$ which are supported away from the ends $C_0$ and $C_{n+1}$. This is achieved by setting $\cev{p}_1 = \vec{p}_1 = 0$ in \eqref{eq: relative strip PT formula}. To prove \Cref{prop: strip PT formula} it therefore suffices to show \Cref{prop: relative strip PT formula}.

We remark that for the resolved conifold $X=\Tot_{\bP^1} \cO(-1) \oplus \cO(-1)$ formula \eqref{eq: relative strip PT formula} has already been proven via degeneration in \cite[Thm.~2]{KOO21:twoLegPT}. For the proof of the general formula we will proceed in the same way.

\subsubsection{Proof via degeneration}
We will prove formula \eqref{eq: relative strip PT formula} by induction on the number $N$ of torus invariant $\bP^1$s via a degeneration argument.

First we claim that the base case $N=0$ is precisely the formula for the capped two-leg vertex proven by Kononov--Okounkov--Osinenko in \cite{KOO21:twoLegPT}. Indeed, for $N=0$ the GKM graph of \smash{$\ol{\CYth}$} is
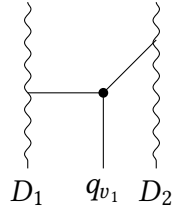
\begin{equation*}
    \begin{tikzpicture}[baseline={(current bounding box.center)}]
    \node[circle,inner sep=1.3pt,fill=black] at (0,0){};
    \draw (-1,0) -- (0,0) -- (0.7,0.7);
    \draw (0,0) -- (0,-1) node[below] {$q_{v_1}$};
    \draw[decorate, decoration={snake, segment length=0.7em, amplitude=0.1em}] (-1,1.2) -- (-1,-1) node[below] {$D_1$};
    \draw[decorate, decoration={snake, segment length=0.7em, amplitude=0.1em}] (0.7,1.2) -- (0.7,-1) node[below] {$D_2$};
    \end{tikzpicture}
\end{equation*}
This is exactly the same geometric setup as considered in \cite[Sec.~2.1.1]{KOO21:twoLegPT}. The following formula for the \PTs\ of \smash{$(\ol{\CYth} \mid D_1 + D_2)$} is proven in \cite[Thm.~1]{KOO21:twoLegPT}:
\begin{equation}
	\label{eq: PT capped 2 leg vertex}
	\PTarg{}{\ol{\CYth}}{\ol{\CYfi} \mid \ol{D}}{\Gm_q}{\gpT'} = \Exp \left(\frac{q_4 [q_4]}{[q_2][q_3][q_5]}\, \cev{p}_1 \, \vec{p}_1 -\frac{q_4^{1/2}}{[q_{v_1}][q_5]} \, \cev{p}_1 -\frac{q_4^{1/2}}{[q_{v_1}][q_5]} \, \vec{p}_1\right)
\end{equation}
This is precisely the specialisation of formula \eqref{eq: relative strip PT formula} to $N=0$ which means the base case is proven.

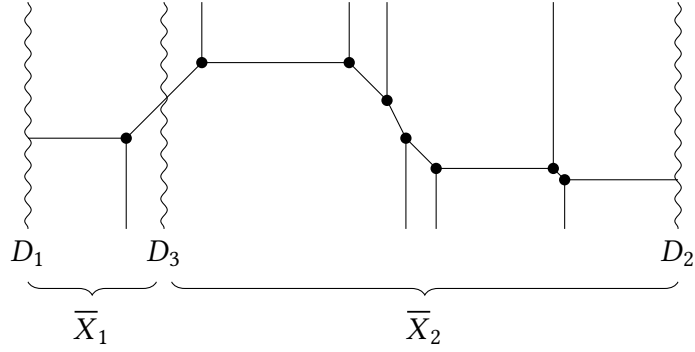
\begin{figure}
    \centering
    \begin{tikzpicture}[baseline={(current bounding box.center)}]
    \foreach \pos in {{0.5,0.7},{1.5,1.7},{3.45,1.7},{3.95,1.2},{4.2,0.7},{4.6,0.3},{6.15,0.3},{6.3,0.15}}{
		\node[circle,inner sep=1.5pt,fill=black] at (\pos){};
	}
    \draw (-0.8,0.7) -- (0.5,0.7) -- (1.5,1.7) -- (3.45,1.7) -- (3.95,1.2) -- (4.2,0.7) -- (4.6,0.3) -- (6.15,0.3) -- (6.3,0.15) -- (7.8,0.15);
    \draw (0.5,0.7) -- (0.5,-0.5);
    \draw (1.5,1.7) -- (1.5,2.5);
    \draw (3.45,1.7) -- (3.45,2.5);
    \draw (3.95,1.2) -- (3.95,2.5);
    \draw (4.2,0.7) -- (4.2,-0.5);
    \draw (4.6,0.3) -- (4.6,-0.5);
    \draw (6.15,0.3) -- (6.15,2.5);
    \draw (6.3,0.15) -- (6.3,-0.5);
    \draw[decorate, decoration={snake, segment length=0.7em, amplitude=0.1em}] (-0.8,2.5) -- (-0.8,-0.5) node[below] {$D_1$};
    \draw[decorate, decoration={snake, segment length=0.7em, amplitude=0.1em}] (7.8,2.5) -- (7.8,-0.5) node[below] {$D_2$};
    \draw[decorate, decoration={snake, segment length=0.7em, amplitude=0.1em}] (1,2.5) -- (1,-0.5) node[below] {$D_3$};
    \draw [decorate,decoration={brace,amplitude=5pt,mirror}]
  (-0.8,-1.2) -- (0.9,-1.2) node[midway,yshift=-1.4em] {$\ol{X}_1$};
    \draw [decorate,decoration={brace,amplitude=5pt,mirror}]
  (1.1,-1.2) -- (7.8,-1.2) node[midway,yshift=-1.4em] {$\ol{X}_2$};
    \end{tikzpicture}
    \caption{The degeneration $\ol{X}_1 \cup_{D_3} \ol{X}_2$.}
    \label{fig: strip PT degen}
\end{figure}

For the induction step, suppose we start from a strip geometry $\CYth$ with $N>0$ torus stable $\bP^1$s. We degenerate the partial compactification \smash{$\ol{\CYth}$} to a union of a compactified vertex $\ol{X}_1$ and a chain of $N-1$ compactified torus stable $\bP^1$s $\ol{X}_2$ glued along a divisor $D_3 \cong \Aaff{2}$. This degeneration is illustrated in \Cref{fig: strip PT degen}. We now require the degeneration formula for PT invariants \cite{LW15:GoodDegQuotSch} which in the K-theoretic setting is discussed for instance in \cite[Sec.~6.5]{Ok15:lectures}. It states that the \PTs\ of $\CYth$ can be recovered from the ones of \smash{$\ol{X}_1$} and \smash{$\ol{X}_2$}:
\begin{equation*}
    \PTarg{}{\ol{\CYth}}{\ol{\CYfi} \mid \ol{D}}{\Gm_q}{\gpT'} = \PTarg{}{\ol{\CYth}_1}{\ol{\CYfi}_1 \mid \ol{D}_1 + \ol{D}_3}{\Gm_q}{\gpT'}  \cdot \mathsf{G}^{-1} \cdot  \PTarg{}{\ol{\CYth}_1}{\ol{\CYfi}_2 \mid \ol{D}_3 + \ol{D}_2}{\Gm_q}{\gpT'}\,.
\end{equation*}
Here, $\mathsf{G}^{-1}$ denotes the so-called gluing operator for which a closed formula was proven in \cite{OS22:QuDifEqNakVars}:
\begin{equation*}
    \mathsf{G}^{-1} = \Exp \left(\frac{[q_2] [q_3] [q_5]}{q_4 [q_4]} \, \cev{p}_1 \, \vec{p}_1 \, Q_1\right)
\end{equation*}
The product indicates that we must contract elements of $K_{\gpT'}\big(\mr{Hilb} (D_3)\big)^{\otimes2}$ with respect to the standard Hall product on symmetric functions. With formula \eqref{eq: PT capped 2 leg vertex} for the capped two-leg vertex and the formula for the gluing operator the first contraction gives
\begin{equation*}
    \PTarg{}{\ol{\CYth}_1}{\ol{\CYfi}_1 \mid \ol{D}_1 + \ol{D}_3}{\Gm_q}{\gpT'}  \cdot \mathsf{G}^{-1} = \Exp \left(\cev{p}_1 \, \vec{p}_1 \, Q_1  -\frac{q_4^{1/2}}{[q_{v_1}][q_5]} \, \cev{p}_1 - \frac{[q_2][q_3]}{q_4^{1/2} [q_{v_1}][q_4]}\,\vec{p}_1 \,  Q_1 \right)
\end{equation*}
Contracting this expression with the formula for a chain of $N-1$ torus stable $\bP^1$s \eqref{eq: relative strip PT formula} we indeed arrive at the desired formula. This closes our proof of \Cref{prop: relative strip PT formula}.

\section{Geometric engineering of pure gauge theories}
So far our discussion focused on the left side of the web of correspondences in \Cref{fig: summary paper}. In this section we will shift gears and relate the enumerative geometry of Calabi--Yau threefolds to that of the affine plane.

\subsection{Geometric setup}
\label{sec: X N m setup}
Fix integers $N>0$ and $m\in\bZ$. In this section we consider Calabi--Yau fivefolds which are the product of the affine plane with a toric Calabi--Yau threefold $X_{N,m}$ whose toric diagram we display in \Cref{fig: X N m toric diagram}. We assume that a torus $\gpT$ acts on the fivefold $Z_{N,m} = X_{N,m} \times \Aaff{2}$ in a way that the action on the fivefold is Calabi--Yau and the torus weights at the fixed points of the threefold are as indicated in \Cref{fig: X N m toric diagram}. As before, we denote the tangent weights at the origin of the affine plane by $\epsilon_4$ and $\epsilon_5$.
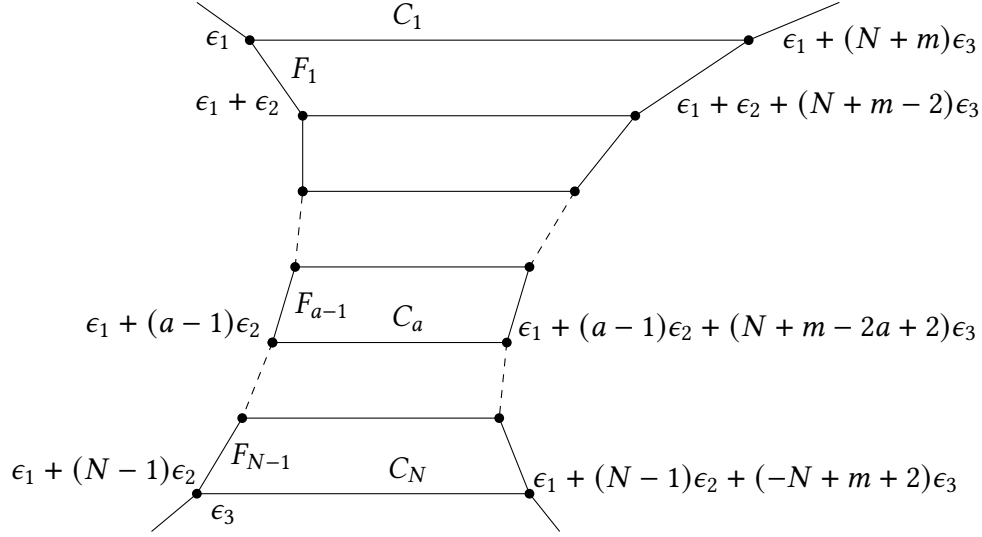
\begin{figure}
\centering
\begin{tikzpicture}
    \foreach \pos in {{0,0},{3.1,0},{3.4,1},{0.3,1},{0,0},{0.4,2},{4,2},{4.8,3},{0.4,3},{0.4,2},{-0.4,-1},{3,-1},{3.4,-2},{-1,-2},{6.3,4},{-0.3,4}}{
		\node[circle,inner sep=1.3pt,fill=black] at (\pos){};
	}
    \draw (0,0) -- (3.1,0) -- (3.4,1) -- (0.3,1) -- (0,0);
    \draw[dashed] (0.3,1) -- (0.4,2);
    \draw[dashed] (3.4,1) -- (4,2);
    \draw (0.4,2) -- (4,2) -- (4.8,3) -- (0.4,3) -- (0.4,2);
    \draw (4.8,3) -- (6.3,4) -- (-0.3,4) -- (0.4,3);
    \draw (6.3,4) -- (7.5,4.5);
    \draw (-0.3,4) -- (-1,4.5);
    \draw[dashed] (0,0) -- (-0.4,-1);
    \draw[dashed] (3.1,0) -- (3,-1);
    \draw (-0.4,-1) -- (3,-1) -- (3.4,-2) -- (-1,-2) -- (-0.4,-1);
    \draw (3.4,-2) -- (3.8,-2.5);
    \draw (-1,-2) -- (-1.6,-2.5);
    \node[below] at (-0.65,-2) {$\epsilon_3$};
    \node[right] at (6.6,4) {$\epsilon_1+(N+m)\epsilon_3$};
    \node[right] at (5.2,3.15) {$\epsilon_1+\epsilon_2+(N+m-2)\epsilon_3$};
    \node[right] at (3.1,0.2) {$\epsilon_1+(a-1)\epsilon_2+(N+m-2a+2)\epsilon_3$};
    \node[right] at (3.3,-1.8) {$\epsilon_1+(N-1)\epsilon_2+(-N+m+2)\epsilon_3$};
    \node[left] at (-0.85,-1.7) {$\epsilon_1+(N-1)\epsilon_2$};
    \node[left] at (0,0.25) {$\epsilon_1+(a-1)\epsilon_2$};
    \node[left] at (-0.4,4) {$\epsilon_1$};
    \node[left] at (0.25,3.15) {$\epsilon_1+\epsilon_2$};
    \node[above] at (1.8,4) {$C_1$};
    \node[above] at (1.8,0) {$C_a$};
    \node[above] at (1.8,-2) {$C_N$};
    \node[right] at (0.1,3.6) {$F_1$};
    \node[right] at (0.15,0.5) {$F_{a-1}$};
    \node[right] at (-0.7,-1.5) {$F_{N-1}$};
\end{tikzpicture}
\caption{The projection of the one-skeleton of the toric Calabi--Yau $X_{N,m}$ under the moment map associated to the action by its Calabi--Yau torus. We denote the horizontal torus preserved compact lines by $C_1,\ldots,C_N$ and left vertical ones by $F_1,\ldots,F_{N-1}$.}
\label{fig: X N m toric diagram}
\end{figure}

\subsection{Motivation}
\label{sec: gauge th motivation}
This class of toric threefolds is well studied in the physics literature because of their connection to supersymmetric gauge theory. It is expected that the partition function of A-twisted refined topological string theory on $X_{N,m}$ equates to Nekrasov's (K-theoretic) 5d instanton partition function of pure supersymmetric $\mathop{SU}(N)$ gauge theory on $\Aaff{2}$ with Chern--Simons level $m$ for the general $\Omega$ background \cite{Nek05:ChasingM,IKV09:refTopVert}. In mathematical terms, the latter function is the generating series of equivariant $\hat{A}$-genera of instanton moduli spaces. These generating series enjoy explicit formulae.

It was proposed in \cite{BS24:refGW} that the refined topological string partition function can mathematically be realised as the equivariant \GWs\ of $X_{N,m}\times \Aaff{2}$. This means we should find that this series --- and by \Cref{conj: GW PT} also the \PTs\ --- matches the explicit formula for the instanton partition function. The purpose of this section is to provide evidence for this claim partly building on earlier work of Arbesfeld who worked out the special case $N=2$, $m=0$ in the PT setting \cite{Arb21:KthDTHilbSurf}. However, to be able to state the conjectural correspondence we first need to introduce more notation.

\subsection{Nekrasov partition function of pure gauge theory}
For an integer partition $\mu=(\mu_0,\mu_1,\ldots)$ let us write
\begin{equation*}
    |\mu|=\sum_{i\geq 0} \mu_i \,,\qquad \sqabs{\mu}=\sum_{i\geq 0} \mu_i^2\,.
\end{equation*}
Moreover, we introduce the following notation:
\begin{gather*}
    f_{\mu}(q_4,q_5) \coloneqq(-1)^{|\mu|}\, q_4^{\sqabs{\mu^{\mr{t}}}/2} \, q_5^{\sqabs{\mu}/2}\,, \qquad \tilde{Z}_\mu(q_4,q_5) \coloneqq \prod_{(i,j)\in\mu} \Big(1-q_4^{\mu_j^{\mr{t}} - i} q_5^{-\mu_i + j+1}\Big)^{-1}\,, \\
    \cR_{\mu,\nu}(q_4,q_5;Q) \coloneqq  \prod_{i,j\geq 0} \frac{1}{1 - Q \,q_4^{i-\mu_j+1} q_5^{-j+\nu_i-1}}\,.
\end{gather*}
With this notation at hand, we define the following power series in the completed semigroup ring of $\Heff{X_{N,m}}$ with coefficients rational functions in $q_4$ and $q_5$:
\begin{equation*}
\begin{split}
    &\qquad \ZNek_{N,m}(q_4,q_5;Q)= \\
    &\sum_{\mu_1,\ldots,\mu_N} \prod_{a=1}^N \left[\left(-Q^{[C_a]}\right)^{|\mu_a|} \, f_{\mu_a}\big(q_4,q_5\big)^{2a-N-m-1}  q_4^{\sqabs{\mu_a^{\mr{t}}}/2} q_5^{-\sqabs{\mu_a}/2} \,\tilde{Z}_{\mu_a}\big(q_4,q_5\big) \, \tilde{Z}_{\mu_a^{\mr{t}}}\big(q_5^{-1},q_4^{-1}\big) \right] \\
    & \hspace{1em}\times \prod_{1\leq a<b\leq N} \cR_{\mu_a^{\mr{t}}, \mu_b } \Big(q_4,q_5 \, ; \, q_5 {\textstyle\prod_{c=a}^{b-1}} Q^{[F_c]}\Big) ~ \cR_{\mu_a^{\mr{t}}, \mu_b } \Big( q_4,q_5 \, ;\, q_4^{-1} {\textstyle\prod_{c=a}^{b-1}} Q^{[F_c]}\Big)\,.
\end{split}
\end{equation*}
In \cite{Tak08:RefVertInstanton}, the above series was matched with the 5d instanton partition function of pure $\mr{SU}(N)$ gauge theory with Chern--Simons level $m$ modulo perturbative corrections under a suitable identification of the curve class variables $Q^\beta$ with the energy scale and the flavour charges of the gauge theory. And indeed, confirming our motivation from \Cref{sec: gauge th motivation} one can show that the \PTs\ of $X_{N,m}$ matches the instanton partition function.

\begin{thm}
    \label{thm: gauge PT correspondence}
    For $|m|<N$ we have
    \begin{equation}
        \label{eq: PT gauge}
        \PTarg{}{X_{N,m}}{Z_{N,m}}{\Gm_q}{\gpT'} = \ZNek_{N,m}(\kappa^{1/2}q ,\kappa^{1/2}q^{-1};Q)\,.
    \end{equation}
    Especially, \Cref{conj: PT GV} holds for $X_{N,m}$.
\end{thm}

We defer the proof of \Cref{thm: gauge PT correspondence} to \Cref{sec: gauge PT correspondence proof}.

In this \namecref{thm: gauge PT correspondence} we adopted the same convention for the torus action on $\CYth_{N,m} \times \Aaff{2}$ as in \Cref{sec: strips} assuming that the torus $\gpT' \times \Gm_q$ acts on the two affine directions with characters $$\kappa^{1/2}q \qquad \text{and} \qquad \kappa^{1/2}q^{-1}$$ where $\kappa$ denotes the character by which $\gpT'$ scales the holomorphic three-form of $\CYth_{N,m}$.

\begin{rmk}
    \label{rmk: curve relations gauge}
    We stress that \eqref{eq: PT gauge} is an equality in the completed semigroup ring of effective curve classes $R\llbracket \Heff{X_{N,m}} \rrbracket$ (for some suitable ring $R$) and is generally \textit{false} in the freely generated power series ring \smash{$R\llbracket Q^{[C_1]},Q^{[F_1]},\ldots \rrbracket$}. This means we need to impose all relations in $H_2^{\mr{eff}}(X_{N,m},\bZ)$ among the curve class variables. Note that the semigroup of effective curve classes is generated by $[F_1],\ldots,[F_{N-1}],[C]$ where each $F_a$ is a compact irreducible component of a fibre $X_{N,m}\rightarrow \bP^1$ and $C$ is a specific section of the fibration. More specifically, we have
    \begin{equation*}
        C = C_{\lfloor \frac{N+m+1}{2}\rfloor}
    \end{equation*}
    with the subscript chosen as in \Cref{fig: X N m toric diagram}. The class of all other torus preserved sections $C_a$ can be decomposed as
    \begin{equation*}
        [C_a] = [C]+\sum_{b=a}^{\left\lfloor \frac{N+m+1}{2}\right\rfloor-1} (N+m-2b) [F_b]
    \end{equation*}
    when $a \leq \frac{N+m+1}{2}$ and otherwise we have
    \begin{equation*}
        [C_a] = [C]+\sum_{b=\left\lceil \frac{N+m+1}{2}\right\rceil}^{a-1} (2b-N-m) [F_b]\,.
    \end{equation*}
    These give all relations among the curve class variables in $R\llbracket \Heff{X_{N,m}} \rrbracket$.
\end{rmk}

Before proving \Cref{thm: gauge PT correspondence} let us first name the implications of this \namecref{thm: gauge PT correspondence} for the Gromov--Witten theory of $X_{N,m}\times \Aaff{2}$.

\begin{conj}
\label{conj: gauge GW correspondence}
    Suppose $(X_{N,m},\gpT)$ is equivariantly Calabi--Yau. Then if $|m|<N$ we have
    \begin{equation}
        \label{eq: gauge GW correspondence}
        \exp \, \GWarg{X_{N,m}\times \Aaff{2}}{\gpT}{} = \ZNek_{N,m}(\re^{u \epsilon_4},\re^{u \epsilon_5};Q)\,.
    \end{equation}
    In particular, \Cref{conj: gen GV} and \Cref{conj: rigidity} hold.
\end{conj}

\begin{rmk}
    We stress that \Cref{rmk: curve relations gauge} equally applies to the above \namecref{conj: gauge GW correspondence}. Moreover, we remark that all infinite products on the right-hand side of equation \eqref{eq: gauge GW correspondence} need to be expanded as
    \begin{align*}
        \cR_{\mu,\nu}(q_4,q_5;Q) & = \Exp\Big(Q\cdot r_{\mu,\nu}(q_4,q_5)\Big) \prod_{\substack{0\leq i \leq \ell(\nu)-1 \\ 0\leq j \leq  \ell(\mu)-1}} \frac{1 - Q \, q_4^{i+1} q_5^{-j-1}}{1 - Q \,q_4^{i-\mu_j+1} q_5^{-j+\nu_i-1}}
    \end{align*}
    where
    \begin{equation*}
    \begin{split}
        & \qquad r_{\mu,\nu}(q_4,q_5)\coloneqq \\
        &\frac{q_4q_5^{-1}}{(1-q_4)(1-q_5^{-1})}+\frac{\sum_{i=0}^{\ell(\nu)-1} q_4^{i+1} q_5^{-\ell(\mu)-1}(q_5^{\nu_i}-1)}{1-q_5^{-1}} +  \frac{\sum_{j=0}^{\ell(\mu)-1} q_5^{-j-1} q_4^{\ell(\nu)+1}(q_4^{-\mu_j}-1)}{1-q_4}\,.
    \end{split}
    \end{equation*}
\end{rmk}

\Cref{conj: gauge GW correspondence} is known to be true in the anti-diagonal specialisation $\epsilon_4 = - \epsilon_5$. Indeed, by \eqref{eq: GW 5 to 3 reduction strip} in this limit the right-hand side of \eqref{eq: gauge GW correspondence} reduces to the generating series of Gromov--Witten invariants of $X_{N,m}$ with $\epsilon_4$ acting as the genus counting variable. Using the topological vertex, this threefold series was matched with the anti-diagonal specialisation of the instanton partition function in \cite{IKP06:SUNtopstr}. Beyond the case where $\epsilon_4 = - \epsilon_5$, we are able to provide the following numerical evidence for our conjecture.

\begin{evidence}
    \label{ev: gauge GW correspondence}
    We verified \Cref{conj: gauge GW correspondence} for all $(N,m)$ listed in \Cref{tab: conj gauge GW correspondence} up to the indicated genus $g_\text{max}$ and for all curve classes for which some positive integer multiple appears in the given list.
    We also verified that the assumption $|m|<N$ is necessary. Indeed, for $(N,m)=(3,4)$ the membrane index $\Omega_{[C_3]}$ depends on equivariant parameters other than $\epsilon_4$ and $\epsilon_5$.
    This claim and \Cref{tab: conj gauge GW correspondence} may be reproduced using the code provided in \cite{website}.

    \begin{table}[H]
        \centering
        \def\arraystretch{1.5}
        \begin{tabular}{c|c|c|c}
            $\boldsymbol{N}$ & $\boldsymbol{m}$ & 
            $\boldsymbol{g_\text{max}}$ &
            \textbf{Verified curve classes}
            \\\hline\hline
            $2$ & $-1$ & $2$
            &
            $[C_1], 4[C_2], 3[C_2], 4[F_1],3[F_1], 2[F_1]+2[C_2]$
        \\\hline
        $2$ & $0$ & $3$
        & \multirow{2}{*}{$\begin{matrix}4[C_1], 3[C_1], 4[C_2], 3[C_2], 4[F_1], 3[F_1], \\ 2[F_1]+2[C_2], 2[F_1]+2[C_1]+2[C_2]\end{matrix}$}
        \\\cline{1-3}
        $2$ & $1$ & $3$
        &
        \\\hline
        $3$ & $-2$ & $2$
        &
        $[C_1], [C_2], 2[F_1]+2[F_2], [C_2]+[F_1], [C_2]+[F_1]+[F_2]$
        \\\hline
        $3$ & $-1$ & $2$
        & 
        $[C_1], 2[C_2], 2[F_1]+2[F_2]$
        \\\hline
        $3$ & $0$ & $2$
        & 
        \multirow{3}{*}{$
        \begin{matrix}
            [C_1], 2[C_2], 4[F_1], 3[F_1], 2[F_2]+2[C_3], 4[C_3], \\3[C_3],2[F_1]+2[F_2], 2[C_2]+2[F_1]+2[F_2]
        \end{matrix}$}
        \\\cline{1-3}
        $3$ & $1$ & $2$
        & 
        \\\cline{1-3}
        $3$ & $2$ & $2$
        & 
        \\\hline
        \end{tabular}
        \caption{Evidence for \Cref{conj: gauge GW correspondence}}
        \label{tab: conj gauge GW correspondence}
    \end{table}
\end{evidence}

\subsection{Proof of \texorpdfstring{\Cref{thm: gauge PT correspondence}}{Theorem 3.1}}
\label{sec: gauge PT correspondence proof}
In this section we will show that the \PTs\ of $\CYth_{N,m}$ indeed agrees with the Nekrasov partition function {$\ZNek_{N,m}$}. We will achieve this by evaluating the earlier via the refined topological vertex. Indeed, the same strategy was pursued by Arbesfeld who computed the \PTs\ in the case $N=2$ where the four non-compact fibres in \Cref{fig: X N m toric diagram} got compactified to $\bP^1$s \cite[Sec.~5]{Arb21:KthDTHilbSurf}. Here, essentially the only task for us is to match formal calculations in Pandharipande--Thomas theory with well-established methods in the physics literature \cite{IKV09:refTopVert,Tak08:RefVertInstanton}.

\subsubsection{Generalities on the refined topological vertex.}
The refined topological vertex is a special limit of the K-theoretic vertex \smash{$\hat{V}^{\PT}_{\mu_1,\mu_2,\mu_2}(q_1,q_2,q_3;q)$} we considered in \Cref{sec: Kth vertex}. To describe the limit, we consider a one-parameter subgroup
\begin{equation*}
\begin{tikzcd}[row sep = 0em]
    \sigma: &[-2.7em] \Gm \ar[r, hook] & \Gm[3]_{q_1,q_2,q_3}\,, \\
    & z \ar[r, maps to] & (z^{r_1},z^{r_2},z^{r_3})
\end{tikzcd}
\end{equation*}
with $r_1+r_2+r_3=0$. It turns out that for special choices of slopes $(r_1,r_2,r_3)$, the limit $\lim_{z \rightarrow 0} \hat{V}^{\PT}_{\boldsymbol{\mu}}(\boldsymbol{q} \cdot \sigma(z) ; q)$ admits an explicit formula in the variables
\begin{equation*}
    q_4 = \kappa^{1/2}q \qquad q_5 = \kappa^{1/2}q^{-1} 
\end{equation*}
where $\kappa = (q_1 q_2 q_3)^{-1}$. The two limits we will require are (i) $r_1 \gg 0 > r_3 \gg r_2$ and (ii) $r_2 \gg r_3 > 0 \gg r_1$. We will use labels $\gg$ and $>$ on the edges of a GKM graph in order to indicate these two choices of slope. The following formulae are proven in \cite[Prop.~4.6]{Arb21:KthDTHilbSurf}:
\begin{center}
\begin{tikzpicture}
    \draw (0,1) -- (0,0) node[midway, sloped] {\small$\gg$};
    \draw (0,0) -- (-0.7,-0.7) node[midway, sloped] {\small$\ll$};
    \draw (0,0) -- (1,0) node[midway, sloped] {\small$<$};
    \node[circle,inner sep=1.5pt,fill=black] at (0,0){};
    \node[left] at (0,0.5) {\small$q_1,\mu_1$};
    \node[left] at (-0.35,-0.3) {\small$q_2,\mu_2$};
    \node[below] at (0.5,0) {\small$q_3,\mu_3$};
    \node at (-2.5,0) {(i)};
    \node[anchor = west] at (2.5,0) {${\displaystyle\lim_{z\rightarrow 0}} ~\hat{V}^{\PT}_{\mu_1,\mu_2,\mu_3}\big(\boldsymbol{q} \cdot \sigma(z) ; q\big) = C_{\mu_1,\mu_2,\mu_3}\big(q_5^{-1},q_4\big)$};
    \draw (0,-1.5) -- (0,-2.5) node[midway, sloped] {\small$\ll$};
    \draw (0,-2.5) -- (-0.7,-3.2) node[midway, sloped] {\small$\gg$};
    \draw (0,-2.5) -- (1,-2.5) node[midway, sloped] {\small$>$};
    \node[circle,inner sep=1.5pt,fill=black] at (0,-2.5){};
    \node[left] at (0,-2) {\small$q_1,\mu_1$};
    \node[left] at (-0.43,-2.8) {\small$q_2,\mu_2$};
    \node[below] at (0.5,-2.5) {\small$q_3,\mu_3$};
    \node at (-2.5,-2.5) {(ii)};
    \node[anchor = west] at (2.5,-2.5) {${\displaystyle\lim_{z\rightarrow 0}} ~ \hat{V}^{\PT}_{\mu_1,\mu_2,\mu_3} \big(\boldsymbol{q} \cdot \sigma(z) ; q\big) = C_{\mu_1,\mu_2,\mu_3} \big(q_4,q_5^{-1}\big)$};
\end{tikzpicture}
\end{center}
Here, $C_{\mu_1,\mu_2,\mu_3}(t,q)$ is the refined topological vertex of Iqbal--Koz{\c{c}}az--Vafa \cite{IKV09:refTopVert}:
\begin{equation}
    \label{eq: ref top vert}
    C_{\mu_1,\mu_2,\mu_3}(t,q) = t^{-\frac{\sqabs{\mu_1^{\mr{t}}}}{2}} q^{-\frac{\sqabs{\mu_2}}{2}} ~ \tilde{Z}_{\mu_3^{\mr{t}}}(t ,q^{-1}) \, 
    \sum_{\delta} \left(\frac{q}{t}\right)^{\frac{|\delta|}{2}} s_{\mu_1/\delta}\big(t^{-\rho}q^{-\mu_3^{\mr{t}}}\big) ~ s_{\mu_2^{\mr{t}}/\delta}\big(q^{-\rho}t^{-\mu_3}\big)  \,.
\end{equation}
In this formula we write $s_{\nu/\delta}(t^{-\rho}q^{-\mu})$ for the skew Schur function $s_{\nu/\delta}(\mathbf{x})$ evaluated at $$\mathbf{x}=(t^{1/2} q^{-\mu_0},t^{3/2} q^{-\mu_1},\ldots)\,.$$

Similarly, edge terms $\hat{E}_\mu(q_1,q_2,q_3;q)$ admit particularly simple formulae in refined limits:
\begingroup
\allowdisplaybreaks
\begin{align*}
& \begin{tikzpicture}
    \draw (0,1) -- (0,0) node[midway, sloped] {\small$>$};
    \draw (0,0) -- (-0.7,-0.7) node[midway, sloped] {\small$\gg$};
    \draw (0,0) -- (1.5,0) node[midway, sloped] {\small$\gg$};
    \draw (1.5,1) -- (1.5,0) node[midway, sloped] {\small$>$};
    \draw (1.5,0) -- (2.2,-0.7) node[midway, sloped] {\small$\gg$};
    \node[circle,inner sep=1.5pt,fill=black] at (0,0){};
    \node[circle,inner sep=1.5pt,fill=black] at (1.5,0){};
    \node[below, anchor = south] at (0.3,-0.6) {\small$\mu$};
    \node[below, anchor = south] at (1.2,-0.6) {\small$\mu^{\mr{t}}$};
    \node[anchor = west] at (4,0) {${\displaystyle\lim_{z\rightarrow 0}} ~\hat{E}_{\mu}\big(\boldsymbol{q} \cdot \sigma(z) ; q\big) =q_4^{-\frac{|\mu|-\sqabs{\mu}}{2}} q_5^{-\frac{|\mu| + \sqabs{\mu} }{2}} $};
\end{tikzpicture} \\
& \begin{tikzpicture}
    \draw (0,-1.5) -- (0,-2.5) node[midway, sloped] {\small$<$};
    \draw (0,-2.5) -- (-0.7,-3.2) node[midway, sloped] {\small$\ll$};
    \draw (0,-2.5) -- (1.5,-2.5) node[midway, sloped] {\small$\ll$};
    \draw (1.5,-1.5) -- (1.5,-2.5) node[midway, sloped] {\small$<$};
    \draw (1.5,-2.5) -- (2.2,-3.2) node[midway, sloped] {\small$\ll$};
    \node[circle,inner sep=1.5pt,fill=black] at (0,-2.5){};
    \node[circle,inner sep=1.5pt,fill=black] at (1.5,-2.5){};
    \node[below, anchor = south] at (0.3,-3.1) {\small$\mu$};
    \node[below, anchor = south] at (1.2,-3.1) {\small$\mu^{\mr{t}}$};
    \node[anchor = west] at (4,-2.5) {${\displaystyle\lim_{z\rightarrow 0}} ~\hat{E}_{\mu}\big(\boldsymbol{q} \cdot \sigma(z) ; q\big) = q_4^{\frac{|\mu|+\sqabs{\mu}}{2}} q_5^{\frac{|\mu|-\sqabs{\mu}}{2}} $};
\end{tikzpicture} \\
& \begin{tikzpicture}
    \draw (0,-4) -- (0,-5) node[midway, sloped] {\small$\gg$};
    \draw (0,-5) -- (-0.7,-5.7) node[midway, sloped] {\small$\ll$};
    \draw (0,-5) -- (1.5,-5) node[midway, sloped] {\small$<$};
    \draw (1.05,-4.1) -- (1.5,-5) node[midway, sloped] {\small$\gg$};
    \draw (1.5,-5) -- (2.35,-5.53) node[midway, sloped] {\small$\gg$};
    \node[circle,inner sep=1.5pt,fill=black] at (0,-5){};
    \node[circle,inner sep=1.5pt,fill=black] at (1.5,-5){};
    \node[below, anchor = south] at (0.3,-5.6) {\small$\mu$};
    \node[below, anchor = south] at (1.2,-5.6) {\small$\mu^{\mr{t}}$};
    \node[anchor = west] at (4,-5) {${\displaystyle\lim_{z\rightarrow 0}} ~\hat{E}_{\mu}\big(\boldsymbol{q} \cdot \sigma(z) ; q\big) = (-1)^{|\mu|} \,  q_4^{\frac{\sqabs{\mu^{\mr{t}}}}{2} } q_5^{-\frac{\sqabs{\mu}}{2}} \, f_{\mu}(q_4,q_5)^n$};
\end{tikzpicture}
\end{align*}
\endgroup
Here, the line bundle over the torus stable $\bP^1$ splits as $\cO_{\bP^1} \oplus \cO_{\bP^1}(-2)$ in the first two cases and as $\cO_{\bP^1}(-1-n) \oplus \cO_{\bP^1}(-1+n)$ in the last. The first two formulae have already been proven in \cite[Prop.~4.3]{Arb21:KthDTHilbSurf}. The third formula has only been proven in the special cases $n\in\{-1,0,1\}$ in \cite[Prop.~4.2 \& 4.3]{Arb21:KthDTHilbSurf}. Since the methods of loc.~cit. readily generalise to arbitrary $n$, we leave the verification of the last formula for general $n$ as an exercise to the reader.

\subsubsection{Rigidity.}
To apply the refined topological vertex to the computation of the \PTs\ of $\CYth_{N,m}$ we use the observation of Nekrasov and Okounkov that the only dependence of this series on the equivariant weights is trough $\kappa$ whenever the moduli space of stable pairs is proper \cite[Prop.~7.5]{NO14:membranes}. Since this is the case for the threefold $\CYth_{N,m}$ when $|m|<N$ we may compute the \PTs\ in any limit of a cocharacter $\sigma : \Gm \hookrightarrow \gpT'$ fixing the holomorphic threeform:
\begin{equation*}
    \PTarg{}{\CYth_{N,m}}{\CYth_{N,m}\times \Aaff{2}}{\Gm_q}{\gpT'} \big( \boldsymbol{q};q\big) = \lim_{z \rightarrow 0} \PTarg{}{\CYth_{N,m}}{\CYth_{N,m} \times \Aaff{2}}{\Gm_q}{\gpT'}{}\big(\boldsymbol{q}\cdot \sigma(z);q\big)\,.
\end{equation*}
We choose $\sigma$ in a way so that the slopes at each vertex are as illustrated in \Cref{fig: X N m pref dir}. In physics terminology this means to choose the preferred direction along the horizontal lines.
\begin{figure}
	\centering\begin{tikzpicture}
    \foreach \pos in {{4.8,2},{0.4,2},{6.3,4},{-0.3,4}}{
		\node[circle,inner sep=1.5pt,fill=black] at (\pos){};
	}
    \draw (4.8,2) -- (0.4,2) node[pos=0.2, anchor=south] {\small$\mu_2^{\mr{t}}$} node[pos=0.8, anchor=south] {\small$\mu_2$};
    \node at (3,2) {\small$<$};
    \draw[dashed] ($(4,0)!0.2!(4.8,2)$) -- ($(4,0)!0.4!(4.8,2)$);
    \draw ($(4,0)!0.4!(4.8,2)$) -- (4.8,2) node[near start, sloped] {\small$\ll$} node[pos = 0.45, anchor=west] {\small$\nu_2$};
    \draw[dashed] ($(0.4,2)!0.8!(0.4,0)$) -- ($(0.4,2)!0.6!(0.4,0)$);
    \draw ($(0.4,2)!0.6!(0.4,0)$) -- (0.4,2) node[near start, sloped] {\small$\ll$} node[pos = 0.5, anchor=east] {\small$\lambda_2^{\mr{t}}$};
    \draw (4.8,2) -- (6.3,4) node[midway, sloped] {\small$\ll$} node[pos=0.7, anchor=west] {\small$\nu_1$} node[pos=0.2, anchor=west] {\small$\nu_1^{\mr{t}}$};
    \draw (6.3,4) -- (-0.3,4) node[pos=0.2, anchor=south] {\small$\mu_1^{\mr{t}}$} node[pos=0.8, anchor=south] {\small$\mu_1$};
    \node at (3,4) {\small$<$};
    \draw (-0.3,4) -- (0.4,2) node[midway, sloped] {\small$\gg$} node[pos=0.8, anchor=east] {\small$\lambda_1$} node[pos=0.25, anchor=east] {\small$\lambda_1^{\mr{t}}$};
    \draw (6.3,4) -- (7.5,5) node[midway, sloped] {\small$\ll$};
    \draw (-0.3,4) -- (-1,5) node[midway, sloped] {\small$\gg$};
    \end{tikzpicture}
	\caption{The choice of $\sigma$. And the decoration of the GKM graph by partitions.}
	\label{fig: X N m pref dir}
\end{figure}
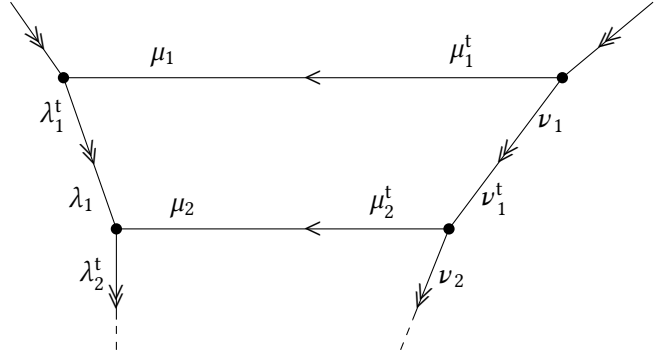

\subsubsection{Applying the vertex formalism.}
Now we decompose the \PTs\ into vertex and edge contributions using formula \eqref{eq: PT vertex formula no ins} and take the refined limit $z \rightarrow 0$. Then if we label the GKM graph by partitions $\boldsymbol{\lambda},\boldsymbol{\mu},\boldsymbol{\nu}$ as indicated in \Cref{fig: X N m pref dir}, the left vertical, right vertical and horizontal edges respectively contribute with edge terms
\begin{gather*}
    \prod_{a=1}^{N-1}Q^{|\lambda_a|[F_a]} q_4^{\frac{|\lambda_a| - \sqabs{\lambda_a^{\mr{t}}} }{2}} q_5^{\frac{|\lambda_a|+\sqabs{\lambda_a^{\mr{t}}}}{2}}\,,  \qquad \prod_{a=1}^{N-1}Q^{|\nu_a|[F_a]} q_4^{-\frac{|\nu_a|+\sqabs{\nu_a^{\mr{t}}}}{2}} q_5^{-\frac{|\nu_a|-\sqabs{\nu_a^{\mr{t}}}}{2}}\,, \\
    \prod_{a=1}^{N} \left(-Q^{[C_a]}\right)^{|\mu_a|} f_{\mu_a}(q_4,q_5)^{2a-N-m-1} \, q_4^{\frac{\sqabs{\mu_a^{\mr{t}}}}{2} } q_5^{-\frac{\sqabs{\mu_a}}{2}}\,.
\end{gather*}
The vertices contribute
\begin{align*}
    \prod_{a=1}^{N} C_{\lambda_{a-1},\lambda_a^{\mr{t}},\mu_a}\big(q_5^{-1} ,q_4\big) ~ C_{\nu_{a},\nu_{a-1}^{\mr{t}},\mu_a^{\mr{t}}} \big(q_4,q_5^{-1}\big)
\end{align*}
where we set $\lambda_0 = \lambda_N = \nu_0 = \nu_N = 0$. If we combine all terms and plug in the formula for the refined topological vertex \eqref{eq: ref top vert} we obtain
\begin{equation}
\label{eq: PT gauge before simplification}
\begin{split}
    &\qquad \PTarg{}{\CYth_{N,m}}{\CYth_{N,m}\times \Aaff{2}}{\Gm_q}{\gpT'}{} = \\
    & \sum_{\boldsymbol{\mu}} \prod_{a=1}^N \left[\left(-Q^{[C_a]}\right)^{|\mu_a|} f_{\mu_a}(q_4,q_5)^{2a-N-m-1} \, q_4^{\frac{\sqabs{\mu_a^{\mr{t}}}}{2} } q_5^{-\frac{\sqabs{\mu_a}}{2}} \, \tilde{Z}_{\mu_a}\big(q_4,q_5\big) \, \tilde{Z}_{\mu_a^{\mr{t}}}\big(q_5^{-1},q_4^{-1}\big) \right] \\
    & \qquad\times \sum_{\boldsymbol{\lambda}, \boldsymbol{\delta}} \prod_{a=1}^{N} \left((q_4q_5)^{-\frac{1}{2}} Q^{[F_a]}\right)^{|\lambda_a|} (q_4 q_5)^{\frac{|\delta_a|}{2}} \, s_{\lambda_{a-1}/\delta_a}\big(q_5^{\rho}q_4^{-\mu_a^{\mr{t}}}\big) ~ s_{\lambda_{a}/\delta_a}\big(q_4^{-\rho}q_5^{\mu_a}\big) \\
    & \qquad\times \sum_{\boldsymbol{\nu}, \boldsymbol{\eta}} \prod_{a=1}^{N} \left((q_4q_5)^{\frac{1}{2}} Q^{[F_a]}\right)^{|\nu_a|} (q_4 q_5)^{-\frac{|\eta_a|}{2}} \, s_{\nu_{a-1}/\eta_a}\big(q_5^{\rho}q_4^{-\mu_a^{\mr{t}}}\big) ~ s_{\nu_{a}/\eta_a}\big(q_4^{-\rho}q_5^{\mu_a}\big)\,.
\end{split}
\end{equation}

\subsubsection{Simplification.}
The last two lines can be simplified by using homogeneity of Schur functions and successively applying Cauchy's identity
\begin{equation*}
    \sum_{\lambda} s_{\lambda/\delta_1}(\mathbf{x}) ~ s_{\lambda/\delta_2}(\mathbf{y}) = \frac{1}{\prod_{i,j \geq 0} (1-x_i y_j)} \sum_{\lambda} s_{\delta_1/\lambda}(\mathbf{y}) ~ s_{\delta_2/\lambda}(\mathbf{x})
\end{equation*}
where $\mathbf{x}=(x_0,x_1,\ldots)$ and $\mathbf{y}=(y_0,y_1,\ldots)$. Concretely, the identity gives
\begin{equation*}
\begin{split}
    &\sum_{\boldsymbol{\lambda}, \boldsymbol{\delta}} \prod_{a=1}^{N} \left((q_4q_5)^{-\frac{1}{2}} Q^{[F_a]}\right)^{|\lambda_a|} (q_4 q_5)^{\frac{|\delta_a|}{2}} \, s_{\lambda_{a-1}/\delta_a}\big(q_5^{\rho}q_4^{-\mu_a^{\mr{t}}}\big) ~ s_{\lambda_{a}/\delta_a}\big(q_4^{-\rho}q_5^{\mu_a}\big) \\
    &\hspace{14em}= \prod_{1\leq a<b\leq N} \cR_{\mu_a^{\mr{t}}, \mu_b } \Big(q_4,q_5 \, ; \, q_5 {\textstyle\prod_{c=a}^{b-1}} Q^{[F_c]}\Big)
\end{split}
\end{equation*}
with a similar expression for the last line of \eqref{eq: PT gauge before simplification}. If we use this expression to simplify the right-hand side of equation \eqref{eq: PT gauge before simplification} we arrive at the desired formula \eqref{eq: PT gauge}. This closes the proof of \Cref{thm: gauge PT correspondence}.

\section{Case studies of local projective spaces}
\label{sec: local proj spaces}
In this section we will numerically investigate the relationship between the \GWs , the \PTs\ and \meminds\ of local projective spaces of the form
\begin{equation*}
    Z=\Tot{}_{\bP^{n}} \, \cL_{n+1} \oplus \ldots \oplus \cL_{5}\,,
\end{equation*}
under the assumption that all line bundles $\cL_i$ have non-positive degree. 

We have already discussed local rational curves ($n=1$) in \Cref{ex: Ar x C,ex: chain of O(-1)s} as a special instance of strip geometries with precisely one torus stable $\bP^1$. Another geometry which is already discussed extensively in the literature is local $\bP^2$, i.e.~$Z = \Tot{}_{\bP^{2}} \big(\cO(-3)\big) \times \Aaff{2}$. We will not recite any of the well-known features of this geometry here and refer the reader to the existing literature \cite{BS24:refGW,KPS23:RefBPS,IK17:RefTopStrKP2,KPS23:RefBPS,HK12:DiredtIntOmega,HKPK13:OmegaBmodRigid,CKK14:RefBPSStabPrs}.

This means the local geometries left to study are
\begin{gather*}
    \Tot{}_{\bP^2} \, \cO(-1) \oplus \cO(-1) \oplus \cO(-1) \,, \qquad \Tot{}_{\bP^2} \, \cO(-2) \oplus \cO(-1) \oplus \cO\,, \\[0.5em]
    \Tot{}_{\bP^3} \, \cO(-2) \oplus \cO(-2) \,, \qquad \Tot{}_{\bP^3} \, \cO(-3) \oplus \cO(-1) \,, \qquad \Tot{}_{\bP^3} \, \cO(-4) \oplus \cO \,.
\end{gather*}
Among these, only the left two geometries were previously discussed in the literature, however, under strong restrictions on the torus action \cite{Sch26:CY5TV,PZ10:CY5}. We defer the discussion of $\Tot{}_{\bP^2} \cO(-2) \oplus \cO(-1) \oplus \cO$ to \Cref{sec: CY4 vanishing} and only mention here that conjecturally all its associated enumerative invariants vanish. This vanishing is special instance of a more generally occurring curious vanishing for local surfaces.

\subsection{\texorpdfstring{$\boldsymbol{\Tot{}_{\bP^2}\,\cO(-1)^{\oplus 3}}$}{Tot O(-1)\^{}3 -> P2}}
As our first example from our list, we consider
\begin{equation*}
    \CYfi = \Tot{}_{\bP^2} \, \cO(-1) \oplus \cO(-1) \oplus \cO(-1)\,.
\end{equation*}

\subsubsection{Geometric setup}
In order to fix a particularly convenient basis for the equivariant variables let us present $\CYfi$ as a quotient:
\begin{equation*}
    \CYfi = \Aaff{6} \setminus V(x_0 x_1 x_3) \Big/ \Gm
\end{equation*}
Here, the one-dimensional torus acts on affine space via
\begin{equation*}
    \big( t , (x_0,\ldots, x_5) \big) \longmapsto (t x_0, t x_1, t x_2, t^{-1} x_3, t^{-1} x_4, t^{-1} x_5)\,.
\end{equation*}
Let us denote by $2\epsilon_0,\ldots,2\epsilon_5$ the tangent weights of the natural $\Gm[6]$-action on affine six-space. This action descends to the quotient $\CYfi$. The tangent weights at the fixed points
\begin{equation*}
    [1:0:0]\,,\qquad [0:1:0]\,,\qquad [0:0:1]
\end{equation*}
of the zero-section $\bP^2 \subset \CYfi$ read
\begin{align*}
    \big(2\epsilon_1 - 2\epsilon_0 \,, && 2\epsilon_2 - 2\epsilon_0 \,, && 2\epsilon_3 + 2\epsilon_0 \,, && 2\epsilon_4 + 2\epsilon_0 \,, && 2\epsilon_5 + 2\epsilon_0\big) \,,\\
    \big(2\epsilon_0 - 2\epsilon_1 \,, && 2\epsilon_2 - 2\epsilon_1 \,, && 2\epsilon_3 + 2\epsilon_1 \,, && 2\epsilon_4 + 2\epsilon_1 \,, && 2\epsilon_5 + 2\epsilon_1\big) \,,\\
    \big(2\epsilon_0 - 2\epsilon_2 \,, && 2\epsilon_1 - 2\epsilon_2 \,, && 2\epsilon_3 + 2\epsilon_2 \,, && 2\epsilon_4 + 2\epsilon_2 \,, && 2\epsilon_5 + 2\epsilon_2\big)\,\phantom{,}
\end{align*}
respectively. A sketch of the GKM graph of $\CYfi$ can be found in \Cref{fig: P2 O min 1 GKM graph}. Finally, let $\gpT\cong \Gm[5]$ be a subtorus of $\Gm[6]$ satisfying the Calabi--Yau condition
\begin{equation*}
    \epsilon_0 + \epsilon_1 + \epsilon_2 + \epsilon_3 +\epsilon_4 + \epsilon_5 = 0\,.
\end{equation*}

\begin{figure}
	\centering
    \begin{tikzpicture}[scale = 1.3]
    \node[circle,inner sep=1.5pt,fill=black] at (0,0){};
    \node[circle,inner sep=1.5pt,fill=black] at (3,0){};
    \node[circle,inner sep=1.5pt,fill=black] at (0,3){};
    \draw (0,0) -- (3,0) node[pos = 0.25, anchor=south, sloped] {\tiny$2\epsilon_1 - 2\epsilon_0$} node[pos = 0.7, anchor=south, sloped] {\tiny$2\epsilon_0 - 2\epsilon_1$};
    \draw (3,0) -- (0,3) node[pos = 0.3, anchor=north, sloped] {\tiny$2\epsilon_2 - 2\epsilon_1$} node[pos = 0.7, anchor=north, sloped] {\tiny$2\epsilon_1 - 2\epsilon_2$};
    \draw (0,0) -- (0,3) node[pos = 0.25, anchor=north, sloped] {\tiny$2\epsilon_2 - 2\epsilon_0$} node[pos = 0.7, anchor=north, sloped] {\tiny$2\epsilon_0 - 2\epsilon_2$};
    \draw (0,0) -- (-1.06,-1.06) node[pos = 0.4, anchor=south, sloped] {\tiny$2\epsilon_3 + 2\epsilon_0$};
    \draw (3,0) -- (3,-1.5) node[pos = 0.4, anchor=south, sloped] {\tiny$2\epsilon_3 + 2\epsilon_1$};
    \draw (0,3) -- (-1.5,3) node[pos = 0.4, anchor=south, sloped] {\tiny$2\epsilon_3 + 2\epsilon_1$};
    \draw (0,0) -- (-1.06,1.06) node[pos = 0.45, anchor=north, sloped] {\tiny$2\epsilon_4 + 2\epsilon_0$};
    \draw (3,0) -- (3,1.5) node[pos = 0.45, anchor=north, sloped] {\tiny$2\epsilon_4 + 2\epsilon_1$};
    \draw (0,3) -- (-0.67,4.34) node[pos = 0.45, anchor=south, sloped] {\tiny$2\epsilon_4 + 2\epsilon_2$};
    \draw (0,0) -- (1.06,-1.06) node[pos = 0.4, anchor=north, sloped] {\tiny$2\epsilon_5 + 2\epsilon_0$};
    \draw (3,0) -- (4.34,-0.67) node[pos = 0.4, anchor=south, sloped] {\tiny$2\epsilon_5 + 2\epsilon_1$};
    \draw (0,3) -- (1.5,3) node[pos = 0.4, anchor=south, sloped] {\tiny$2\epsilon_5 + 2\epsilon_2$};
    \end{tikzpicture}
	\caption{The GKM graph of $\CYfi = \Tot{}_{\bP^2} \, \cO(-1) \oplus \cO(-1) \oplus \cO(-1)$.}
	\label{fig: P2 O min 1 GKM graph}
\end{figure}

\subsubsection{Low-degree membrane indices}
To state conjectural low-degree formulae for \meminds\ of $\CYfi$ it is convenient to introduce some more notation: Let us write
\begin{align*}
     e_1 &= q_0 + q_1 + q_2\,,  &e_2 &= q_0 q_1 +q_0  q_2+ q_1 q_2\,, & e_3 &= q_0 q_1 q_2 \\
     \sigma_1 &= q_3 + q_4 + q_5\,,  &\sigma_2 &= q_3 q_4 +q_3  q_5+ q_4 q_5\,, & \sigma_3 &= q_3 q_4 q_5
\end{align*}
for elementary symmetric polynomials in the variables $q_i = \re^{u \epsilon_i}$. Moreover, for $n \in \bZ_{>0}$ let us denote by
\begin{equation*}
    \qnum{n}{q} \coloneqq \frac{q^{n/2}-q^{-n/2}}{q^{1/2}-q^{-1/2}}
\end{equation*}
the $q$-number $n$. Finally, we introduce the shorthand
\begin{equation*}
    \ol{q}_0 = q_0^{-1} q_1\,,\qquad\ol{q}_1 = q_1^{-1} q_2\,,\qquad \ol{q}_2 = q_2^{-1} q_0\,.
\end{equation*}
Then writing $H$ for the class of a line in $\bP^2$, we predict the following formulae for the \meminds\ of $\CYfi$ in low degree:
\begin{align}
    \label{eq: P2 O min 1 meminds} 
    \Omega_{H} & = - \prod_{i=0}^2\frac{1}{\qnum{2}{\ol{q}_i}} \,, \nonumber\\
    \Omega_{2 H} & = \prod_{i=0}^2\frac{1}{\qnum{2}{\ol{q}_i}\,\qnum{4}{\ol{q}_i} \,\qnum{2}{\ol{q}_i^2 \ol{q}_{i+1}} }  \nonumber \\[0.3em]
    & \quad \times\big(8+4 e_1 \sigma_1+3 e_1^2 \sigma_1^2-2 e_2 \sigma_1^2-6 e_1^2 \sigma_2+4 e_2 \sigma_2-e_1^4 \sigma_2^2+e_1^2 e_2 \sigma_2^2 +3 e_2^2 \sigma_2^2 \nonumber \\
    & \qquad -2 e_1 \sigma_2^2 \sigma_3^{-1} +9 e_1^3 \sigma_3-30 e_1 e_2 \sigma_3+2 e_1^4 \sigma_1 \sigma_3-2 e_1^2 e_2 \sigma_1 \sigma_3 -6 e_2^2 \sigma_1 \sigma_3  \nonumber \\
    & \qquad -2 e_1^3 e_2 \sigma_1^2 \sigma_3 +e_1 e_2^2 \sigma_1^2 \sigma_3+4 e_1^3 e_2 \sigma_2 \sigma_3 -2 e_1 e_2^2 \sigma_2 \sigma_3+e_1^3 e_2^2 \sigma_2^2 \sigma_3 -2 e_1 e_2^3 \sigma_2^2 \sigma_3 \nonumber \\
    & \qquad +e_1^6 \sigma_3^2-9 e_1^4 e_2 \sigma_3^2 +18 e_1^2 e_2^2 \sigma_3^2+9 e_2^3 \sigma_3^2-2 e_1^3 e_2^2 \sigma_1 \sigma_3^2 +4 e_1 e_2^3 \sigma_1 \sigma_3^2 \\
    & \qquad +e_1^2 e_2^3 \sigma_1^2 \sigma_3^2-e_2^4 \sigma_1^2 \sigma_3^2-2 e_1^2 e_2^3 \sigma_2 \sigma_3^2+2 e_2^4 \sigma_2 \sigma_3^2+2 e_1^3 e_2^3 \sigma_3^3-9 e_1 e_2^4 \sigma_3^3+e_2^6 \sigma_3^4\big)\,, \nonumber \\[0.8em]
    \Omega_{3H} &= \prod_{i=0}^2 \frac{\qnum{3}{\ol{q}_i}}{\qnum{2}{\ol{q}_i}\,\qnum{4}{\ol{q}_i}\,\qnum{6}{\ol{q}_i} \,\qnum{2}{\ol{q}_i^2 \ol{q}_{i+1}}   \qnum{2}{\ol{q}_i^3 \ol{q}_{i+1}}  \qnum{2}{\ol{q}_i^3 \ol{q}_{i+1}^2}} ~ \times  ~ \big( \cdots \big)\,, \nonumber \\[0.8em]
    \Omega_{4H} & = \cdots \nonumber
\end{align}
Here, all indices of $\ol{q}_i$ are understood modulo three. Our conjectural formulae for the numerator of the degree three and the degree four \memind\ were unfortunately too long to print here. However, they follow the same pattern: The numerator is a Laurent polynomial in the elementary symmetric polynomials and the denominator is a combination of $q$-numbers. The formulae can be found alongside our code in \cite{website}.


We stress that in accordance with \Cref{conj: gen GV,conj: PT GV}, all \meminds\ at worst develop denominators of two under the restriction $q_i \rightarrow q_{i+1}$, i.e.~$\ol{q}_i \rightarrow 1$, of the equivariant variables. Note that already in degree three this requires a non-trivial cancellation of odd factors.

\subsubsection{Comparison with GW invariants}
Under a specialisation of the torus weights, the low-degree \meminds\ of $\CYfi$ have already been computed by the second author. These calculations are in perfect agreement with the formulae \eqref{eq: P2 O min 1 meminds}.

\begin{prop} \cite[Thm. B \& Sec.~2.4]{Sch26:CY5TV}
     In the specialisation $\epsilon_3=-\epsilon_0+\epsilon_1-\epsilon_2$, $\epsilon_4=-\epsilon_0-\epsilon_1+\epsilon_2$ \Cref{conj: gen GV} holds for $\Tot{}_{\bP^2} \,\cO(-1)^{\oplus 3}$ in any degree $\beta = d H$. Moreover, for $d\leq 4$ the \meminds\ agree with \eqref{eq: P2 O min 1 meminds}. 
\end{prop}

Beyond this restriction on the torus action we can provide the following numerical evidence backing our conjectural formulae in \eqref{eq: P2 O min 1 meminds}.

\begin{evidence}
    \label{ev: P2 O min 1}
    Comparing with GW invariants, we verified the conjectured formulae for $\Omega_{dH}$ in \eqref{eq: P2 O min 1 meminds} up to certain genera, as summarized in \Cref{tab: evidence for P2_111}.
    This may be reproduced using the code in \cite{website}.

    \begin{table}
        \centering
        \def\arraystretch{1.5}
        \begin{tabular}{r||c|c|c|c}
            \textbf{Degree $\boldsymbol{d}$} &
            1&2&3&4
            \\\hline
            \textbf{Verified up to genus} & 
            5 & 4 & 3 & 3
        \end{tabular}
        \caption{Evidence for \Cref{conj: gen GV} regarding \meminds\ $\Omega_{dL}$ of $\Tot{}_{\bP^2}\,\cO(-1)^{\oplus 3}$.}
        \label{tab: evidence for P2_111}
    \end{table}
\end{evidence}

The GW theory of $\CYfi$ shares many features with that of the resolved conifold. For example in forthcoming work \cite{DS26:LocalPn}, D\"{a}nzer and the second author conjecture that the fixed-genus all-degree generating series of GW invariants of $\CYfi$ is a linear combination of polylogarithms. Based on the formulae \eqref{eq: P2 O min 1 meminds} one may for instance expect the following equality in genus two and three:
\begingroup
\allowdisplaybreaks
\begin{align*}
    & \hspace{-2em}\sum_{d>0} Q^d ~ \GWarg{\CYfi}{\gpT}{2,d H} = \\[-0.5em]
    &\frac{3}{64} (\ol{e}_2-\ol{\sigma}_2)\, \mr{Li}_{-2}(-Q) +  \frac{1}{64}\big(3(\ol{e}_2+\ol{\sigma}_2) - 2\ol{e}_1^2\big)\, \mr{Li}_{-1}(-Q)\,, \\[3.5em]
    &\hspace{-2em} \sum_{d>0} Q^d ~ \GWarg{\CYfi}{\gpT}{3,d H} = \\[-0.5em]
    & \frac{21 }{2560} (\ol{e}_2-\ol{\sigma}_2)^2\, \mr{Li}_{-5}(-Q) + \frac{11 }{1536}(\ol{e}_2-\ol{\sigma}_2) \big(3 (\ol{e}_2+\ol{\sigma}_2)-2 \ol{e}_1^2\big)\, \mr{Li}_{-4}(-Q) \\ 
    + & \frac{1}{1536} \big(8 \ol{e}_1^4-24 \ol{e}_1^2 (\ol{e}_2+\ol{\sigma}_2)+23 \ol{e}_2^2+26 \ol{e}_2 \ol{\sigma}_2+23 \ol{\sigma}_2^2\big)\, \mr{Li}_{-3}(-Q) \\ 
    + & \frac{1}{1536} (\ol{e}_2-\ol{\sigma}_2) \big(3 (\ol{e}_2+\ol{\sigma}_2)-2 \ol{e}_1^2\big) \, \mr{Li}_{-2}(-Q) + \frac{1}{3840}(\ol{e}_2-\ol{\sigma}_2)^2 \, \mr{Li}_{-1}(-Q) \,.
\end{align*}
\endgroup
Here, we write
\begin{equation*}
    \ol{e}_1 = \epsilon_0 + \epsilon_1 + \epsilon_2\,, \qquad \ol{e}_2 = \epsilon_0 \epsilon_1 + \epsilon_0 \epsilon_2 + \epsilon_1 \epsilon_2\,, \qquad \ol{\sigma}_2 = \epsilon_3 \epsilon_4 + \epsilon_3 \epsilon_5 + \epsilon_4 \epsilon_5\,.
\end{equation*}
for elementary symmetric polynomials. Note that the right-hand side of this equation is a rational function in $Q$ that is invariant under
\begin{equation*}
    Q \longrightarrow Q^{-1} \qquad (\epsilon_0,\epsilon_1,\epsilon_2) \longleftrightarrow (\epsilon_3,\epsilon_4,\epsilon_5) \,.
\end{equation*}
Note that this symmetry is to be expected by the crepant resolution conjecture \cite{DS26:LocalPn}. However in contrast to the resolved conifold, the \GWs\ of $\CYfi$ does not exhibit a \textit{conifold gap} when expanded around $Q=1$. One may thus raise the following question.

\begin{question}
    \label{qu: P2 O min 1}
    Is there a substitute for the conifold gap property which alongside the finite generation and crepant transformation property characterises the all-degree generating series of $\Tot{}_{\bP^2} \,\cO(-1)^{\oplus 3}$ uniquely?
\end{question}

\subsubsection{Comparison with PT invariants}
To form the \PTs\ we first have to choose a one-dimensional subtorus of $\gpT$ whose fixed locus inside $\CYfi$ is a threefold containing the zero section $\bP^2$. Since $\CYfi$ is the total space of three identical line bundles over $\bP^2$ all possible choices are essentially equivalent after identifying the equivariant variables. We choose
\begin{equation*}
\begin{tikzcd}[row sep=0em]
    \Gm_q \times \gpT' \ar[r,hook] & \Gm[6]\,, \\
    \big(q,(q_0,q_1,q_2,q_3)\big) \ar[r,maps to] & (q_0^2,q_1^2,q_2^2,q_3^2,q ,(q_0 q_1 q_2 q_3)^{-2}q^{-1})
\end{tikzcd}
\end{equation*}
With this choice we have $\CYth = \CYfi^{\Gm_q} =\Tot{}_{\bP^2} \, \cO(-1)$. One can now numerically compute the \PTs\ in low curve degree via the methods described in \Cref{sec: PT numerics} and compare the results with our formulae for \meminds .

\begin{evidence}
    \label{ev: P2 O min 1 PT}
    \Cref{conj: PT GV} holds for $\CYfi = \Tot{}_{\bP^2} \,\cO(-1)^{\oplus 3}$ in degree $d\leq 4$ for $7$  
    orders beyond the leading order in $q$ with $\Omega_{dH}$ as given in equation \eqref{eq: P2 O min 1 meminds}. Together with \Cref{ev: P2 O min 1} this is compatible with the \GWPT\ (\Cref{conj: GW PT}). Moreover, confirming \Cref{conj: PT sym}, in the computed range the \PTs\ is symmetric under the permutation of
    \begin{equation}
        \label{eq: P2 O min 1 PT sym}
        q_3\,,\qquad q_4=q^{-1/2}\,, \qquad q_5=q^{1/2}q_0q_1q_2q_3 \,.
    \end{equation}
\end{evidence}

\begin{rmk}
    \label{rmk: PT to Omega pipeline}
    Let us briefly explain how we arrived at our guess for the \meminds\ \eqref{eq: P2 O min 1 meminds}. Actually, it was mostly informed by our numerical calculation of PT invariants described above. First note that all \meminds\ in equation \eqref{eq: P2 O min 1 meminds} are Laurent polynomials in $q$. This means one only has to compute a finite number of PT invariants per degree in order to determine $\Omega_{dH}$ uniquely. Whenever these calculations exceeded the capabilities of our computer we used the anticipated permutation symmetry in the variables \eqref{eq: P2 O min 1 PT sym} to fix the remaining terms.
\end{rmk}

\subsection{\texorpdfstring{$\boldsymbol{\Tot{}_{\bP^3}\, \cO(-a)\oplus \cO(a-4)}$}{Tot O(-a) + O(a-4) -> P3}}
Now let us consider
\begin{equation*}
    \CYfi = \Tot{}_{\bP^3}\big( \cO(-a)\oplus \cO(a-4)\big)\qquad a\in\{0,\ldots,4\}\,.
\end{equation*}

\subsubsection{Geometric setup}
As in the last section we present this fivefold as the quotient
\begin{equation*}
    \CYfi = \Aaff{6} \setminus V(x_0 x_1 x_3) \Big/ \Gm
\end{equation*}
where the one-dimensional torus acts on affine space via
\begin{equation*}
    \big( t , (x_0,\ldots, x_5) \big) \longmapsto (t x_0, t x_1, t x_2, t x_3, t^{-a} x_4, t^{a-4} x_5)\,.
\end{equation*}
As before we denote by $2\epsilon_0,\ldots,2\epsilon_5$ the tangent weights of the natural $\Gm[6]$-action on affine space. The action descends to the quotient $\CYfi$. The tangent weights at the fixed point $[1:0:0:0]\in\bP^3 \subset \CYfi$ read
\begin{align*}
    2\epsilon_1 - 2\epsilon_0 \,, && 2\epsilon_2 - 2\epsilon_0 \,, && 2\epsilon_3 - 2\epsilon_0 \,, && 2\epsilon_4 + 2a\epsilon_0 \,, && 2\epsilon_5 + 2(4-a)\epsilon_0
\end{align*}
with similar expressions for the other fixed points. Finally, let $\gpT\cong \Gm[5]$ be a subtorus of $\Gm[6]$ realising the Calabi--Yau condition
\begin{equation*}
    \epsilon_0 + \epsilon_1 + \epsilon_2 + \epsilon_3 +\epsilon_4 + \epsilon_5 = 0\,.
\end{equation*}

\subsubsection{Low-degree \meminds}
\begin{table}
    \centering
    \def\arraystretch{1.5}
    \begin{tabular}{c|l}
        $\boldsymbol{a}$ & $\boldsymbol{\Omega_{L}}$ \\ \hline
        $0$ & $(q_4^2+1+q_4^{-2})\big(2 e_1 e_3 e_4^{-1}-2 - e_2^2 e_4^{-1}\big)$ \\
        $1$ & $\big((e_1 e_3 e_4 - e_4^2) q_4^2 - e_2 e_4  + (e_1 e_3 - e_4)q_4^{-2} \big) \big/ \prod_{0\leq i < j \leq 3} (q_i+q_j)$ \\
        $2$ & $0$
    \end{tabular}
    \caption{The degree one membrane index of $\Tot{}_{\bP^3}\big( \cO(-a)\oplus \cO(a-4)\big)$ for $a\in\{0,1,2\}$.}
    \label{tab: loc P3 mem ind d1}
\end{table}
Based on numerical experiments we found conjectural expressions for \meminds\ $\Omega_{dL}$ in low degree $d$ where $L$ denotes the class of a line in $\bP^3$. The range in which we were able to guess these invariants can be inferred from \Cref{tab: evidence for PTot P3 -a a-4}. We refer the reader interested in this data to \cite{website} and only present few of the \meminds\ here. For instance in degree one we expect the invariants listed in \Cref{tab: loc P3 mem ind d1}. In the table, we write $e_i$ for the $i$th elementary symmetric polynomial in $q_0,q_1,q_2,q_3$.

The fact that the degree-one \memind\ for $a=2$ vanishes is a coincidence. Already in degree two one finds that $\Omega_{2L}=2$ and the expression for $\Omega_{3L}$ is already an intricate Laurent polynomial in the equivariant variables. One may wonder whether this feature persists in higher degree (cf.~\cite[Conj.~2.8]{Sch26:CY5TV}). 

Similarly, the feature that $\Omega_{L}$ is a Laurent polynomial when $a=0$ appears to hold in higher degree as well. In contrast, for $a=1$ \meminds\ appear to have a more intricate pole structure. This raises the following general question.

\begin{question}
    \label{qu: local P3}
    Is there a sufficient condition under which the \memind\ of a Calabi--Yau fivefold is a Laurent polynomial?
\end{question}

\subsubsection{Comparison with GW invariants}
For $a=2$, the low-degree \meminds\ of $\CYfi$ have been computed and compared with GW invariants in earlier work of the second author.

\begin{prop} \cite[Thm.~B \& Sec.~2.5]{Sch26:CY5TV}
    \Cref{conj: gen GV} holds for $\CYfi = \Tot{}_{\bP^3} \, \cO(-2)^{\oplus 2} $ in any degree $\beta = d L$ in the specialisation $\epsilon_4=-\epsilon_0-\epsilon_1$. For $d\leq 3$ the \meminds\ obey the general formulae found in \cite{website}.
\end{prop}

Beyond this limit we can provide the following numerical evidence for \Cref{conj: gen GV}.

\begin{evidence}
    \label{ev: local P3}
    We have verified the conjectured formulae for the \meminds\ $\Omega_{dL}$ of $\Tot{}_{\bP^3} \, \cO(-a)\oplus \cO(a-4)$ in all degrees and up to the genera shown in \Cref{tab: evidence for PTot P3 -a a-4}.
    This can be reproduced using the code provided in \cite{website}.

    \begin{table}
        \centering
        \def\arraystretch{1.5}
        \begin{tabular}{r||c|c|c||c|c||c|c|c}
        $\boldsymbol{a}$ &
        \multicolumn{3}{c||}{$\boldsymbol{0}$}
        &
        \multicolumn{2}{c||}{$\boldsymbol{1}$}
         &
        \multicolumn{3}{c}{$\boldsymbol{2}$}
        \\\hline
             \textbf{Degree $\boldsymbol{d}$} &
            1 & 2 & 3 &
            1 & 2 &
            1 & 2 & 3
            \\\hline
            \textbf{Verified up to genus } & 
            5 & 3 & 3 &
            5 & 5 &
            5 & 3 & 3
        \end{tabular}
        \caption{Evidence for \Cref{conj: gen GV} regarding the \meminds\ $\Omega_{dL}$ of $\Tot{}_{\bP^3} \, \cO(-a)\oplus \cO(a-4)$.}
        \label{tab: evidence for PTot P3 -a a-4}
    \end{table}
\end{evidence}

\subsubsection{Comparison with PT invariants}
We also checked our conjectural formulae for \meminds\ against numerical PT calculations.

\begin{evidence}
    \label{ev: local P3 PT}
    For all the degrees listed in \Cref{tab: evidence for PTot P3 -a a-4} the \PTs\
    \begin{equation*}
        \PTarg{dL}{\bP^3}{\Tot{}_{\bP^3}\cO(-a)\oplus \cO(a-4)}{\Gm_{q_4}}{\gpT'}
    \end{equation*}
    matches our conjectural expression for the \meminds\ up to $7$ orders beyond the leading order in $q$. Together with \Cref{ev: local P3} this is consistent with the \GWPT . Moreover verifying \Cref{conj: PT sym}, in the computed range the series is symmetric under exchanging the line bundles $\cL_4$ and $\cL_5$, that is, it is invariant under
    \begin{equation*}
        a, (q_4,q_5)\qquad \longleftrightarrow \qquad 4-a,(q_5,q_4)\,.
    \end{equation*}
    These results may be reproduced using the accompanying code \cite{website}.
\end{evidence}

\section{Curious vanishing for local surfaces}
\label{sec: CY4 vanishing}
In this section we will investigate local surfaces of the form
\begin{equation}
    \label{eq: CY4 times A1}
    \CYfi = \Tot{}_S ~ \cL_1 \oplus \cL_2 \oplus \cO_{S}
\end{equation}
where $S$ is a smooth projective surface and $\cL_1$, $\cL_2$ are sufficiently negative line bundles. Among others this includes the local projective space $\Tot{}_{\bP^2} \, \cO(-2) \oplus \cO(-1) \oplus \cO_{\bP^2} $ whose analysis we skipped earlier in \Cref{sec: local proj spaces}.

\subsection{Vanishing of the GW series}
Remarkably, under a properness assumption on the moduli space of stable maps to the fourfold
\begin{equation*}
    \CYfo = \Tot{}_S \, \cL_1 \oplus \cL_2
\end{equation*}
we experimentally find that the \GWs\ of $\CYfi$ vanishes.

\begin{conj}
    \label{conj: curious CY4 vanishing}
    Suppose $(\CYfi,\gpT)$ is equivariantly Calabi--Yau and $\beta$ is an effective curve class in $\CYfi$ so that $\Mbar_g(\CYfo,\beta)$ is proper for all $g\geq 0$. Then 
    \begin{equation*}
        \GWarg{\CYfi}{\gpT}{\beta}= 0\,.
    \end{equation*}
\end{conj}

\begin{rmk}
    The assumption of the \namecref{conj: curious CY4 vanishing} is for instance met when $\cL_i = \cO_S(-D_i)$ for $D_i$ a smooth curve with $D_i^2 >0$, $i\in\{1,2\}$.
\end{rmk}

\begin{rmk}
    The vanishing claimed in \Cref{conj: curious CY4 vanishing} is rather curious as the \GWs\ of the fourfold $\CYfo$ generally \textit{does not} vanish. See \cite{KP08:CY4} for explicit case studies. The additional affine direction $\Aone$ which leads to an insertion of $\lambda$-classes appears crucial for the vanishing to hold. The assumption that the $\gpT$-action is Calabi--Yau appears similarly crucial.
\end{rmk}

As with most of our conjectures, we can prove \Cref{conj: curious CY4 vanishing} under a restrictive assumption on the torus action.

\begin{prop}
    \label{prop: curious CY4 vanishing limit}
    \Cref{conj: curious CY4 vanishing} holds under the following assumptions: There are divisors $D_1,D_2$ with zero-dimensional intersection $D_1\cap D_2$ satisfying $\cL_i = \cO_S(-D_i)$, $i\in\{1,2\}$, and $D_2$ is a smooth curve. Moreover, $\gpT=\Gm$ acts trivially on $\Tot \,\cO_S(-D_1)$ and with opposite weights on the fibres of $\cO_{S}(-D_2)$ and $\cO_S$. Then the vanishing holds for all effective curve classes $\beta$ satisfying $\cL_i \cdot \beta < 0$, $i\in\{1,2\}$.
\end{prop}

We present the proof of this \namecref{prop: curious CY4 vanishing limit} at the end of this section. The assumptions of the \namecref{prop: curious CY4 vanishing limit} in particular apply to $\Tot_{\bP^2} \cO(-2) \oplus \cO(-1) \oplus \cO$. For this and other local surfaces we can provide the following numerical evidence for \Cref{conj: curious CY4 vanishing} without any assumption on the torus action.

\begin{evidence}
    \label{ev: curious CY4 vanishing}
    We have verified \Cref{conj: curious CY4 vanishing} up to genus 3 in the cases listed in \Cref{tab: evidence curious CY4 vanishing}.
    We write $H$ for the class of a line in $\bP^2$ and $(d_1,d_2)$ for the class $d_1[\mathrm{pt}\times\bP^1]+d_2[\bP^1\times\mathrm{pt}]$ on $\bP^1\times\bP^1$.
    In the case of the final row, we additionally observed that $\Omega_{(1,0)}\neq 0$, which confirms the necessity of the properness assumption in \Cref{conj: curious CY4 vanishing}.
    We provide code to reproduce this claim and \Cref{tab: evidence curious CY4 vanishing} in \cite{website}.
    
    \begin{table}[H]
        \centering
        \def\arraystretch{1.5}
        \begin{tabular}{c|c|c|c}
            $\boldsymbol{S}$ & $\boldsymbol{\cL_1}$ & $\boldsymbol{\cL_2}$ & \textbf{Verified curve classes}
            \\\hline\hline
            $\bP^2$ & $\cO(-1)$ & $\cO(-2)$ & 
            $H, 2H, 3H, 4H$
            \\\hline
            $\bP^1\times\bP^1$ & $\cO(-1,-1)$ & $\cO(-1,-1)$ &
            $(1,0), (2,0), (3,0), (4,0), (1,1), (2,1), (2,2), (3,1)$
            \\\hline
            $\bP^1\times\bP^1$ & $\cO(-2,0)$ & $\cO(0,-2)$ &
            $(1,1), (2,1), (2,2),(3,1)$
            \\\hline
        \end{tabular}
        \caption{Evidence for \Cref{conj: curious CY4 vanishing}.}
        \label{tab: evidence curious CY4 vanishing}
    \end{table}
\end{evidence}

\subsection{Vanishing of the PT series}
By the \GWPT , from \Cref{conj: curious CY4 vanishing} we should expect that the \PTs\ of local surfaces of the form \eqref{eq: CY4 times A1} vanishes as well.

\begin{conj}
    \label{conj: curious CY4 vanishing PT}
    Suppose $(\CYfi,\gpT)$ is Calabi--Yau and let $\beta$ be an effective curve class in $\CYfi$ so that the moduli space \smash{$\Mbar{}_g^\bullet(\CYfo,\beta)$} of stable maps to the fourfold with possibly disconnected domain and no contracted connected components is proper for all $g\geq 0$. Then 
    \begin{equation*}
        \PTarg{\beta}{\CYth}{\CYfi}{\Gm_q}{\gpT'} = 0\,.
    \end{equation*}
\end{conj}

We can provide the following numerical evidence for the vanishing claimed.

\begin{evidence}
    \label{ev: curious CY4 vanishing PT}
    We checked the vanishing of the \PTs\ of
    \begin{equation*}
        \Tot{}_{\bP^2} \, \cO(-1) \oplus \cO(-2) \oplus \cO \qquad \text{and} \qquad \Tot{}_{\bP^1\times \bP^1} \, \cO(-1,-1) \oplus \cO(-1,-1) \oplus \cO
    \end{equation*}
    up to $6$ orders beyond leading order in $q$ and degree $d\leq 4$ respectively bi-degree $(d_1,d_2)$ with $d_1+d_2 \leq 4$. We checked the vanishing for all six respectively three distinct choices of threefolds $\CYth \subset \CYfi$. Note that this vanishing is also non-trivial evidence for the symmetry conjecture (\Cref{conj: PT sym}).

    Regarding the fivefold $\Tot{}_{\bP^1\times \bP^1} \,\cO(-2,0) \oplus \cO(0,-2) \oplus \cO$, note that curve classes of bi-degree $(d_1,0)$ and $(0,d_2)$ do not satisfy the properness criterion of \Cref{conj: curious CY4 vanishing PT}. And indeed, we find that the PT invariants in these curve classes are non-zero. However, we conjecture the \PTs\ takes the following simple form
    \begin{equation*}
    \begin{split}
        &\qquad \PTarg{}{\CYth}{\Tot{}_{\bP^1\times \bP^1} \cO(-2,0) \oplus \cO(0,-2) \oplus \cO}{\Gm_{q}}{\gpT'} = \\
        & \Exp\left(
        Q_1 \frac{q_2 q_4 q_5(1-q_1 q_3)(1+q_2q_4)}{(1-q_4)(1-q_2 q_4)(1-q_5)} +Q_2 \frac{q_1 q_3 q_5(1-q_2 q_4)(1+q_1q_3)}{(1-q_3)(1-q_1 q_3)(1-q_5)}\right)\,.
    \end{split}
    \end{equation*}
    Here, $q_1$ and $q_2$ denote the tangent weights at $(0,0)\in\bP^1 \times \bP^1$ and we write $q_3,q_4,q_5$ for the tangent weights on the fibres over this point. We verified this formula numerically up to $6$ orders beyond leading order in $q$ in bi-degree $(d_1,d_2)$ with $d_1+d_2 \leq 4$. This especially means that consistent with the \GWPT\ the \textit{connected} PT invariants in the bi-degrees listed in \Cref{tab: evidence curious CY4 vanishing} do vanish in low order in $q$.
\end{evidence}

\subsection{Proof of \texorpdfstring{\Cref{prop: curious CY4 vanishing limit}}{Proposition 5.4}}
We will prove the vanishing of the \GWs\ under the additional assumptions stated in \Cref{prop: curious CY4 vanishing limit} via a degeneration and dimension count argument. More precisely, we will perform a degeneration to the normal cone of $D_2$ in $S$. Recall that by assumption $D_2$ is a smooth curve in $S$ satisfying $\cL_2 \cong \cO_S(-D_2)$. Similarly, let us choose a divisor $D_1 \neq D_2$ (not necessarily assumed to be smooth) satisfying $\cL_1 \cong \cO_S(-D_1)$.

\subsubsection{A formula for the virtual fundamental class}
Under the degeneration to the normal cone of $D_2$ in $S$ we are going to relate the virtual fundamental class of the moduli space of stable maps to $\CYfi$
\begin{equation*}
    \Mbar_{g,n}(\CYfi,\beta) = \Mbar_{g,n}(S,\beta) \times \Aone
\end{equation*}
to the moduli space of stable log maps to the threefold
\begin{equation*}
    \CYth \coloneqq \Tot \,\cO_S(-D_1)
\end{equation*}
relative to the smooth divisor $D \coloneqq \Tot \cO_S(-D_1)|_{D_2}$. We denote this moduli space by
\begin{equation*}
    \Mbar_{g,n,\boldsymbol{d}}(\CYth \mid D, \beta)
\end{equation*}
where $n$ denotes the number of interior markings and $\boldsymbol{d}$ is an ordered partition of $D_2\cdot \beta$ indicating the tangency order of contact markings mapping to $D$. Further, let us write
\begin{equation*}
    \phi:\Mbar_{g,n,\boldsymbol{d}}(\CYth \mid D, \beta) \longrightarrow \Mbar_{g,n}(\CYth, \beta)
\end{equation*}
for the morphism that forgets the log structure and the contact markings. We will show the following relation.

\begin{prop}
    \label{prop: curious CY4 vanishing virt classes}
    Suppose, there are divisors $D_1,D_2$ with zero-dimensional intersection $D_1\cap D_2$ satisfying $\cL_i = \cO_S(-D_i)$, $i\in\{1,2\}$. Further suppose that $D_2$ is a smooth curve. Let $\Gm$ act trivially on $\CYth$ and with opposite weights $\pm\epsilon$ on the fibres of $\cL_2$ and $\cO_S$. Then for all effective curve classes $\beta$ with $\cL_i \cdot \beta < 0$, $i\in\{1,2\}$, we have
    \begin{equation*}
        [\Mbar_{g,n}(\CYfi,\beta)]^{\vir}_{\Gm} = \sum_{m > 0} \epsilon^{2g-1+m} \sum_{\substack{d_1,\ldots,d_m >0 \\ \sum_i d_i = D_2\cdot \beta}} \frac{(-1)^{D_2\cdot \beta + g -1}}{m! \prod_i d_i} \big( \phi_* [\Mbar_{g,n,\boldsymbol{d}}(X \mid D, \beta)]^{\vir} \big) \times [\Aone]_{\Gm}\,.
    \end{equation*}
\end{prop}

Setting $n=0$ and pushing this identity forward to a point we get
\begin{equation*}
    \GWarg{\CYfi}{\gpT}{g,\beta} = \sum_{m > 0} \epsilon^{2g-2+m} \sum_{\substack{d_1,\ldots,d_m >0 \\ \sum_i d_i = D_2\cdot \beta}} \frac{(-1)^{D_2\cdot \beta + g -1}}{m! \prod_i d_i} \int_{[\Mbar_{g,0,\boldsymbol{d}}(X \mid D, \beta)]^{\vir}}1 \,.
\end{equation*}
The vanishing $\GWarg{\CYfi}{\gpT}{g,\beta}=0$ claimed in \Cref{prop: curious CY4 vanishing limit} then follows from the observation that the right-hand side integrand has dimension $m>0$. Thus, each integral vanishes for dimension reasons.

\subsubsection{Degeneration}
To prove \Cref{prop: curious CY4 vanishing limit} it therefore remains to show \Cref{prop: curious CY4 vanishing virt classes}. We will establish the \namecref{prop: curious CY4 vanishing virt classes} via a degeneration to the normal cone argument. All arguments will almost be verbatim to the ones in the proof of \cite[Thm.~4.1]{BS24:refGW} which concerned the situation $D_2=0$. For this reason we will skip most details and focus on highlighting the differences to loc.~cit. 

Consider the degeneration to the normal cone of $D_2$ in $S$ as a family $\cS$ over $\Aone_t$. The general fibre of this family is $\cS_{t\neq 0} =S$ while the special fibre $\cS_0$ has two irreducible components: $S$ and $P=\bP_{D_2}(\cO \oplus N_{D_2}S)$ glued along $D_2\subset S$ and the zero section in $P$. We extend this to a two-point degeneration of $\CYfi$ by specifying how the divisors $D_1,D_2$ extend to the family. We choose $\cD_1$ to be the pre-image of $D_1 \times \Aone$ under $\cS \rightarrow S\times \Aone$ and $\cD_2$ as the closure of $D_2 \times \Aone\setminus\{0\}$ in $\cS$. The latter intersects the special fibre $\cS_0$ in the infinity section $D_\infty \subset P$ while the restriction of $\cD_1$ to $P$ is supported on fibres of $P\rightarrow D_2$. This is where we use the assumption that the intersection $D_1 \cap D_2$ is zero-dimensional.

This provides us with a two-point degeneration of $\CYfi$:
\begin{equation*}
    \cZ = \Tot \, \cO_{\cS}(-\cD_1) \oplus \cO_{\cS}(-\cD_2) \oplus \cO_{\cS}\,.
\end{equation*}
We illustrate this setup in \Cref{fig: degeneration}. The total space $\cZ$ contains the family $\cS$ as the zero section. In between sits $\cX = \Tot \, \cO_{\cS}(-\cD_1)$. One component of its special fibre is what we denoted $X=\Tot \, \cO_{S}(-D_1)$ before and the other one we are going to denote $Q= \Tot \, \cO_{\cS}(-\cD_1)|_{P}$. Later the observation that the two components are glued along a surface $$D \coloneqq X \cup Q \cong \Tot \, \cO_{\bP^1}(-2)$$ that admits a holomorphic symplectic form will be crucial.

Of course, on the level of moduli spaces there is essentially no difference between the moduli space of stable maps to $\CYfi$, $\CYth$ or $S$ by the properness assumption of \Cref{conj: curious CY4 vanishing}:
\begin{equation*}
    \Mbar_{g,n}(Z,\beta) = \Mbar_{g,n}(X,\beta) \times \Aone = \Mbar_{g,n}(S,\beta)\times \Aone\,.
\end{equation*}
However, for later reference we record that virtual fundamental classes do differ by an explicit insertion:
\begin{equation}
    \label{eq: virt Z vs X}
    [\Mbar_{g,n}(Z,\beta)]^{\vir}_{\Gm} = \HodgeLambda[g]{\epsilon} \,e^{\Gm}\!(\mathbf{R}\pi_*f^*\cO_{S}(-D_2))\cap \left([\Mbar_{g,n}(X,\beta)]^{\vir} \times [\Aone]_{\Gm}\right)\,.
\end{equation}

\begin{figure}
    \centering
    \begin{tikzpicture}
        \node[anchor=east] at (-0.5,1) {$\cS$:};
        \node[anchor=east, code] at (-0.5,3.5) {$\cO_{\cS}(-\cD_1)$:};
        \node[anchor=east, purple] at (-0.5,4.3) {$\cO_{\cS}(-\cD_2)$:};
        \node[anchor=east] at (-0.5,5.1) {$\cO_{\cS}$:};
        \draw (4,2) -- (1,2) -- (0,0) -- (3,0);
        \draw[code] (3.125,0.25) .. controls (3.125,0.25) and (0.4,0.2) .. (0.8,1)  .. controls (1.2,1.8) and (3.875,1.75) .. (3.875,1.75);
        \draw[purple] (3,0) -- (4,2) node[midway, left, purple] {$D_2$};
        \node at (2,1) {$S$};
        \node[code] at (0.8,0.3) {$D_1$};
        \node[code] at (2,3.5) {$\cO_S(-D_1)$};
        \node[purple] at (2,4.3) {$\cO_S(-D_2)$};
        \node at (2,5.1) {$\cO_S$};
        \draw[-stealth] (2,3) -- (2,2.3);
        \node at (4.2,0.8) {\LARGE$\rightsquigarrow$};
        \draw (5,1) -- (8,-1) -- (9,1) -- (6,3) -- (5,1);
        \draw (8,-1) -- (9,1);
        \draw (9,1) -- (11,3);
        \draw (10,1) -- (8,-1);
        \draw[purple] (11,3) -- (10,1);
        \draw[code] (8.125,-0.75) .. controls (8.125,-0.75) and ($(5.1,1.2) + (0.21,-0.14)$) .. ($(5.35,1.9) + (0.21,-0.14)$)  .. controls ($(5.9,2.8) + (0.21,-0.14)$) and (8.875,0.75) .. (8.875,0.75);
        \draw[code] (8.125,-0.75) -- (10.125,1.25);
        \draw[code] (8.875,0.75) -- (10.875,2.75);
        \node at (7.3,0.8) {$S$};
        \node[code] at (7.3,3.5) {$\cO_S(-D_1)$};
        \node[purple] at (7.3,4.3) {$\cO_S$};
        \node at (7.3,5.1) {$\cO_S$};
        \draw[-stealth] (7.3,3) -- (7.3,2.3);
        \node[code] at (5.5,1) {$D_1$};
        \node[code] at (9.5,0) {$F_1$};
        \node[code] at (10.1,2.6) {$F_2$};
        \node at (9.5,1) {$P$};
        \node[code] at (9.5,3.5) {$\cO_P(-F_1-F_2)$};
        \node[purple] at (9.5,4.3) {$\cO_P(-D_{\infty})$};
        \node at (9.5,5.1) {$\cO_P$};
        \draw[-stealth] (9.5,3) -- (9.5,2.3);
        \node at (8.3,0.25) {$D_0$};
        \node[purple] at (11,1.9) {$D_\infty$};
    \end{tikzpicture}
    \caption{The degeneration $\cZ \to \Aone$.}
    \label{fig: degeneration}
\end{figure}

\subsubsection{Degeneration formula}
Using the degeneration formula for stable log maps \cite[Thm.~1.5]{KLR23:DegFormulaLogMaps} (or equivalently for relative stable maps \cite{Li02:DegenFormula}) we may decompose the virtual fundamental class of \smash{$\Mbar_{g,n}(\CYth,\beta)$} in terms of stable log maps to the irreducible components of $\cX_0$. This decomposition is labelled by bi-partite graphs $\Gamma$ recording the combinatorial data of such stable log maps. The vertices of $\Gamma$ are separated into $\CYth$ and $Q$-vertices and are decorated with a genus label, a curve class and a subset of markings $\{1,\ldots,n\}$. The edges $e$ of $\Gamma$ are decorated with integers $d_e >0$ indicating a tangency order of a marking mapping to $D = X \cap Q$. All decorations must be compatible with the data $(g,n,\beta)$ in some natural way.

Now for an $\CYth$-vertex $v$ we denote by $\Mbar_v$ the moduli space of stable log maps to $(X\mid D)$ with genus, curve class, interior and contact markings as prescribed by the star of $\Gamma$ at $v$ with a similar notation for $Q$-vertices. Each moduli space \smash{$\Mbar_v$} comes with evaluation maps to $D$ that are associated with half-edges. We glue the moduli spaces along these evaluation maps. So denote by $\Delta:D^{E(\Gamma)}\rightarrow D^{E(\Gamma)} \times D^{E(\Gamma)}$ the inclusion of the diagonal. The degeneration formula then yields the following decomposition (cf. \cite[Eq.~(4.5)]{BS24:refGW}):
\begin{equation}
    \label{eq: deg formula}
    [\Mbar_{g,n}(Z,\beta)]^{\vir}_{\Gm} = \sum_{\Gamma} (-\epsilon^2)^{|E(\Gamma)|} \frac{\prod_e d_e}{|E(\Gamma)|!} \phi'_* \Delta^!\left({\textstyle \prod_v} C_v \cap [\Mbar_v]^{\vir}\right)\times [\Aone]_{\Gm}\,.
\end{equation}
Here, $\phi'$ is a morphism that composes a glued stable log map to $\cS_0$ with the projection $\cS_0\rightarrow S$, stabilises the domain and forgets the log structure. Moreover, we write $C_v$ for the cohomology classes
\begin{equation*}
    C_v = \begin{cases}
        (-\epsilon^2)^{g-1} & v\text{ is an } X\text{-vertex}\\
        \HodgeLambda[g]{\epsilon} \,e^{\Gm}\!(\mathbf{R}\pi_*f^*\cO_{S}(-2F)) & v\text{ is a } Q\text{-vertex}
    \end{cases}
\end{equation*}
These classes are the result of distributing the insertion on the right-hand side of \eqref{eq: virt Z vs X} over $\CYth$ and $Q$-vertices.

\subsubsection{Vanishing}
The key observation to prove \Cref{prop: curious CY4 vanishing virt classes} is that most terms in the degeneration formula \eqref{eq: deg formula} vanish.

\begin{prop}
    \label{prop: deg formula vanishing}
    We have $\phi'_* \Delta^!\left({\textstyle \prod_v} C_v \cap [\Mbar_v]^{\vir}\right) = 0$ unless the following conditions hold for all $Q$-vertices $v$:
    \begin{enumerate}
        \item\label{item: betav} the curve class $\beta_v$ decorating $v$ is the multiple of a fibre class of $P\rightarrow D_0$;
        
        \item\label{item: gv} $g_v=0$;

        \item\label{item: confined markings} $v$ carries a single contact marking (an edge of $\Gamma$) and no interior markings.
    \end{enumerate}
\end{prop}

\begin{proof}[Sketch of proof]
    The \namecref{prop: deg formula vanishing} follows by a slight modification of the arguments used to prove \cite[Prop.~4.3]{BS24:refGW}:
    \begin{enumerate}
        \item Suppose $v$ is a $Q$-vertex whose curve class $\beta_v$ has a non-trivial push-forward along $\pi:P\rightarrow D_0$, that is, it is not a fibre class. Then \smash{$[\Mbar_v]^{\vir}$} vanishes since this class can be pulled back from the virtual fundamental class of the moduli space of stable maps to $D$ in curve class $\pi_* \beta_v \neq 0$. The latter virtual fundamental class vanishes by cosection localisation and the fact that $D \cong \Tot{}_{\bP^1} \, \cO(-2)$ admits a holomorphic symplectic form (cf. \cite[Sec.~4.4.1]{BS24:refGW}).

        \item Again, let $v$ be a $Q$-vertex. Since now we may assume that $\beta_v$ is the multiple of a fibre, the virtual fundamental class $[\Mbar_v]^{\vir}$ may be written as an intersection $e(\bE^\vee \boxtimes T_{D}) \cap [\Mbar_v]^{\vir_\pi}$ for some modified class $[\Mbar_v]^{\vir_\pi}$ where $\bE$ is the Hodge bundle (cf. \cite[Eq.~(3)]{MP06:topViewGW}). Then writing $\mr{pt}$ for the class of a point in $\bP^1\subset D$ by Mumford's relation for the Chern classes of the Hodge bundle \cite[Sec. 5]{Mu83} we find that
        \begin{equation*}
            e(\bE^\vee \boxtimes T_{D}) = \HodgeLambda[g]{2\,\mr{pt}} \, \HodgeLambda[g]{-2\,\mr{pt}} = (-4\,\mr{pt}^2)^{g_v}
        \end{equation*}
        which is zero if $g_v>0$.

        \item The final confinement of markings can be argued verbatim to \cite[Sec.~4.4.2]{BS24:refGW}. \qedhere
    \end{enumerate}
\end{proof}

\subsubsection{Evaluating the remaining contributions}
For a graph $\Gamma$ to contribute non-trivially to the degeneration formula, all of its $Q$-vertices must be adjacent to precisely one edge by \Cref{prop: deg formula vanishing} \labelcref{item: confined markings}. This means $\Gamma$ must be of star shape with a unique $\CYth$-vertex $\tilde{v}$ connected to $Q$-vertices $v_1,\ldots,v_m$ by $m$ edges. By \Cref{prop: deg formula vanishing} \labelcref{item: betav}--\labelcref{item: confined markings}, $\tilde{v}$ must carry all $n$ interior markings and be decorated with $g_{\tilde{v}}=g$ and $\beta_v = \beta$. The remaining $Q$-vertices are of genus $g_{v_i}=0$ and their curve class $\beta_{v_i}$ must be some $d_{i}$th multiple of a fibre-class for some $d_i\in\bZ_{>0}$. Now note that the last requirement on the curve classes implies that under $\phi'$ each domain of a stable log map in $\Mbar_{v_i}$ gets contracted as $\phi'$ composes with the projection $\cS_0\rightarrow S$ and stabilises the domain. This means we can first push forward all cycles of $Q$-vertices along the evaluation map $\mr{ev}_i:\Mbar_{v_i} \rightarrow \bP^1 \subset D$ when gluing along the diagonal $\Delta$ in the degeneration formula \eqref{eq: deg formula}. Regarding the last push-forward, a dimension count reveals that we must have (cf.~\cite[Sec.~4.5]{BS24:refGW})
\begin{equation*}
    \mr{ev}_i{}_* \big(C_{v_i} \cap [\Mbar_{v_i}]^{\vir}\big) = \epsilon^{-1} a_{i} [\bP^1]
\end{equation*}
for some $a_i\in\bQ$. The constant $a_i$ can then be determined exactly as in \cite[Sec.~4.5]{BS24:refGW} by confining the image of a stable log map to a fibre over a point in $\bP^1\subset D$. The result of this calculation is $a_i = (-1)^{d_i-1}/d_i^2$. Plugging this into the degeneration formula \eqref{eq: deg formula} then yields
\begin{equation*}
    [\Mbar_{g,n}(Z,\beta)]^{\vir}_{\Gm} = \sum_{m>0}  \epsilon^{2g-2+m} \sum_{\substack{d_1,\ldots,d_m >0 \\ \sum_i d_i = D_2\cdot \beta}} \frac{(-1)^{D_2\cdot \beta+g-1}}{m! \prod_i d_i} \big( \phi_* [\Mbar_{\tilde{v}}]^{\vir} \big)\times [\Aone]_{\Gm}
\end{equation*}
and thus \Cref{prop: curious CY4 vanishing virt classes} follows from identifying $\Mbar_{\tilde{v}}=\Mbar_{g,n,\boldsymbol{d}}(X \mid D, \beta)$.

\begingroup
\setlength{\emergencystretch}{.5em}
\renewcommand*{\bibfont}{\footnotesize}
\printbibliography\medskip
\endgroup

\end{document}